\documentclass[1 [leqno,11pt]{amsart} 
\author{Stefano Scrobogna}
\title{Derivation of limit equation for a singular perturbation of a 3D periodic Boussinesq system.}

\usepackage{times}
\usepackage[T1]{fontenc}
\usepackage{amssymb}
\usepackage[english]{babel}
\usepackage{amssymb,fullpage,amsthm}
\usepackage{mathtools}

\newcommand{\RN}[1]{%
  \textup{\uppercase\expandafter{\romannumeral#1}}%
}
\DeclareMathAlphabet{\mathcal}{OMS}{cmsy}{m}{n}
\usepackage{hyperref}
\usepackage{cite}
\renewcommand{\d}{\textnormal{d}}
\newcommand{\2}{{L^2\left(\mathbb{T}^3\right)}}

\newcommand{\osc}{\textnormal{osc}}

\newcommand{\RF}{\textnormal{RF}}

\newcommand{\p}{{L^p\left(\mathbb{T}^3\right)}}

\newcommand{\Linfty}{L^\infty\left(\mathbb{T}^3\right)}
\newcommand{\dx}{\textnormal{d}{x}}
\newcommand{\dive}{\textnormal{div}\;}
\newcommand{\fine}{\hfill$\blacklozenge$}
\newcommand{\QG}{\text{QG}}
\newcommand{\nhp}{\nabla_h^\perp}
\newcommand{\intT}{\int_{\mathbb{T}^3}}
\newcommand{\nh}{\nabla_h}
\newcommand{\Dh}{\Delta_h}
\newcommand{\ck}{\check{k}}
\newcommand{\cm}{\check{m}}
\newcommand{\cn}{\check{n}}
\newcommand{\Lplus}{\mathcal{L}\left(\frac{t}{\varepsilon}\right)}
\newcommand{\Lminus}{\mathcal{L}\left(-\frac{t}{\varepsilon}\right)}
\newcommand{\Hs}{{{H}^s\left( \T^3 \right)}}
\newcommand{\R}{\mathbb{R}}

\newcommand{\uh}{\bar{u}^h}
\newcommand{\thq}{{\triangle}^h_q}

\newcommand{\sumf}{\sum_{\left| q-q' \right|\leqslant4}}
\newcommand{\sumi}{\sum_{q'>q-4}}
\def\L2h{\mathop{{L^2\left( \T^2_h \right)}}}
\newcommand{\tq}{\triangle_q}
\newcommand{\Sq}{S_q}
\newcommand{\tQ}{\triangle_{q'-1}}
\newcommand{\SQ}{S_{q'-1}}
\newcommand{\tv}{\triangle_q}
\newcommand{\Tv}{\triangle_{q'}}

\newcommand{\Svq}{S_{q'-1}}
\newcommand{\set}[1]{\left\lbrace #1 \right\rbrace}
\newcommand{\oh}{\omega^h}
\newcommand{\curlh}{\textnormal{curl}_h}
\newcommand{\pare}[1]{\left( #1 \right)}

\newcommand{\Fr}{\textnormal{Fr}}

\newcommand{\PA}{\mathbb{P}\mathcal{A}}

\newcommand{\hra}{\hookrightarrow}

\newcommand{\loc}{\textnormal{loc}}
\newcommand{\T}{\mathbb{T}}

\newcommand{\pN}{\psi^\varepsilon_N}
\newcommand{\tR}{\tilde{\mathcal{R}}^\varepsilon_{\osc,N}}

\newcommand{\qg}{{quasi-geostrophic }}
\newcommand{\NS}{Navier-Stokes}

\theoremstyle{theorem}
\newtheorem{theorem}{Theorem}[section]
\newtheorem*{theorem*}{Theorem}
\newtheorem{prop}[theorem]{Proposition}
\newtheorem{lemma}[theorem]{Lemma}

\theoremstyle{definition}

\newtheorem{rem}[theorem]{Remark}
\numberwithin{equation}{section}

\makeatletter
\newcommand{\settheoremtag}[1]{
  \let\oldthetheorem\thetheorem
  \renewcommand{\thetheorem}{#1}
  \g@addto@macro\endtheorem{
    \addtocounter{theorem}{-1}
    \global\let\thetheorem\oldthetheorem}
  }
\makeatother

\begin{document}

\date{\today}
\address{\noindent \textsc{S. Scrobogna, IMB, Universit\'e de Bordeaux, 351, cours de la Lib\'eration, Bureau 255, 33405 Talence , France}}
\email{stefano.scrobogna@math.u-bordeaux1.fr}
 

\maketitle
\begin{abstract}
We consider a system describing the long-time dynamics of an hydrodynamical, density-dependent flow under the effects of gravitational forces. We prove that if the Froude number is sufficiently small such system is globally well posed with respect to a  $ H^s, \ s>1/2 $ Sobolev regularity. Moreover if the Froude number converges to zero we prove that the solutions of the aforementioned system converge (strongly) to a stratified three-dimensional Navier-Stokes system. No smallness assumption is assumed on the initial data.
\end{abstract}

\section{Introduction.}\label{sec:introduction}

In the present article we study the behavior of strong solutions of the following modified Boussinesq system
\begin{equation}\label{perturbed BSSQ} \tag{PBS$_\varepsilon$}
\left\lbrace
\begin{aligned}
&\partial_t v^\varepsilon  + v^\varepsilon \cdot\nabla v^\varepsilon  -\nu \Delta  v^\varepsilon -\displaystyle\frac{1}{\varepsilon} \theta^\varepsilon \overrightarrow{e}_3 &=& -\displaystyle \frac{1}{\varepsilon} \nabla \Phi^\varepsilon,\\
&\partial_t \theta^\varepsilon + v^\varepsilon \cdot\nabla \theta^\varepsilon  -\nu' \Delta \theta^\varepsilon + \displaystyle\frac{1}{\varepsilon} v^{3,\varepsilon} & =& \;0,\\
&\dive v^\varepsilon =\;0,\\
& \left. \left( v^\varepsilon, \theta^\varepsilon \right)\right|_{t=0}= \left( v_0, \theta_0 \right),
\end{aligned}
\right.
\end{equation}
for data which are periodic-in-space in the regime $ \varepsilon \to 0 $. The space variable $ x $ shall be many times considered separately with respect to the horizontal and vertical components, i.e. $ x= \left( x_h, x_3 \right)= \left( x_1, x_2, x_3 \right) $. We denote $ \Delta= \partial_1^2 + \partial_2^2 + \partial_3^2 $ the standard laplacian,  $ \Dh= \partial_1^2+\partial_2^2 $ is the laplacian in the horizontal directions.
The symbol $ \nabla $ represents the gradient in all space directions $ \nabla= \left( \partial_1, \partial_2, \partial_3 \right)^\intercal $, while we denote  $ \nh=\left( \partial_1,\partial_2 \right)^\intercal , \nhp=\left( -\partial_2,\partial_1 \right)^\intercal $ respectively the horizontal gradient and the "orthogonal" horizontal gradient. Considered a vector field $ w $ we denote $ \dive w= \partial_1 w^1+ \partial_2w^2+\partial_3 w^3 $.
Given two three-components vector fields $ w, z $ the notation $ w\cdot\nabla z  $ indicates the operator
$$
w\cdot \nabla z = \sum_{i=1}^3 w^i\partial_i z.
$$
 Generally for any two-components vector field $u=\left( u^1,u^2 \right)$ we shall denote as $ u^\perp = \left( -u^2,u^1 \right) $. The viscosity $ \nu, \nu' $ above  are strictly positive constants $ \nu, \nu'\geqslant c >0 $.

As we mentioned already the goal of the present paper is to study the behavior  of (strong) solutions of \eqref{perturbed BSSQ} in the regime $ \varepsilon\to 0 $ for periodic-in-space data, i.e. given $ a_i>0, i=1,2,3 $ we consider the domain
$$
\T^3= \prod_{i=1}^3 \left[0, 2\pi\; a_i\right],
$$
and we look for a divergence-free vector field $ v^\varepsilon : \R_+\times \T^3 \to \R^3 $ and a scalar function $ \theta^\varepsilon : \R_+\times \T^3 \to \R  $ such that $ \left( v^\varepsilon, \theta^\varepsilon \right) $ solves \eqref{perturbed BSSQ}. The functions $\left( v^\varepsilon, \theta^\varepsilon\right)$ depend on $\left( t,x \right)\in \mathbb{R}_+\times \T^3$.  The system \eqref{perturbed BSSQ} belongs to a much wider family of problems which may be written in the following general form
\begin{equation}\label{singular problem generic form}
\left\lbrace
\begin{aligned}
&{\partial_t V^\varepsilon} + v^\varepsilon\cdot \nabla V^\varepsilon + A_2\left( D \right) V^\varepsilon + \frac{1}{\varepsilon}\mathcal{S}\left( V^\varepsilon \right)= 0,\\
&V(0,x)=V_0(x),
\end{aligned}\right.
\end{equation}
where $A_2$ is an elliptic operator and $\mathcal{S}$ is skew-symmetric.\\
 The problem of systems with skew symmetric singular perturbation is not at all new in the literature. S. Klainerman and A. Majda in \cite{kleinerman_majda_singular_limit} develop a first generic theory whose aim is to study a number of problems arising in physics, when certain physical magnitudes blow-up, which can be described by the aid of quasilinear symmetric hyperbolic systems.\\ 
 
 Another system which falls in the family of singular perturbations problems is the \NS\ -Coriolis equations 
 \begin{equation}\label{rotating fluid}
 \tag{RF$_\varepsilon$}
 \left\lbrace
 \begin{aligned}
 &\partial_t v^\varepsilon_\RF + v^\varepsilon_\RF \cdot \nabla v_\RF^\varepsilon -\nu \Delta v_\RF^\varepsilon +\frac{1}{\varepsilon} e^3\wedge v_\RF^\varepsilon = -\frac{1}{\varepsilon} \nabla p_\RF^\varepsilon,\\
 & \dive v^\varepsilon_{\RF}=0,\\
& v_\RF^\varepsilon\left( 0,x \right)= v_{\RF,0}^\varepsilon\left( x \right).
 \end{aligned}
 \right.
 \end{equation}
 E. Grenier in \cite{grenierrotatingeriodic} proved that, as long as the initial data is a bidimensional flow (case which we refer as \textit{well prepared initial data}) the solutions of \eqref{rotating fluid} in a periodic setting converge strongly, after a suitable renormalization, to those of a two-dimensional \NS\ system.
 
A. Babin A. Mahalov and B. Nicolaenko studied at first the equation \eqref{rotating fluid} in the periodic setting in a series of work  ( such as \cite{BMN1}, \cite{BMNresonantdomains} , this list is non exhaustive) when the initial data is considered to be generic or \textit{ill-prepared}, in the sense that it is not a bidimensional flow. Purely three-dimensional perturbations hence can interact constructively between each other, such as in standard \NS\ equations. This problem is overcome in a twofold way: at first in \cite{BMN1} a geometric hypotheses on the domain is done so that no bilinear interaction can occur. Such domains are said to be \textit{non-resonant}, we shall adopt this kind of approach in the present work. In \cite{BMNresonantdomains} instead the domain is considered to be generic, but the authors manage to prove that three-dimensional bilinear interactions are localized in a very specific way in the frequency space. This observation allow hence to deduce an improved product rule which hence can be used to prove that the limit system, despite being three-dimensional and nonlinear, is well posed. This is the key observation which allows them to prove a result of strong, global convergence to a two-dimensional \NS\ system after renormalization.

Finally in \cite{gallagher_schochet} I. Gallagher studied systems in an even more generic form than \eqref{singular problem generic form}, giving a generic theory for the convergence of parabolic systems with singular perturbation in periodic domains. This allowed her to obtain some global strong convergence results for rapidly rotating fluids \eqref{rotating fluid} and for a system describing density dependent fluids under the effects of rotation and gravitational stratification called the  \textit{ primitive equations} (see  \eqref{primitive equations}).
The convergence theory developed by the author is based on a theory developed by S. Schochet in \cite{schochet} in the setting of quasilinear hyperbolic systems and hence adapted to the parabolic case. Such technique consists in determining a "smart" change of variable, which cancels interactions which converge to zero only in a distributional sense. The resulting new unknown is an $\mathcal{O}\pare{ \varepsilon }$ perturbation of the original unknown, but the equation satisfied by the new variable has a simpler spectrum of nonlinear interactions, making hence possible to prove that this new unknown is globally well posed and deducing the result for the initial functions. This technique shall be adopted in the present work as well. 
 \\

As mentioned many times already we are interested in the dynamics of the system \eqref{perturbed BSSQ} in the limit as $\varepsilon\to 0$ in the periodic case. Recently K. Widmayer in \cite{Widmayer_Boussinesq_perturbation} proved that, in the whole space and for the inviscid case, the limit system solves a two dimensional incompressible stratified Euler  equation
\begin{equation}\tag{E-2D}
\left\lbrace
\begin{aligned}
& \partial_t \bar{u}^h_{\text{E}} + \bar{u}^h_{\text{E}} \cdot \nh \bar{u}^h_{\text{E}} +\nh \bar{p}_{\text{E}}=0,\\
&\dive_h \bar{u}^h_{\text{E}} =0.
\end{aligned}
\right.
\end{equation}
His proof relied on the fact that, in the whole space $ \R^3 $ the highly perturbative part of the solution decay at infinity as $ \varepsilon \to 0 $. \\
Let us rewrite the system \eqref{perturbed BSSQ} into the following more compact form
\begin{equation}\tag{\ref{perturbed BSSQ}}
\left\lbrace
\begin{aligned}
& {\partial_t V^\varepsilon}+ v^\varepsilon \cdot \nabla V^\varepsilon - \mathbb{D}V^\varepsilon + \frac{1}{\varepsilon}\mathcal{A}V^\varepsilon = -\frac{1}{\varepsilon} \left( \begin{array}{c}
\nabla \Phi^\varepsilon\\
0
\end{array} \right),\\
& V^\varepsilon=\left( v^\varepsilon,\theta^\varepsilon \right),\\
&\dive v^\varepsilon=0,
\end{aligned}
\right.
\end{equation} where
\begin{align}\label{matrici}
\mathcal{A}= & \left( \begin{array}{cccc}
0&0&0&0\\
0&0&0&0\\
0&0&0&1\\
0&0&-1&0\\
\end{array} \right),
&
\mathbb{D} =& \left( \begin{array}{cccc}
\nu\Delta&0&0&0\\
0&\nu\Delta&0&0\\
0&0&\nu\Delta&0\\
0&0&0&\nu'\Delta\\
\end{array} \right).
\end{align}

\subsection{A survey on the notation adopted.}\label{sec:notation_and_results}

All along this note we consider real valued vector fields, i.e. applications $ V:\R_+\times \T^3 \to  \R^4 $. We will often associate to a vector field $V$ the vector field $v$ which  shall be simply the projection on the first three components of $V$. The vector fields considered are periodic in all their directions and they have zero global average $ \int_{\T^3} V \dx=0 $, which is equivalent to  assume that the first Fourier coefficient $ \hat{V} \left( 0 \right)=0 $. We remark that the zero average propriety stated above is preserved in time $t$ for both \NS\ equations as well as for the system \eqref{perturbed BSSQ}.\\
Let us define the Sobolev space $ \Hs $, which consists in all the tempered distributions $ u $ such that
\begin{equation}
\label{eq:non-hmogeneous_Sobolev_norms}
\left\| u \right\|_\Hs= \left( \sum_{n\in\mathbb{Z}^3}\left(1+ \left| n \right|^{2} \right)^{s}\left| \hat{u}_n \right|^2 \right)^{1/2}<\infty.
\end{equation}
Since we shall consider always vector fields whose average is null the Sobolev norm defined above in particular is equivalent to the following semi-norm
$$
\left\| \left( -\Delta \right)^{s/2} u \right\|_\2 \sim \left\| u \right\|_\Hs, \hspace{1cm}s\in\R,
$$
which appears naturally in parabolic problems.\\
Let us define the operator $\mathbb{P}$ as the three dimensional Leray operator $\mathbb{P}^{(3)}$ wich leaves untouched the fourth component, i.e.
$$
\mathbb{P}= \left(
\begin{array}{c|c}
\delta_{i,j}-{\Delta^{-1}}{\partial_i\partial_j}& 0\\  \hline
0 & 1
\end{array}\right)_{i,j=1,2,3}= \left( \begin{array}{c|c}
\mathbb{P}^{(3)} & 0 \\ \hline
0 & 1
\end{array} \right).
$$ 
The operator $ \mathbb{P} $ is  a pseudo-differential operator, in the Fourier space its symbol is
\begin{equation}
\label{Leray projector}
\mathbb{P}_n= \left(
\begin{array}{c|c}
\delta_{i,j}-\dfrac{\check{n}_i\; \check{n}_j}{ \left| \check{n} \right|^2}& 0\\[4mm]  \hline 
0 & 1
\end{array}\right)_{i,j=1,2,3},
\end{equation}
where $ \delta_{i,j} $ is Kronecker's delta and $ \check{n}_i=n_i/a_i, \left| \check{n} \right|^2= \sum_i \check{n}_i^2$.\\

\subsection{Anisotropic spaces.}\label{sec:anisotropic_spaces}

The problem presents a singular perturbation $ \mathcal{A} $ which does not acts symmetrically on the two-dimensional unit-sphere $ \mathbb{S}^2 $, namely there is a relevant external force acting along the vertical direction. This asymmetry in the balance of forces induces the solutions of \eqref{perturbed BSSQ} to behave differently along the horizontal and vertical directions. For this reason 
 we are forced to introduce  anisotropic spaces, which means spaces which behaves differently in the horizontal or vertical direction. Let us recall that, in the periodic case, the non-homogeneous Sobolev anisotropic spaces are defined by the norm
$$
\left\| u\right\|_{H^{s,s'}\left(\mathbb{T}^3\right)}^2= \sum_{n=\left(n_h,n_3\right)\in\mathbb{Z}^3} \left(1 + \left|\check{n_h}\right|^2\right)^s \left(1 + \left|\check{n_3}\right|^2\right)^{s'} \left| \hat{u}_n \right|^2 ,
$$
where we denoted $\check{n}_i=n_i/a_i, n_h=\left(n_1,n_2\right)$ and the Fourier coefficients $\hat{u}_n$ are given by $u=\sum_n \hat{u}_n e^{2\pi i\check{n}\cdot x}$. In the whole text $\mathcal{F}$ denotes the Fourier transform. In particular our notation will be
$$
\mathcal{F} u (n)= \hat{u}(n)=\hat{u}_n=\intT u(x)e^{2\pi i \check{n}\cdot x}\dx .
$$

Let's recall as well the definition of the anisotropic Lebesgue spaces, we denote with $L^p_hL^q_v$ the space $L^p\left( \mathbb{T}^2_h; L^q\left(\mathbb{T}^1_v\right)\right)$, defined by the norm:
$$
\left\| f\right\|_{L^p_hL^q_v}=\left\| \left\| f\left( x_h,\cdot\right) \right\|_{L^q\left(\mathbb{T}^1_v\right)} \right\|_{L^p\left( \mathbb{T}^2_h\right)}=
\left( \int_{\mathbb{T}^2_h}\left( \int_{\mathbb{T}^1_v} \left| f\left( x_h,x_3\right)\right|^q \d x_3\right) ^\frac{p}{q} \d x_h \right)^\frac{1}{p},
$$
in a similar way we demote the space $L^q_vL^p_h$. It is well-known that the order of integration is important as it is described in the following lemma
\begin{lemma} 
Let $1\leqslant p \leqslant q$ and $f: X_1\times X_2 \to \mathbb{R}$ a function belonging to $L^p\left( X_1; L^q\left( X_2\right) \right)$ where $\left( X_1; \mu_1\right),\left(X_2;\mu_2\right)$ are measurable spaces, then $f\in L^q\left( X_2; L^p\left( X_1 \right)\right)$ and we have the inequality
$$
\left\|f\right\|_{L^q\left( X_2; L^p\left( X_1 \right)\right)}\leqslant \left\|f\right\|_{L^p\left( X_1; L^q\left( X_2\right) \right)}.
$$
\end{lemma}
In the anisotropic setting the Cauchy-Schwarz inequality becomes;
$$
\left\| fg\right\|_{L^p_hL^q_v} \leqslant \left\|f\right\|_{L^{p'}_hL^{q'}_v}\left\| g\right\|_{L^{p''}_hL^{q''}_v},
$$
where $1/p=1/p'+1/p'',1/q=1/q'+1/q''$.\\

We shall need as well to define spaces which are of mixed Lebesgue-Sobolev type of the form
$$
L^p_v \left( H^\sigma_h \right)= L^p \left( \T^1_v; H^\sigma \left( \T^2_h \right) \right), \; p\in \left[ 1, \infty\right].
$$

\subsection{Results.}\label{sec:results}

Theorem \ref{main result} is the main result proved in the present work. Unfortunately in order to understand in detail the statement of such theorem some notational explanation (notably Section \ref{sec:notation_and_results}) was introduced. The first part of the present section instead focuses  in introducing some result which is classical in the theory of \NS\ equations and which is of the utmost importance in order to develop the theory in the present work.\\

We recall at first the celebrated Leray and Fujita-Kato theorems. The first is a result of existence of distributional solutions for \NS\ equations,  while the second is  a result of (local) well-posedness in Sobolev spaces for \NS\ equations. The proof of such results is considered to be nowadays somehow classical and can be found in many texts, we refer to \cite{FK2} and \cite{FK1} or \cite{monographrotating}.

\begin{theorem*} 
[\textsc{Leray}]\label{thm:Leray}
Let us consider the following system describing the evolution of an incompressible viscid fluid in the $ d $-dimensional periodic space $ \T^d $
\begin{equation}\tag{NS}\label{NS}
\left\lbrace
\begin{array}{l}
\partial_t u + \left( u\cdot \nabla \right)u -\nu \Delta u = -\nabla p,\\
\nabla\cdot u=0,\\
u\left( 0 \right)=u_0.
\end{array}
\right.
\end{equation}
Let $ u_0 $ be a divergence-free vector field in $ L^2 \left(\T^d \right) $, then \eqref{NS} has a weak solution $ u $ such that
\begin{equation*}
u \in L^\infty \left( \R_+; L^2 \left( \T^d \right) \right), \hspace{1cm} \nabla u \in L^2 \left( \R_+; L^2 \left( \T^d \right) \right).
\end{equation*}
\end{theorem*}

\begin{theorem*}[\textsc{Fujita-Kato}]\label{WPdDNS}
Let $ u_0\in H^{\frac{d}{2}-1}\left( \T^d \right) $, then there exists a positive time $ T^\star $ such that \eqref{NS} has a unique solution $ u\in L^4\left( \left[0,T^\star\right]; H^{\frac{d-1}{2}}\left( \T^d \right) \right) $ which also belongs to 
$$
\mathcal{C}\left( \left[0,T^\star\right]; H^{\frac{d}{2}-1}\left( \T^d \right)  \right)\cap L^2\left( \left[0,T^\star\right]; H^{\frac{d}{2}}\left( \T^d \right)  \right).$$
Let us denote $ T^\star_{u_0} $ be the maximal lifespan of the solution of \eqref{NS} with initial datum $ u_0 $, then there exists a constant $ c>0 $ such that if 
$$
\left\| u_0 \right\|_{H^{\frac{d}{2}-1}}\leqslant c \nu\Longrightarrow T^\star_{u_0}=\infty.
$$
\end{theorem*}

Since the perturbation appearing in \eqref{perturbed BSSQ} is skew symmetric we know that the bulk force $ \mathcal{A} V^\varepsilon $ does not apport any energy in any $ \Hs $ space, whence Leray and Fujita-Kato theorem can be applied \textit{mutatis mutandis} to the system \eqref{perturbed BSSQ}, and in particular  this is the formulation which we shall use:

\begin{theorem}\label{thm:Leray_adapted}
Let $ V_0 = \left( v_0, \theta_0 \right)\in \2 $ and such that $ \dive v_0 =0 $. Then for each $ \varepsilon >0 $ there exists a weak solution $ V^\varepsilon $ of \eqref{perturbed BSSQ} which belongs to the energy space
\begin{equation*}
V^\varepsilon \in L^\infty \left( \R_+; \2 \right), \hspace{1cm} \nabla V^\varepsilon \in L^2 \left( \R_+; \2 \right).
\end{equation*}
\end{theorem}

\begin{theorem}\label{thm local existence strong solution}
If $V_0\in H^s\left( \T^3 \right)$ with $s > 1/2$ there exists a positive time $T^\star$ independent of $\varepsilon >0$ and a unique strong solution $V^\varepsilon$ of \eqref{singular problem generic form} in the space $ L^4 \left( \left[0,T^\star\right]; H^{s+\frac{1}{2}}\left( \T^3 \right) \right) $ which also belongs to the space  $\mathcal{C}\left( \left[0, T^\star\right]; H^s\left( \T^3 \right) \right)\cap L^2\left( \left[0, T^\star\right]; H^{s+1}\left( \T^3 \right) \right)$. In particular if $ \left\| V^\varepsilon_0 \right\|_{H^{s+\frac{1}{2}}\left( \T^3 \right)}\leqslant c \nu $ for some positive and small constant $ c $ then the solution is global in $ \R_+ $.
\end{theorem}

In the framework of $ d $-dimensional \NS\ equations the propagation of $ H^{\frac{d}{2}-1} $ Sobolev regularity  is usually referred as propagation of  \textit{critical regularity}. 
It is hence a generally accepted choice of lexicon to denote the regularity $ H^s, \ s>d/2 -1 $ as \textit{subcritical} and $ H^s, \ s<d/2 -1 $ as \textit{supercritical}. \\

The dynamics of \eqref{perturbed BSSQ} varies accordingly to the real parameter $ \varepsilon $. The asymptotic regime $ \varepsilon\to 0 $ is of particular interest since it describes long-time dynamics of stratified fluids (for a more detailed physical discussion we refer to Section \ref{sec:physics_of_the_system}), it is hence relevant to prove that \eqref{perturbed BSSQ} admits a limit when $\varepsilon \to 0$. The limit system may be written as follows:
\begin{equation}\label{limit system} \tag{$\mathcal{S}_0$}
\left\lbrace
\begin{aligned}
& \partial_t U + \mathcal{Q}\left( U, U \right) - \mathbb{D}U =0,\\
&\left. U \right|_{t=0}=V_0.
\end{aligned}
\right.
\end{equation}

The sense in which system \eqref{perturbed BSSQ} converges to \eqref{limit system} shall be explained in detail in Section \ref{linear problem} and \ref{sec:filtered limit}. Section \ref{sec:filtered limit} is entirely devoted to explain in detail in what consists the limit form $\mathcal{Q}$ and $\mathbb{D}$.\\
As it is proven in \cite{gallagher_schochet} any system in the generic form \eqref{singular problem generic form} converges to a limit system of the form \eqref{limit system} in the sense of distributions. In the Section \ref{sec:convergence} we extend this convergence to a strong setting. Such technique has been introduced by S. Schochet in \cite{schochet} in the framework of quasilinear symmetric hyperbolic systems, but the theory in the parabolic setting was developed by I. Gallagher in \cite[Theorem 1]{gallagher_schochet}. The statement of \cite[Theorem 1]{gallagher_schochet} is the following:

\begin{theorem*}[\textsc{Gallagher}]
\label{conv theorem with method schochet}
Let $U_0 \in H^s\left( \T^d \right)$ with $s\geqslant \frac{d}{2}-1$. Let $T<T^\star$, and $U$ be the local, strong solution of \eqref{limit system} determined by Theorem \ref{WPdDNS} satisfy
$$
U\in \mathcal{C}\left( \left[ 0, T \right]; H^s\left( \T^d \right) \right) \cap L^2\left( \left[ 0 , T \right]; H^{s+1}\left( \T^d \right) \right)
$$ then, for $\varepsilon> 0$ small enough the associate solution $V^\varepsilon$ of \eqref{singular problem generic form} is also defined on $ \left[ 0, T \right] $ and 
$$
V^\varepsilon-\Lplus U=o\left( 1 \right),
$$
in $ \mathcal{C}\left( \left[ 0, T \right]; H^s\left( \T^d \right) \right) \cap L^2\left( \left[ 0 , T \right]; H^{s+1}\left( \T^d \right) \right) $.
\end{theorem*}

The operator $ \mathcal{L}\left( \tau \right) $ appearing in the above theorem is nothing but the backward propagator $ e^{\tau \PA} $, we refer to Section \ref{linear problem} and \ref{sec:filtered limit} for a more detailed introduction.

The result we prove has rather long and technical statement, but it simply addresses a stability result of \eqref{perturbed BSSQ} to a simplified 3-dimensional nonlinear model, and it is divided in four parts:
\begin{enumerate}
\item \label{l1} as $ \varepsilon \to 0 $ the system \eqref{perturbed BSSQ} converges, in the sense of distributions, to a limit system,

\item \label{l2} the limit system can be simplified, in particular it can be written as the sum of two systems. The first one is similar to a 2D-\NS\ system, the second is a linear system,

\item \label{l3} the aforementioned systems are, individually, globally well posed. Hence the limit system is globally well-posed,

\item \label{l4} \eqref{perturbed BSSQ} converges (now strongly) to the limit system which now we know to be globally well-posed. We deduce the convergence to be global.  
\end{enumerate} 
We would like to spend a couple of words more on the result \eqref{l4} of the list here above. The convergence procedure gives an additional result which is crucial: we proved in the point \eqref{l3} that the limit system solved by $ U $ is globally well posed in some Sobolev space: $ V^\varepsilon $ solution of \eqref{perturbed BSSQ} converges \textit{globally} to a renormalization of $ U $, hence $ V^\varepsilon $ is globally well posed as well if $ \varepsilon $ is sufficiently small.\\
The first result we prove is the following compactness result concerning the solutions \textit{\`a la Leray} of the system \eqref{perturbed BSSQ}:

\begin{theorem}\label{thm:topological_convergence}
Let $\mathcal{L} \pare{ \tau } = e^{\tau \PA}$ where $\mathcal{A}$ and $\mathbb{P}$ are defined respectively in \eqref{matrici} and \eqref{Leray projector},  and let $V_0 \in \2$ with zero horizontal average and such that $\dive v_0=0$. The sequence  $\pare{\Lplus V^\varepsilon}_{\varepsilon >0}$ whith $V^\varepsilon$ energy solution determined in Theorem \ref{thm:Leray_adapted} is weakly compact in the $L^2_{\loc}\pare{ \R_+\times \T^3 }$ topology and each element $U$ of the topological closure of $\pare{\Lplus V^\varepsilon}_{\varepsilon >0}$ w.r.t. the $L^2_{\loc}\pare{ \R_+\times \T^3 }$ topology solves \eqref{limit system} in the sense of distributions and belongs to the energy space
\begin{equation*}
L^\infty \pare{ \R_+; \2 }\cap L^2 \pare{ \R_+; H^1 \pare{ \T^3 } }.
\end{equation*}
\end{theorem}

The second result we prove is the following simplification of the limit system in the abstract form \eqref{limit system}:

\begin{theorem}\label{thm:simplification_limit_system}
Let us consider a divergence-free vector field $ V_0 $ with zero horizontal average, let us define
\begin{align*}
\oh_0= & \; \curlh V^h_0,&
\bar{U}_0 = & \left( \nhp \Dh^{-1} \oh_0, 0, 0 \right),&
U_{\osc, 0}= & \; V_0-\bar{U}_0.\\
= & \; -\partial_2 V^1_0  + \partial_1 V^2_0 &
= & \left( \uh_0, 0, 0 \right).
\end{align*}
The projection of \eqref{limit system} onto $\ker \PA$ is the following 2d-\NS\ stratified system with full diffusion
\begin{equation}\label{eq:2DstratifiedNS}
\left\lbrace
\begin{aligned}
& \partial_t \uh \left( t,  x_h, x_3 \right) + \uh\left( t,  x_h, x_3 \right) \cdot \nh \uh \left( t,  x_h, x_3 \right) - \nu \Delta \uh\left( t,  x_h, x_3 \right) = -\nh \bar{p} \left( t,  x_h, x_3 \right)\\
& \dive_h \uh \left( t, x_h, x_3 \right) =0,\\
& \left. \uh\left( t,  x_h, x_3 \right) \right|_{t=0}=\uh_0\left(   x_h, x_3 \right).
\end{aligned}
\right.
\end{equation}
While the projection of \eqref{limit system} onto $\pare{ \ker \PA }^\perp$ satisfies,  for almost all $ \left( a_1, a_2, a_3 \right)\in \R^3 $ parameters of the three-dimensional torus $ \T^3=\prod_i \left[0,2\pi\; a_i\right] $,  the following linear system in the unknown $U_{\osc}$
\begin{equation} \label{limit system perturbed BSSQ}
\left\lbrace
\begin{aligned}
&\partial_t U_\osc +2\mathcal{Q}\left( \bar{U}, U_\osc \right) - \left( \nu+\nu' \right)\Delta U_\osc =0,\\
& \dive U_{\osc}=0,\\
&\left. U_\osc \right|_{t=0}=U_{\osc, 0}= V_0-\bar{U}_0.
\end{aligned}
\right.
\end{equation}
\end{theorem}

Theorem \ref{thm:simplification_limit_system} hinges to a rather important deduction: the limit system in the abstract form \eqref{limit system} is hence the superposition of \eqref{eq:2DstratifiedNS} and \eqref{limit system perturbed BSSQ}. General theory of 2D \NS\ systems and of linear parabolic equations gives hence the tools the prove a global well-posedness result which reads as follows:

\begin{theorem}\label{thm:GWP_limit_system}
Let us consider a vector field $ V_0= \left( v_0, V^4_0 \right)=\left( V^1_0, V^2_0, V^3_0, V^4_0 \right)\in \Hs$, $s>1/2 $.
 Let $ V_0 $ be of global zero average and of horizontal zero average, i.e.
 \begin{align}\label{eq:Horizontal_average_initial_data}
 \int_{\T^3} V_0 \left( y \right)\d y =&\;0, &
 \int_{\T^2_h} V_0 \left( y_h, x_3 \right)\d y_h =& \;0.
 \end{align}

Let us assume $ \uh_0\in L^\infty_v \left( H^\sigma_h \right) $ and $  \nh \uh_0\in L^\infty_v \left( H^\sigma_h \right)  $ with $ \sigma >0 $, then $ \uh $ solution of  \eqref{eq:2DstratifiedNS},  
 is globally well posed in $ \R_+ $, and belongs to the space
$$
\uh \in \mathcal{C}\left( \R_+; \Hs \right)\cap L^2 \left( \R_+; H^{s+1}\left( \R^3 \right) \right), \hspace{1cm} s>1/2,
$$
and for each $ t>0 $ the following estimate holds true
\begin{equation*}
\left\| \uh \left( t \right)\right\|_\Hs^2 + \nu \int_0^t \left\| \nabla\uh\left( \tau \right) \right\|_{H^{s}\left( \T^3 \right)}^2\d\tau
 \leqslant \mathcal{E}_1 \left( U_0 \right).
\end{equation*}
Where the function $ \mathcal{E}_1 $ is defined as the right-hand-side of equation \eqref{eq:stong_Hs_bound_ubar}.\\
Let   $ U_\osc $ be the solution of the  linear system \eqref{limit system perturbed BSSQ}. It
 is globally defined and it belongs to the space 
$$
U_\osc \in \mathcal{C}\left( \R_+; \Hs \right)\cap L^2 \left( \R_+; H^{s+1}\left( \R^3 \right) \right),
$$
for $ s>1/2 $. For each $ t> 0  $ the following bound holds true
$$
\left\| U_\osc \left( t \right) \right\|_{\Hs}^2 + 
\frac{\nu+\nu'}{2} \int_0^t \left\| \nabla U_\osc \left( \tau \right) \right\|^2_{H^{s}\left( \T^3 \right)}\d\tau
\leqslant  \mathcal{E}_2 \left( U_0 \right),
$$
and the function $ \mathcal{E}_2 $ is defined as the right-hand-side of equation \eqref{eq:stong_Hs_bound_uosc}.
\end{theorem}

The last question we address to is the stability of the dynamics of \eqref{perturbed BSSQ} in the limit regime $\varepsilon \to 0$. As mentioned above this is done with a methodology introduced by I. Gallagher in \cite{gallagher_schochet} and already outlined in the introduction:

\begin{theorem}\label{main result}
Let $ V_0 $ be in $ \Hs $ for $ s>1/2$ and of zero horizontal average as in Theorem \ref{thm:GWP_limit_system}. For $ \varepsilon > 0 $ small enough \eqref{perturbed BSSQ} is globally well posed in $ \mathcal{C}\left( \mathbb{R}_+; H^s\left( \T^3 \right) \right) \cap L^2\left( \mathbb{R}_+; H^{s+1}\left( \T^3 \right) \right) $, and, if $U$ is the solution of \eqref{limit system}, then 
$$
V^\varepsilon -\Lminus U =o\left( 1 \right),
$$
in $ \mathcal{C}\left( \mathbb{R}_+; H^s\left( \T^3 \right) \right) \cap L^2\left( \mathbb{R}_+; H^{s+1}\left( \T^3 \right) \right) $.
\end{theorem}

Let us, now, outline the structure of the paper:
\begin{itemize}\label{structure paper}
\item In Section \ref{linear problem} we shall study the linear problem associated to the singular perturbation $ \varepsilon^{-1}\PA $ characterizing the system \eqref{perturbed BSSQ}. By mean of a careful spectral analysis of the penalized operator $ \PA $ we define what shall be called the \textit{non-oscillating} and \textit{oscillating} subspace. The first is the subspace in Fourier variables defined by the divergence-free elements belonging to the kernel of $ \PA $. Being in the kernel of such operator the evolution of such elements shall not be influenced by the highly external force $ \varepsilon^{-1}\PA $ and hence it shall not exhibit any oscillating behavior. On the other hand the element belonging to the oscillating subspace, which is the orthogonal complement of $ \ker \PA $ will present an oscillating behavior which depends (inversely) on the parameter $ \varepsilon $.
\item In Section \ref{sec:filtered limit} we prove Theorem \ref{thm:topological_convergence}. We apply the Poicar\'e semigroup
$$
\mathcal{L}\left( \tau \right)= e^{\tau \PA},
$$
to the system \eqref{perturbed BSSQ}. The new variable $ U^\varepsilon= \Lplus V^\varepsilon $ satisfies an equation which is very close to the three-dimensional periodic \NS\ equation which we denote as \textit{the filtered system}. We avoid to give a detailed description of the filtered system now, but the reader which is already familiar with this kind of mathematical tools is referred to \eqref{filtered system}. What has to be retained is the fact that it is possible to construct from \eqref{perturbed BSSQ} another family of systems, indexed by $\varepsilon$, which is somehow better suited for the study of the problem. Thanks to this similarity we can deduce that the weak solutions $ \left( U^\varepsilon \right)_\varepsilon $ are in fact uniformly bounded in some suitable space, and thanks to standard compactness arguments we deduce that
$$
U^\varepsilon \to U,
$$
weakly. In particular $ U $ satisfies a three-dimensional Navier-Stokes-like equation, whose bilinear interaction (defined in \eqref{eq:def_limit_bilinear_form}) has better product rules than the standard transport-form. Lastly we deduce that $ V^\varepsilon $ can in fact be written as
$$
V^\varepsilon= \text{ stationary state }+\text{ high oscillation }+\text{ remainder}.
$$
\item in Section \ref{the limit} we prove Theorem \ref{thm:simplification_limit_system} via a study of the limit (in the sense of distributions) of the filtered system  as $ \varepsilon \to 0 $. In particular such limit has two qualitatively different behaviors once we consider its projection onto $ \left( \ker \PA \right) $ and $ \left( \ker \PA \right)^\perp $:
\begin{itemize}
\item The projection of the limit system \eqref{limit system} onto $ \left( \ker \PA \right) $ presents, as a bilinear interaction, bilinear interactions of elements of $ \left( \ker \PA \right) $ only, and in particular it is represented by a two-dimensional, stratified, \NS\ equation with additional vertical diffusion.
\item  The projection of the limit system \eqref{limit system} onto $ \left( \ker \PA \right)^\perp $ is, for almost all three-dimensional tori, a linear equation in the unknown $ U_\osc $. Such deduction is a result of a geometrical analysis on the domain, we denote the domains which satisfies such properties as \textit{non-resonant domains}.
\end{itemize}
\item The Section \ref{global_existence} is devoted to the proof of Theorem \ref{thm:GWP_limit_system}. As well as in Section \ref{the limit} we divide the proof in two sub-parts, considering the projection of the solutions onto $ \left( \ker \PA \right) $ and $ \left( \ker \PA \right)^\perp $ respectively
\begin{itemize}
\item The kernel part, as already stated, is a two-dimensional stratified \NS\ equation. We take advantage of the fact that, along the vertical direction, the equation is purely diffusive without transport term. This allows us to prove that in fact, for some suitable anisotropic strong norms, the solution decay exponentially-in-time, and hence the global-in-time result.
\item For the oscillating subspace we exploit the fact that the solution satisfied by $ U_\osc $ is linear to achieve the global result.
\end{itemize}
\item Lastly, in Section \ref{sec:convergence}, we prove Theorem \ref{main result} using a smart change of variable, to prove that
$$
V^\varepsilon - \Lminus U \to 0 \hspace{1cm} \text{in } L^\infty \left( \R_+, \Hs \right) \cap
L^2 \left( \R_+, H^{s+1}\left( \R^3 \right) \right), \; s>1/2.
$$
\end{itemize}

\subsection{Physical motivation of the system \eqref{perturbed BSSQ} and previous works on symilar systems.}\label{sec:physics_of_the_system}

In the present section we linger for a while on the physical motivations which induce us to study the system \eqref{perturbed BSSQ} and we continue to (briefly) expose some relevant result concerning various system related to \eqref{perturbed BSSQ}.\\
In the following $ v= \left( v^1, v^2, v^3 \right) $ represents the velocity flow of the fluid, and $   \Fr $ is a  positive constant which have a physical relevance. It will be defined precisely in the following. \\

  The system describing the motion of a fluid with variable density under the effects of (external) gravitational force is (see \cite[Chapter 11]{cushman2011introduction}):
\begin{equation}\label{eq:primitive1}
\left\lbrace
\begin{aligned}
& \partial_t v^1 + v\cdot \nabla v^1 &=& -  \frac{p_0}{\rho_0}\ \partial_1
\phi +\nu \Delta v^1,\\
& \partial_t v^2 + v\cdot \nabla v^1  &= & -  \frac{p_0}{\rho_0}\ \partial_2
\phi +\nu \Delta v^2,\\
& \partial_t v^3 + v\cdot \nabla v^3 +\frac{1}{\Fr} \theta & = & -  \frac{p_0}{\rho_0}\ \partial_3
\phi +\nu \Delta v^3,\\
& \partial_t \theta + v\cdot \nabla \theta - \frac{1}{\Fr} v^3 & = &\; \nu' \Delta \theta,\\
& \dive v =0.
\end{aligned}
\right.
\end{equation}
The values $ \nu, \nu' $ are modified kinematic viscosities which depend on the nondimentionalization used to deduce \eqref{eq:primitive1}. 
The term $ p_0 $ appearing in \eqref{eq:primitive1} is called the \textit{reference pressure}. 

 We point out some characteristic features which characterize a motion described by a system of the form \eqref{eq:primitive1}:
\begin{itemize}
\item The quantity $ \Fr $ is said the \textit{Froude number} and measures the stratification of a fluid. We shall define it in detail in what follows. It is important that in the third and fourth equation of \eqref{eq:primitive1} the same prefactor $ 1/\Fr $ influences the motion: the application $ \left( v^3, \theta \right)\mapsto \Fr^{-1} \left( \theta, -v^3 \right) $ is hence linear and skew-symmetric, inducing hence a zero-apportion to the global $ H^s $ energy of the motion.

\item The element $ \frac{1}{\Fr}\ \theta $ appearing in the third equation of \eqref{eq:primitive1} is the nondimensionalization of the downward acceleration induced by gravity.

\item The element $ -\frac{1}{\Fr} \ v^3 $ describes the upward acceleration provided by Archimede's principle and caused by the tendency of the fluid to dispose itself in horizontal stacks of decreasing density.

\end{itemize}

The system \eqref{perturbed BSSQ} falls in a wide category of problems known as \textit{singular perturbation problems}: notably a very well-known example of such problem are the \NS-Coriolis equations \eqref{rotating fluid}. \\
 It is relevant to mention that Chemin et al. in \cite{CDGGanisotranddispersion} proved global strong convergence of solutions of \eqref{rotating fluid} with only horizontal viscosity $ -\nu_h \Dh $ instead of the full viscosity $ -\nu\Delta $ to a purely 2D \NS\ system in the case in which $\varepsilon \to 0$ and the space domain is $\R^3$. This result is attained with methodologies which are very different with respect to the ones mentioned for the periodic setting. 

A system describing the motion of a fluid when stratification and rotation have comparable frequencies is the following one, known as \textit{primitive equations}:
\begin{equation}\tag{PE$_\varepsilon$}\label{primitive equations}
\left\lbrace
\begin{aligned}
& 
\begin{multlined}[t]
{\partial_t V_{\mathcal{P}}^\varepsilon} + v_{\mathcal{P}}^\varepsilon\cdot \nabla V_{\mathcal{P}}^\varepsilon - \mathbb{D}V_{\mathcal{P}}^\varepsilon + \frac{1}{\varepsilon}
\left(
\begin{array}{cccc}
-V_{\mathcal{P}}^{2,\varepsilon}, & V_{\mathcal{P}}^{1,\varepsilon}, & \frac{1}{F} V_{\mathcal{P}}^{4,\varepsilon}, & -  \frac{1}{F} V_{\mathcal{P}}^{3,\varepsilon}
 \end{array}
 \right)^\intercal \\
=-\frac{1}{\varepsilon}\left( \begin{array}{cc}
\nabla \Phi_{\mathcal{P}}^\varepsilon, & 0
\end{array} \right)^\intercal,
\end{multlined}
\\
&V_{\mathcal{P}}^\varepsilon\left( 0,x \right)= V^\varepsilon_{{\mathcal{P}},0}\left( x \right),
\end{aligned}
\right.
\end{equation}
where $\mathbb{D}$ is defined in \eqref{matrici}. J.-Y. Chemin studied the system \eqref{primitive equations} in the case $F\equiv 1$ in \cite{chemin_prob_antisym} obtaining a global existence result under a smallness condition made only on a part of the initial data. He proved in fact that, setting $ \Omega^\varepsilon= -\partial_2 v_{\mathcal{P}}^{1,\varepsilon} + \partial_1 v_{\mathcal{P}}^{2,\varepsilon}-F\partial_3 v_{\mathcal{P}}^{4,\varepsilon} $ the system \eqref{primitive equations} converges to what it is known as the \qg system, i.e.
\begin{equation}\tag{QG}
\label{quasi geostrophic}
\partial_t V_\QG- \Gamma V_\QG +\left( \begin{array}{c}
-\partial_2\\ \partial_1 \\0 \\ -F\partial_3
\end{array} \right)\Delta^{-1}\left( V^h_\QG \cdot \nh \Omega^\varepsilon \right) =0,
\end{equation}
where $ V_\QG=\left( V^h_\QG, V^3_\QG, V^4_\QG \right)=\left( -\partial_2 \Delta^{-1}\Omega, \partial_1 \Delta^{-1}\Omega, 0, -F\partial_3\Delta^{-1}\Omega \right) $ and $ \Gamma $ is an elliptic operator of order two defined as 
$
\Gamma = \left( \Dh + F^2\partial_3^2 \right)^{-1}\Delta\left( \nu \Dh + \nu' \partial_3^2 \right).
$ \\

 I. Gallagher proved strong convergence of solutions of \eqref{primitive equations} to a limit system in the form \eqref{limit system} in the periodic case always in the work \cite{gallagher_schochet}, F. Charve proved first weak convergence of solutions of \eqref{primitive equations} to solution of \eqref{quasi geostrophic} in \cite{charve2}, and strong convergence in \cite{charve1}. The case in which the system \eqref{primitive equations} presents only horizontal diffusion, hence it is a mixed parabolic hyperbolic type has been studied by Charve and Ngo in \cite{charve_ngo_primitive} in the whole space in the case $ \nu_h=\nu_h^\varepsilon=\mathcal{O}\left( \varepsilon^\alpha \right),\; \alpha >0 $. 
 
 We mention as well the work of D. Bresch, D. G\'erard-Varet and E. Grenier from \cite{BG-VG06}. In this work the authors consider the primitive equations in the form
 \begin{equation}\tag{$ \widetilde{\text{PE}}_\varepsilon $}\label{primitive equations 2}
 \left\lbrace
 \begin{aligned}
 &\partial_t v_P^\varepsilon + u_P^\varepsilon \cdot \nabla v^\varepsilon_P - \nu \Delta v^\varepsilon_P +\frac{1}{\varepsilon}e^3\wedge v^\varepsilon_P= -\frac{1}{\varepsilon}\nh \phi_P^\varepsilon,\\
& \partial_3 \phi^\varepsilon_P= \theta^\varepsilon_P\\
& \text{div}_h\; v^\varepsilon_P= -\partial_3 w^\varepsilon_P, \\
& \partial_t \theta^\varepsilon_P + u^\varepsilon_P\cdot \nabla \theta^\varepsilon_P	-\nu'\Delta \theta^\varepsilon_P + w^\varepsilon_P	=Q,\\
& u^\varepsilon_P = \left( v^\varepsilon_P,w^\varepsilon_P \right).
 \end{aligned}
 \right.
 \end{equation}
 The methodology used in \cite{BG-VG06} although is completely different with respect to the other works mentioned. The penalization in particular is \textit{not skew-symmetric}, this prevents the authors to apply energy methods as in the other works mentioned.\\

\subsection{Elements of Littlewood-Paley theory.}\label{elements LP}

A tool that will be widely used all along the paper is the theory of Littlewood--Paley, which consists in doing a dyadic cut-off of the  frequencies.\\
Let us define the (non-homogeneous)  truncation operators as follows:
\begin{align*}
\tv u= & \sum_{n\in\mathbb{Z}^3} \hat{u}_n \varphi \left(\frac{\left|\check{n}\right|}{2^q}\right) e^{i\check{n} \cdot x}, &\text{for }& q\geqslant 0,\\
\triangle_{-1}u=& \sum_{n\in\mathbb{Z}^3} \hat{u}_n \chi \left( \left|\check{n} \right| \right)e^{i\check{n} \cdot x},\\
\tv u =& 0, &\text{for }& q\leqslant -2,
\end{align*}
where $u\in\mathcal{D}'\left(\mathbb{T}^3 \right)$ and  $\hat{u}_n$ are the Fourier coefficients of $u$. The functions $\varphi$ and $\chi$ represent a partition of the unity in $\mathbb{R}$, which means that are smooth functions with compact support such that
\begin{align*}
\text{supp}\;\chi \subset&\; B \left(0,\frac{4}{3}\right), & \text{supp}\;\varphi \subset& \;\mathcal{C}\left( \frac{3}{4},\frac{8}{3}\right),
\end{align*}
and such that for all $t\in\mathbb{R}$,
$$
\chi\left( t\right) +\sum_{q\geqslant 0} \varphi \left( 2^{-q}t\right)=1.
$$

Let us define further the low frequencies cut-off operator
$$
S_q u= \sum_{q'\leqslant q-1}\Tv u.
$$

\subsubsection{Anisotropic paradifferential calculus.}\label{paradifferential calculus}

The dyadic decomposition turns out to be very useful also when it comes to study the product between two distributions. We can in fact, at least formally, write for two distributions $u$ and $v$
\begin{align}\label{decomposition vertical frequencies}
u=&\sum_{q\in\mathbb{Z}}\tv u ; &
v=&\sum_{q'\in\mathbb{Z}}\Tv v;&
u\cdot v = & \sum_{\substack{q\in\mathbb{Z} \\ q'\in\mathbb{Z}}}\tv u \cdot \Tv v.
\end{align}

We are going to perform a Bony decomposition (see  \cite{bahouri_chemin_danchin_book}, \cite{Bony1981}, \cite{chemin_book} for the isotropic case and \cite{chemin_et_al},\cite{iftimie_NS_perturbation} for the anisotropic one).
\\
Paradifferential calculus is  a mathematical tool for splitting the above sum in three parts
$$
u\cdot v = T_u v+ T_v u + R\left(u,v\right),
$$
where
\begin{align*}
T_u v=& \sum_q S_{q-1} u\; \tv v, &
 T_v u= & \sum_{q'} S_{q'-1} v \; \Tv u,&
 R\left( u,v \right) = & \sum_k \sum_{\left| \nu\right| \leqslant 1} \triangle_k  u\; \triangle_{k+\nu} v.
\end{align*}
The following almost orthogonality properties hold
\begin{align*}
\tv \left( \Svq a \Tv b\right)=&0, & \text{if }& \left|q-q'\right|\geqslant 5,\\
\tv \left( \Tv a \triangle_{q'+\nu}b\right)=&0, & \text{if }& q'< q-4,\; \left| \nu \right|\leqslant 1,
\end{align*}
and hence we will often use the following relation
\begin{align}
\tv\left( u\cdot v \right)= &\sum_{\left| q -q'\right| \leqslant 4} \tv\left(S_{q'-1} v\; \Tv u\right) +
\sum_{\left| q -q'\right| \leqslant 4} \tv\left(S_{q'-1} u\; \Tv v\right)+
\sum_{q'\geqslant q-4}\sum_{|\nu|\leqslant 1}\tv\left(  \Tv a \triangle_{q'+\nu}b\right)\nonumber ,\\
=& \sum_{\left| q -q'\right| \leqslant 4} \tv\left(S_{q'-1} v \; \Tv u\right) + \sum_{q'>q-4} \tv\left( S_{q'+2} u \Tv v\right).\label{Paicu Bony deco}
\end{align}

In the paper \cite{chemin_lerner} J.-Y. Chemin and N. Lerner introduced the following decomposition, which will be  used by Chemin et al. in \cite{chemin_et_al} in its anisotropic version. This particular decomposition turns out to be very useful in our context
\begin{multline}
\label{bony decomposition asymmetric}
\tv \left(uv\right) = S_{q-1} u\; \tv v +\sum_{|q-q'|\leqslant 4} \left\lbrace\left[ \tv, \Svq u \right]\Tv v + \left( S_q u-\Svq u \right) \tv\Tv v\right\rbrace
\\
 + \sum_{q'>q-4} \tv \left(  S_{q'+2} v\;\Tv u\right),
\end{multline}
where the commutator $\left[\tv, a\right]b$ is defined as
$$
\left[\tv, a\right]b= \tv \left( ab \right) - a \tv b.
$$
There is an interesting relation of regularity between dyadic blocks and full function in the Sobolev spaces, i.e.
\begin{equation}
\label{regularity_dyadic}
\left\| \tq f \right\|_\2 \leqslant C c_q 2^{-qs}\left\| f \right\|_\Hs,
\end{equation}
with $ \left\| \left\lbrace c_q \right\rbrace_{q\in\mathbb{Z}} \right\|_{\ell^2\left( \mathbb{Z} \right)}\equiv 1 $. In the same way we denote as $ b_q $ a sequence in $ \ell^1 \left( \mathbb{Z} \right) $ such that $ \sum_q \left| b_q \right| \leqslant 1$.\\

In particular in \label{anisotropic_paradiff} Section \ref{global_existence} we shall need paradifferential calculus in the horizontal variables, everything is the same as in the isotropic case except that we shall take the Fourier transform only on the horizontal components, i.e.
$$
\mathcal{F}_h f \left( n_h, x_3 \right)= \int_{\T^2_h} f\left( x_h,x_3 \right)e^{-2\pi i x_h\cdot n_h}\d x_h,
$$
and we can define hence the horizontal truncation operators (as well as the low frequencies cut off, reminders, etc...)  $ \thq $ in the same way as we did for $ \tq $ except that we act only on the horizontal variables. This difference shall be denoted by the fact that that we will always put an index $ h $ when it comes to the horizontal anisotropic paradifferential calculus.

\subsubsection{Some basic estimates.}\label{basic estimates}
The interest in the use of the dyadic decomposition is that the derivative of a function localized in  frequencies of size $2^q$ acts like the multiplication with the factor $2^q$ (up to a constant independent of $q$). In our setting (periodic case) a Bernstein type inequality holds. For a proof of the following lemma in the anisotropic (hence as well isotropic) setting we refer to the work \cite{iftimie_NS_perturbation}. For the sake of self-completeness we state the result in both isotropic and anisotropic setting.

\begin{lemma}\label{bernstein inequality}
Let $u$ be a function such that $\mathcal{F}u $ is supported in $ 2^q\mathcal{C}$, where $\mathcal{F}$ denotes the Fourier transform. For all integers $k$ the following relation holds
\begin{align*}
2^{qk}C^{-k}\left\| u \right\|_{\p}\leqslant & \left\|\left( -\Delta \right)^{k/2} u \right\|_{\p} \leqslant 2^{qk}C^{k}\left\| u \right\|_{\p}.
\end{align*}

Let now $r\geqslant r' \geqslant 1$ be real numbers. Let  $\text{supp}\mathcal{F}u \subset  2^q B$, then
\begin{align*}
\left\| u \right\|_{ L^r}\leqslant & C \cdot 2^{3q\left( \frac{1}{r'}-\frac{1}{r}\right)}\left\| u \right\|_{ L^{r'}}.
\end{align*}
Let us consider now a function $ u $ such that $ \mathcal{F}u $ is supported in $ 2^q\mathcal{C}_h \times 2^{q'}\mathcal{C}_v $. Let us define $ D_h= \left( -\Dh \right)^{1/2}, D_3=\left| \partial_3 \right| $, then
$$
C^{-q-q'}2^{qs+q's'}\left\| u \right\|_\p\leqslant 
\left\| D_h^s D_3^{s'} u \right\|_\p \leqslant C^{q+q'}2^{qs+q's'}\left\| u \right\|_\p,
$$
and given $ 1\leqslant p'\leqslant p\leqslant \infty $, $1\leqslant r'\leqslant r\leqslant \infty $, then 
\begin{align*}
\left\| u \right\|_{L^p_hL^r_v} \leqslant & C^{q+q'} 2^{2q \left( \frac{1}{p'}-\frac{1}{p} \right) + q' \left( \frac{1}{r'}-\frac{1}{r} \right)
} \left\| u \right\|_{L^{p'}_h L^{r'}_v},\\
\left\| u \right\|_{L^r_vL^p_h} \leqslant & C^{q+q'} 2^{2q \left( \frac{1}{p'}-\frac{1}{p} \right) + q' \left( \frac{1}{r'}-\frac{1}{r} \right)
} \left\| u \right\|_{ L^{r'}_vL^{p'}_h}.
\end{align*}
\end{lemma}

The following are inequalities of Gagliardo-Niremberg type, which combined with the anisotropic version of the Bernstein lemma will give us some information that we will use continuously all along the paper. We will avoid to give the proofs of such tools since they are already present in \cite{paicu_rotating_fluids}.

\begin{lemma}
There exists a constant $C$ such that for all periodic vector fields $u$ on $\mathbb{T}^3$ with zero horizontal average ($\int_{\mathbb{T}^2_h}u\left( x_h, x_3 \right)\d x_h=0$) we have
\begin{equation}\label{GN type ineq}
\left\|u\right\|_{L^2_v L^4_h}\leqslant C\cdot \left\| u\right\|^{1/2}_{\2}\left\| \nh u \right\|^{1/2}_{\2}.
\end{equation}
\end{lemma}

Finally we state a lemma that shows that the commutator with the truncation operator  is a regularizing operator. 
\begin{lemma}\label{estimates commutator}
Let $\mathbb{T}^3$ be a 3D torus and $p,r,s$ real positive numbers such that $r',s',p,r,s\geqslant 1 $ $\frac{1}{r'}+\frac{1}{s'}=\frac{1}{2}$ and $\frac{1}{p}=\frac{1}{r}+\frac{1}{s}$. There exists a constant $C$ such that for all vector fields $u$ and $v$ on $\mathbb{T}^3$ we have the inequality
$$
\left\| \left[ \tv , u\right] v \right\|_{L^2_vL^p_h}\leqslant C \cdot 2^{-q}\left\|\nabla u \right\|_{L^{r'} _v L^r_h} \left\|v\right\|_{L^{s'}_vL^s_h},
$$
indeed there exists an isotropic counterpart of such Lemma (see \cite{paicu_NS_periodic}, \cite{Scrobo_primitive_horizontal_viscosity_periodic}).
\end{lemma}

\section{The linear problem.}\label{linear problem}
Let us introduce at first some notation. Let us consider the generic linear problem associated with the linear operator $ \PA $:
\begin{equation}\label{equation linear}
\left\lbrace
\begin{aligned}
&\partial_\tau W + \PA \; W = 0,\\
&\dive \; w=0,\\
& W= \left( w, W^4 \right),\\
&\left. W \right|_{\tau=0}=W_0.
\end{aligned}
\right.
\end{equation}
$\mathcal{A}$ is the skew symmetric penalized matrix defined in \eqref{matrici}. $\mathbb{P}$ is the Leray projection onto the divergence free vector fields, without changing $V^{4}$ which is defined in \eqref{Leray projector}. In the present section (and everywhere) the Fourier modes are considered to be $ \check{n}= \left( n_1/a_1, n_2/a_2, n_3/a_3 \right) $ where $ n_i\in \mathbb{Z} $ and the $ a_i $'s are the parameters of the three-dimensional torus. We shall generally ignore the check notation (unless differently specified) in order to simplify the overall notation.\\

To the sake of completeness we give here the action of the matrix $ \PA $ in the Fourier space, which is
$$
\mathcal{F}\left( \PA \; u \right) = 
\left( \begin{array}{cccc}
0&0&0&-\frac{ n_1 n_2}{\left| n \right|^2}\\
0&0&0&-\frac{ n_2 n_3}{\left| n \right|^2}\\
0&0&0&1-\frac{ n_3^2}{\left| n \right|^2}\\
0&0&-1&0
\end{array} \right)\; \hat{u}_n.
$$

 Such kind of equation has been thoroughly first studied by Poincar\'e in \cite{Poincare_linear}. The study of the linear equation \eqref{equation linear} is essential in the study of the nonlinear problem \eqref{perturbed BSSQ}. The solution of \eqref{equation linear} is obviously
\begin{equation}
\label{eq:solution_linear_problem}
W\left( \tau \right) = e^{-\tau\PA} W_0 = \mathcal{L}\left( -\tau \right) W_0.
\end{equation}
We want to give an explicit sense to the propagator $ e^{-\tau\PA} $. To do so we perform a spectral analysis of the operator $ \PA $. After some  calculations we obtain that the matrix $\PA$ admits an eigenvalue $\omega^0\left( n \right)\equiv 0$ with multiplicity 2 and other two eigenvalues
\begin{equation}
\label{eigrnvalue}
\omega^\pm \left( n \right)= \pm i \frac{\left|n_h\right|}{\left|n\right|}=\pm i\omega\left( n \right).
\end{equation}
The matrix $ \left( \widehat{\PA} \right)_n $  admits a basis of normal (in the sense that they have norm one) eigenvectors. In particular the basis is the following one
\begin{align*}
\tilde{e}^0_1=&\left( \begin{array}{c}
1\\0\\0\\0
\end{array} \right)
&
\tilde{e}^0_2=&\left( \begin{array}{c}
0\\1\\0\\0
\end{array} \right)
 &
e^\pm\left( n \right)= & \frac{1}{\sqrt{2}}\left(
\begin{array}{c}
\pm \; i \; \frac{n_1 n_3}{\left| n_h \right|\; \left| n \right|}\\[2mm]
\pm \; i \; \frac{n_2 n_3}{\left| n_h \right|\; \left| n \right|}\\[2mm]
\mp \; i \; \frac{\left| n_h \right|}{\left| n \right|}\\[2mm]
1
\end{array}\right).
\end{align*} 
We imposed that the solutions of \eqref{equation linear} are divergence-free, in the sense that they are orthogonal, in the Fourier space, to the vector $ \left( n_1, n_2, n_3, 0 \right) $. Now, not all the subspace $ \mathbb{C}\tilde{e}^0_1 \oplus \mathbb{C}\tilde{e}^0_2 $, which is the kernel of the operator $ \PA $, satisfies this property. In any case there exist a subspace of  $ \mathbb{C}\tilde{e}^0_1 \oplus \mathbb{C}\tilde{e}^0_2 $ which is divergence free. This space is the space generated by 
$$
e^0\left( n \right)=\frac{1}{\left|n_h\right|} \left(
\begin{array}{c}
-n_2\\
n_1\\
0\\
0
\end{array}\right),
$$
we underline again that $ \mathbb{C}e^0 \subset \mathbb{C}\tilde{e}^0_1 \oplus \mathbb{C}\tilde{e}^0_2 $.\\
We have hence identified a basis of divergence-free, orthogonal eigenvectors associated to the linear problem \eqref{equation linear}, which is
 \begin{align}\label{eigenvectors}
 e^0\left( n \right)=&\frac{1}{\left|n_h\right|} \left(
\begin{array}{c}
-n_2\\
n_1\\
0\\
0
\end{array}\right), &
e^\pm\left( n \right)= &\frac{1}{\sqrt{2}} \left(
\begin{array}{c}
\pm \; i \; \frac{n_1 n_3}{\left| n_h \right|\; \left| n \right|}\\[2mm]
\pm \; i \; \frac{n_2 n_3}{\left| n_h \right|\; \left| n \right|}\\[2mm]
\mp \; i \; \frac{\left| n_h \right|}{\left| n \right|}\\[2mm]
1
\end{array}\right).
 \end{align}
A case of particular interest which shall be crucial in Section \ref{the limit} is the subspace $ \set{ n_h = 0 } $ of the frequency space. Performing the required computation we prove that the only eigenvalue is $ \omega \left( n \right) \equiv 0$ with multiplicity four. The Fourier multiplier $ \left( \widehat{\PA} \right) \left( n \right) $ associated to $ \PA $ in this case  is
$$
\left( \widehat{\PA} \right) \left( 0, n_3 \right)= \left( 
\begin{array}{cccc}
0&0&0&0\\
0&0&0&0\\
0&0&0&0\\
0&0&-1&0
\end{array}
 \right),
$$ 
which admits three eigenvalues related to $ \omega $:
\begin{align}\label{eigenvectors_nh=0}
\tilde{e}^0_1 = &
\left( \begin{array}{c}
1\\0\\0\\0
\end{array} \right),
&
\tilde{e}^0_2 = &
\left( \begin{array}{c}
0\\1\\0\\0
\end{array} \right),
&
\tilde{e}^0_3 = &
\left( \begin{array}{c}
0\\0\\0\\1
\end{array} \right).
\end{align}
We underline the fact that $ \tilde{e}_i^0, i=1,2,3 $ are divergence-free on the restriction $ \set{ n_h=0 } $ of the Fourier space. We remark the fact that the hypothesis \eqref{eq:Horizontal_average_initial_data} automatically excludes the case which the initial data $ V_0 $ is a function depending on $ x_3 $ only. In fact
\begin{align*}
\int_{\T_h^2}V_0 \left( x_3 \right)\d y_h =0 \  \Longrightarrow \  V\left( x_3 \right)=0 \text{ for each } x_3 \in \T^1_v.
\end{align*}
This case is hence \textit{not considered} in the present work.\\
 The eigenvectors in \eqref{eigenvectors} are  orthogonal with respect the standard $ \mathbb{C}^4 $ scalar product,
 whence the generic solution given in \eqref{eq:solution_linear_problem} can be expressed in the following form
 $$
 W= \bar{W}+ W_\osc.
 $$
 We denoted as $ \bar{W} $ the orthogonal projection of $ W $ onto the space $ \mathbb{C}e^0 $, i.e. onto the divergence-free part of the kernel. This projection takes the form (in the Fourier variables)
 \begin{equation}\label{eq:proj_bar}
 \mathcal{F} \; \bar{W} \left( n \right)= \left(\left. \mathcal{F} \; W \left( n \right) \right| e^0  \left( n \right) \right)_{\mathbb{C}^4} \; e^0 \left( n \right).
 \end{equation}
 In the same way $ W_\osc $ is defined as
 \begin{equation}\label{eq:proj_osc}
\mathcal{F} \; W_\osc \left( n \right) =\left(\left. \mathcal{F} \; W \left( n \right) \right| e^+  \left( n \right) \right)_{\mathbb{C}^4} \; e^+ \left( n \right) + \left(\left. \mathcal{F} \; W \left( n \right) \right| e^-  \left( n \right) \right)_{\mathbb{C}^4} \; e^- \left( n \right),
 \end{equation}
 i.e. the decomposition of the (Fourier) data along the directions of $ e^+,\; e^- $.\\
 
 We shall denote these two parts of the solution respectively as the \textit{oscillating} and the \textit{non-oscillating} part of the solution. This choice of the names has an easy mathematical justification. Let us in fact consider $ \bar{W}_0 $ and let us consider the evolution imposed by the laws of the system \eqref{equation linear} on such vector field. By mean of the explicit solution given in \eqref{eq:solution_linear_problem} we obtain (recall that $ \bar{W} $ belongs to $ \ker \; \PA $)
 $$
 \bar{W} \left( t \right)= \bar{W}_0.
 $$
 Hence the non-oscillating part of the solution $ \bar{W} $ is in fact a stationary (in the sense that is not time-dependent) flow. This is reasonable since once we consider the linear system \eqref{linear problem} restricted on $ \ker \; \PA $ there is no external force at all acting on it.\\
 On the other hand along the direction of (say) $ e^+ $ the evolution at time $ \tau $ of the solution has direction (in the Fourier space)
 $$
 \left( e^{-\tau \left( \widehat{\PA} \right)_n} \right) \; e^+ \left( n \right)=e^{-i \; \tau \; \omega \left( n \right)}\; e^+\left( n \right),
 $$ 
 hence it spins with a angular speed $ \omega=\omega\left( n \right)=\frac{\left| n_h \right|}{\left| n \right|} $.\\
 Introducing hence a small parameter $ \varepsilon $ we act, on a physical point of view, on the system in very well-defined way:  the spinning linear force grows and with it the turbulent behavior of the solution.\\
 
 Indeed as long as we consider a generic time-dependent nonlinearity the problem does not behave in such a rigid and well-defined way. Let us consider hence a nonlinear problem associated to \eqref{equation linear}
 \begin{equation}\label{equation nonlinear generic}
\left\lbrace
\begin{aligned}
&\partial_\tau W_N + \PA \; W_N = \mathcal{N}\left( \tau \right),\\
&\dive \; w_N=0,\\
& W_N= \left( w_N, W_N^4 \right),\\
&\left. W_N \right|_{\tau=0}=W_{N,0}.
\end{aligned}
\right.
\end{equation}
The nonlinearity $ \mathcal{N}\left( \tau \right) $ is very generic and only time-dependent, but this is not restrictive, since we want to give a qualitative analysis of the behavior of the solutions.\\
By mean of the Duhamel formula the solution is expressed as
$$
W_N \left( \tau \right)= e^{-\tau \PA} W_{N,0} + \int_0^\tau e^{-\left( \tau-\sigma \right)\PA} \mathcal{N}\left( \sigma \right) \d \sigma.
$$
This generic formulation does not say much, but we can still extract interesting information. Let us denote $ \bar{\mathcal{N}} $ the projection of $ \mathcal{N} $ onto the nonoscillatory space, i.e.
$$
\mathcal{F}\; \bar{\mathcal{N}}= \left(\left. \mathcal{F\; N} \right| e^0  \right)\; e^0.
$$
The projection of $ W_N $ onto the nonoscillatory space shall hence be described by the law
$$
\bar{W}_N \left( \tau \right)= \bar{W}_{N,0}+\int_0^t \bar{\mathcal{N}} \left( \sigma \right)\d \sigma.
$$
The influence of the propagator (spinning behavior) is not any more a direct consequence of the application of Duhamel formula, but it can still be present  as long as $ \bar{\mathcal{N}} $ is depending on the spinning eigenvectors $ e^\pm $. We shall see that this will be a major problem in the comprehension of the limiting process as $ \varepsilon\to 0 $.
 
 \section{The filtered limit.}\label{sec:filtered limit}
 
 Our strategy shall be to "filter out" the system \eqref{perturbed BSSQ} by mean of the propagator $\mathcal{L}$ defined above. Such technique is classic in singular problems on the torus (\cite{gallagher_schochet}, \cite{grenierrotatingeriodic}, \cite{paicu_rotating_fluids}, \cite{Scrobo_primitive_horizontal_viscosity_periodic}). Let us apply from the left the operator $\mathcal{L}\left( \frac{t}{\varepsilon} \right)$ to the equation \eqref{perturbed BSSQ}.  Setting $U^\varepsilon=\mathcal{L}\left( \frac{t}{\varepsilon} \right) V^\varepsilon$ we obtain that the vector field $U^\varepsilon$ satisfies the following evolution equation
 \begin{equation}\tag{$\mathcal{S}_\varepsilon$}
\label{filtered system}
\left\lbrace
\begin{array}{l}
\partial_t U^\varepsilon+\mathcal{Q}^\varepsilon \left(U^\varepsilon,U^\varepsilon\right)-
\mathbb{D}^\varepsilon U^\varepsilon= 0,\\
\dive v^\varepsilon =0,\\
\bigl. U^\varepsilon\bigr|_{t=0}=V_0,
\end{array}
\right.
\end{equation}
where
\begin{align*}
\mathcal{Q}^\varepsilon \left( A,B\right)= & \frac{1}{2} \Lplus \mathbb{P} \left[ \Lminus A \cdot \nabla \Lminus B + \Lminus B \cdot \nabla \Lminus A
\right],\\
\mathbb{D}^\varepsilon A = & \Lplus \mathbb{D} \Lminus A.
\end{align*}

It is interesting to notice that the application of the Poincar\'e semigroup $\mathcal{L}$ allowed us to  deduce an equation (namely \eqref{filtered system}) on which we can obtain uniform bounds for the sequence $ \left( \partial_t U^\varepsilon \right)_{\varepsilon} $. It results in fact that $ \left( \partial_t U^\varepsilon \right)_{\varepsilon} $ is uniformly bound in $ L^p \left( \R_+; H^{-N} \right), \ p \in [2, \infty] $ for $ N $ large. This  shall result to be fundamental in order to obtain some compactness result in the same fashion as it is done for solutions \`a la Leray of \NS\ equations.
\\

\subsection{Uniform bounds of the weak solutions and formal identification of the limit system.} \label{sec:uniform_bounds}

In this section we prove Theorem \ref{thm:topological_convergence}; whose extended claim is here proposed:

\begin{lemma} \label{lem:conv_comp_arg}
 The sequence $ \left( U^\varepsilon \right)_{\varepsilon > 0} $ of distributional solutions of \eqref{perturbed BSSQ} identified in Theorem \ref{thm:Leray_adapted} is uniformly bounded in 
\begin{equation*}
L^\infty \left( \R_+; \2 \right)\cap L^2 \left( \R_+; H^1 \left( \T^3 \right) \right),
\end{equation*}
and sequentially compact in $ L^2_{\loc} \left( \R_+; \2 \right) $. Every $ U $ belonging to the topological closure of $  \left( U^\varepsilon \right)_{\varepsilon > 0}  $ w.r.t. the $ L^2_{\loc} \left( \R_+; \2 \right) $ topology belongs to the energy space $ L^\infty \left( \R_+; \2 \right)\cap L^2 \left( \R_+; H^1 \left( \T^3 \right) \right) $ and is a distributional solution of the limit system \eqref{limit system}.
\end{lemma}

Lemma \ref{lem:conv_comp_arg} is composed of two parts:
\begin{itemize}
\item Topological convergence of a (not relabeled) subsequence $ \left( U^\varepsilon \right)_{\varepsilon >0} $ to an element $ U $ as $ \varepsilon \to 0 $ in some suitable topology and,

\item Determination of the limit system to whom \eqref{filtered system} converges. 
\end{itemize}
More in specific the second point above proves that there exists a bilinear form $ \mathcal{Q} $ such that
\begin{equation*}
\mathcal{Q}^\varepsilon \left( U^\varepsilon, U^\varepsilon \right)\xrightarrow{\varepsilon \to 0} \mathcal{Q} \left( U, U \right),
\end{equation*}
in a weak sense.\\
The first point above will be proved in Lemma \ref{lem:weak_convergence_AL_lemma}, while the second will be studied in detail in Lemma \ref{lem:limit_forms_generic}.\\

\begin{lemma}\label{lem:weak_convergence_AL_lemma}
Let $ U^\varepsilon $ be a Leray solution of \eqref{perturbed BSSQ}. Let the initial data $ V_0 $ be  bounded in $ \2 $. The sequence $ \left( U^\varepsilon \right)_\varepsilon $ is compact in the space $ L^2_\loc \left( \R_+; \2 \right) $ and converges (up to a subsequence) to an element $ U $ which belongs to the space 
$$
L^\infty \left( \R_+; \2 \right)\cap L^2 \left( \R_+; {H}^1 \left( \T^3 \right) \right),
$$
and the following uniform bound holds
\begin{equation}\label{eq:L2_uniform_bound}
\left\| U^\varepsilon \left( t \right) \right\|^2_\2 + 2 c \int_0^t \left\| \nabla U^\varepsilon \left( \tau \right) \right\|^2_\2 \d \tau
\leqslant \; \left\| V_0 \right\|^2_\2,
\end{equation}
where $ c=\min \set{ \nu, \nu' } $.
\end{lemma}

\begin{proof}
A standard $ \2 $ estimate on the equation \eqref{filtered system} shows that
$$
\left\| U^\varepsilon \left( t \right) \right\|^2_\2 + 2 c \int_0^t \left\| \nabla U^\varepsilon \left( \tau \right) \right\|^2_\2 \d \tau
\leqslant \; \left\| V_0 \right\|^2_\2.
$$
Let us prove that $ \left( \partial_t U^\varepsilon \right)_\varepsilon $ is uniformly bounded in $ L^2_{\loc} \left( \R_+; {H}^{-\frac{3}{2}}\left( \T^3 \right) \right) $ in terms of the $ \2 $ norm of the initial data $ V_0 $. Since $ \mathcal{L} \left( \tau \right) $ is unitary as an application between any Sobolev space $ H^\sigma, \sigma\in \R $ we can safely say that as long as concerns energy estimates in Sobolev spaces we can identify $ \mathcal{Q}^\varepsilon \left( U^\varepsilon, U^\varepsilon \right)\sim U^\varepsilon \cdot \nabla U^\varepsilon $. Being this the case we can use the product rules in Sobolev spaces to deduce that
\begin{align*}
\left\| U^\varepsilon \cdot \nabla U^\varepsilon \right\|_{L^2_{\loc}\left( \R_+;  H^{-3/2} \right)}\leqslant & \ C \left\| U^\varepsilon \otimes U^\varepsilon \right\|_{L^2_{\loc}\left( \R_+;H^{-1/2}\right)},\\
\leqslant & \ C \left\| U^\varepsilon \right\|_{L^2_{\loc}\left( \R_+;H^1\right)}
\left\| U^\varepsilon \right\|_{L^\infty \left( \R_+; L^2 \right)}.
\end{align*}
It is very easy moreover to deduce that $ \left\| -\mathbb{D}^\varepsilon U^\varepsilon \right\|_{L^2_{\loc} \left( \R_+; H^{-1} \right)}\lesssim \left\|  U^\varepsilon \right\|_{L^2_{\loc} \left( \R_+; H^{1} \right)} $. Whence since $ \partial_t U^\varepsilon= -\mathcal{Q}^\varepsilon \left( U^\varepsilon, U^\varepsilon \right) +\mathbb{D}^\varepsilon U^\varepsilon $ we conclude. 
 Since $ \2 $ is compactly embedded in $ H^{-3/2} \left( \T^3 \right) $ and that $ H^1 \left( \T^3 \right) $ is continuously embedded in  $ \2 $ it is sufficient to apply Aubin-Lions lemma \cite{Aubin63} to deduce the claim.
\end{proof}

The convergence of $ U^\varepsilon $ to the element $ U $ does not give any qualitative information of the (eventual) solution which is satisfied by $ U $. Especially the bilinear limit involving $ \mathcal{Q}^\varepsilon \left( U^\varepsilon, U^\varepsilon \right) $ is a priori not well-defined.\\

The result we prove now is needed in order to prove that, in the limit as $ \varepsilon \to 0 $, the limit $ \displaystyle \lim _{\varepsilon\to 0} \mathcal{Q}^\varepsilon \left( U^\varepsilon, U^\varepsilon \right) $ belongs to the space spanned by the eigenvectors $ e^0, \ e^\pm $.
\begin{lemma}\label{lem:limit_horizontal_average}
The limit
\begin{equation*}\label{lem:limit_horizontal_average}
\lim _{\varepsilon \to 0} \left( \int_{\T^2_h} \left( \mathcal{Q}^\varepsilon \left( U^\varepsilon, U^\varepsilon \right) \right)^h \d x_h , 0, 0 \right)=0,
\end{equation*}
hods in a distributional sense.
\end{lemma}

Since the proof of Lemma \ref{lem:limit_horizontal_average} is rather long and technical is postponed at the end of the present section.

 By mean of stationary phase theorem  we prove the following lemma
\begin{lemma} \label{lem:limit_forms_generic}
Let $ U^\varepsilon $ be a Leray solution of \eqref{filtered system} and let $ U $ be the limit of one of the converging subsequences  of $ \left( U^\varepsilon \right)_\varepsilon $ identified in Lemma \ref{lem:weak_convergence_AL_lemma}. Then the following limits hold (in the sense of distributions, subsequence not relabeled)
\begin{align*}
\lim_{\varepsilon\to 0}\mathcal{Q}^\varepsilon \left( U^\varepsilon, U^\varepsilon \right)= & \  \mathcal{Q}\left( U, U \right),\\
\lim_{\varepsilon\to 0} \mathbb{D}^\varepsilon \; U^\varepsilon = & \  \mathbb{D}\; U,
\end{align*}
where $ \mathcal{Q} $ and $ \mathbb{D} $ has the following form
\begin{align}
\mathcal{F\; Q}\left( U, U \right) = & \  \mathbb{P}_n \sum_{\substack{\omega^{a,b,c}_{k,m,n}=0\\k+m=n \\ a,b,c \in \set{ 0,\pm}}} \left( \left. \left( \sum_{j=1,2,3}{U}^{a,j} \left( k \right)  m_j \right) \; {U}^b \left( m \right) \right| e^c \left( n \right) \right)_{\mathbb{C}^4}\; e^c \left( n \right)\label{eq:def_limit_bilinear_form}\\
\mathcal{F} \; \mathbb{D}\; U = & \sum_{\omega^{a,b}_n=0} \left( \left. \mathbb{D}\left( n \right)  {U}^a \left( n \right) \right| e^b \left( n \right) \right)_{\mathbb{C}^4}\; e^b \left( n \right),\label{eq:def_limit_linear_form}
\end{align}
where $ \omega^{a,b,c}_{k,m,n}=\omega^a \left( k \right)+\omega^b \left( m \right)-\omega^c \left( n \right) $ and $ \omega^{a,b}_n= \omega^a\left( n \right)- \omega^b \left( n \right) $, the Fourier multiplier $ \mathbb{D}\left( n \right) $ in noting but the Fourier multiplier associated to the matrix $ \mathbb{D} $ defined in \eqref{matrici}, the eigenvalues $ \omega^i $ are defined in \eqref{eigrnvalue} and the operator $ \mathbb{P} $ is defined in \eqref{Leray projector}.
\end{lemma}

\begin{rem}
Lemma \ref{lem:limit_horizontal_average} proves that only the firs two components of the horizontal average of $ \mathcal{Q}^\varepsilon \left( U^\varepsilon, U^\varepsilon \right) $ converge (weakly) to zero. We need to prove such result since the eigenvectors defined in \eqref{eigenvectors} present in their firs two components a Fourier symbol of the form $ \left| n_h \right|^{-1} $, and such operator is well-defined only for vector fields with zero horizontal average. These are hence applied on the bilinear interaction as it is shown in \eqref{eq:def_limit_bilinear_form}. \fine
\end{rem}

\begin{proof}
We start proving \eqref{eq:def_limit_linear_form} since it is easier. We claim that
\begin{equation*}
\mathbb{D}^\varepsilon U^\varepsilon \xrightarrow{\varepsilon \to 0} \mathbb{D} U,
\end{equation*}
in $\mathcal{D}'$. Indeed via standard manipulations we can express the difference $ \mathbb{D}^\varepsilon U^\varepsilon - \mathbb{D} U $ as
\begin{equation*}
\mathbb{D}^\varepsilon U^\varepsilon - \mathbb{D} U = \mathbb{D} \left( U^\varepsilon -U \right) + \left( \mathbb{D}^\varepsilon -\mathbb{D} \right) U^\varepsilon.
\end{equation*}
The element $ \mathbb{D} \left( U^\varepsilon -U \right)\xrightarrow{\varepsilon \to 0}0 $ in $ \mathcal{D}' $ since $ U^\varepsilon\to U $ w.r.t. the $ L^2_{\loc} \left( \R_+; \2 \right) $ topology as it is proved in Lemma \ref{lem:conv_comp_arg}.

\noindent 
Hence all we have to prove is that
$$
\mathcal{F} \left( \left( \mathbb{D}^\varepsilon -\mathbb{D} \right) U^\varepsilon \right) =
\sum_{\omega^{a,b}_n \neq 0} e^{i \frac{t}{\varepsilon}\; \omega^{a,b}_n} \left(  \mathbb{D}\left( n \right)  {U}^{a, \varepsilon} \left( n \right) \left| e^b \left( n \right) \right. \right)_{\mathbb{C}^4}\; e^b \left( n \right) \to 0,
$$
as $ \varepsilon\to 0 $ in the sense of distributions. To do so we consider $ \phi \in \mathcal{D}\left( \R_+\times \T^3 \right) $. Since for $ s,t<3/2 $, $ s+t>0 $ the map
$$
\begin{aligned}
H^s\left( \T^3 \right)\times H^t \left( \T^3 \right) &\; \to & & H^{s+t-\frac{3}{2}}\left( \T^3 \right),\\
\left( u,v \right)& \; \mapsto & & u\otimes v,
\end{aligned}
$$
is continuous we deduce that
$$
\left( \partial_t U^\varepsilon \right)_\varepsilon \hspace{5mm}\text{uniformly bounded in}\; L^2 \left( \R_+; H^{-3/2}\left( \T^3 \right) \right).
$$
We want hence to prove that
\begin{align}\label{eq:termine che deve andare a zero 1}
S_1^\varepsilon=\sum_n \sum_{\omega^{a,b}_n \neq 0} \int e^{i \frac{t}{\varepsilon}\; \omega^{a,b}_n} \left(  \mathbb{D}\left( n \right)  {U}^{a, \varepsilon} \left(t, n \right) \left| e^b \left( n \right) \right. \right)_{\mathbb{C}^4}\; e^b \left( n \right)\hat{\phi} \left( t,n \right)\d t \xrightarrow{\varepsilon\to 0} 0.
\end{align}
Indeed the sum on the left-hand-side of \eqref{eq:termine che deve andare a zero 1} is well defined (in the sense that the sum is smaller than infinity). We can decompose it as
\begin{align*}
S_1^\varepsilon= &\; S_{1,N}^\varepsilon+ S_1^{N, \varepsilon},\\
S_{1,N}^\varepsilon = & \; \sum_{\left| n \right|\leqslant N} \sum_{\omega^{a,b}_n \neq 0} \int e^{i \frac{t}{\varepsilon}\; \omega^{a,b}_n} \left(  \mathbb{D}\left( n \right)  {U}^{a, \varepsilon} \left(t, n \right) \left| e^b \left( n \right) \right. \right)_{\mathbb{C}^4}\; e^b \left( n \right)\hat{\phi} \left( t,n \right)\d t,\\
S_1^{N, \varepsilon} = & \; \sum_{\left| n \right| > N} \sum_{\omega^{a,b}_n \neq 0} \int e^{i \frac{t}{\varepsilon}\; \omega^{a,b}_n} \left(  \mathbb{D}\left( n \right)  {U}^{a, \varepsilon} \left(t, n \right) \left| e^b \left( n \right) \right. \right)_{\mathbb{C}^4}\; e^b \left( n \right)\hat{\phi} \left( t,n \right)\d t.
\end{align*} 
The term $ S^{N, \varepsilon}_1 $ is indeed an $ o_N \left( 1 \right) $ function, considering in fact that the symbol $ \mathbb{D}\left( n \right) $ can be bounded as $ \left| \mathbb{D}\left( n \right) \right|\leqslant C \left| n \right|^2 $ we deduce 
\begin{align*}
S_1^{N, \varepsilon} \leqslant & \; \frac{1}{N}\sum_{\left| n \right| > N} \sum_{\omega^{a,b}_n \neq 0} \int    \left| \mathbb{D}\left( n \right)  {U}^{a, \varepsilon} \left(t, n \right) \right| \left| n \right|\left| \hat{\phi} \left( t,n \right) \right|\d t,\\
 \leqslant & \; \frac{C}{N}\sum_{\left| n \right| > N} \sum_{\omega^{a,b}_n \neq 0} \int    \left| n \right|\left| \  {U}^{a, \varepsilon} \left(t, n \right) \right| \left| n \right|^2\left| \hat{\phi} \left( t,n \right) \right|\d t\\
 \leqslant & \; \frac{C}{N} \left\| U^\varepsilon \right\|_{L^2 \left( \R_+; H^1 \right)} \left\| \phi \right\|_{L^2 \left( \R_+; H^2 \right)},
\end{align*}
which indeed tends to zero as $ N\to \infty $ thanks to the uniform bound \eqref{eq:L2_uniform_bound}.\\
 For the term $ S_{1,N}^\varepsilon $ 
We exploit the fact that
\begin{align}
\label{eq:derivata in tempo}
e^{i \frac{t}{\varepsilon}\; \omega^{a,b}_n} = -\frac{i\; \varepsilon}{\omega^{a,b}_n}\; \partial_t\left( e^{i \frac{t}{\varepsilon}\; \omega^{a,b}_n} \right),
\end{align}
and the fact that $ \left| \omega^{a,b}_n \right|\geqslant c=c_N > 0 $ in the set $ \left| n \right|\leqslant N $. Using \eqref{eq:derivata in tempo} on $ S_{1,N} $ and integrating by parts we obtain that
\begin{multline*}
S_{1,N}^\varepsilon
=
\sum_{\left| n \right|\leqslant N} \sum_{\omega^{a,b}_n \neq 0} 
\frac{i\; \varepsilon}{\omega^{a,b}_n}
\int e^{i \frac{t}{\varepsilon}\; \omega^{a,b}_n} \left(  \mathbb{D}\left( n \right)  \partial_t {U}^{a, \varepsilon} \left(t, n \right) \left| e^b \left( n \right) \right. \right)_{\mathbb{C}^4}\; e^b \left( n \right)\hat{\phi} \left( t,n \right)\d t\\
+\sum_{\left| n \right|\leqslant N} \sum_{\omega^{a,b}_n \neq 0}
\frac{i\; \varepsilon}{\omega^{a,b}_n}
 \int e^{i \frac{t}{\varepsilon}\; \omega^{a,b}_n} \left(  \mathbb{D}\left( n \right)  {U}^{a, \varepsilon} \left(t, n \right) \left| e^b \left( n \right) \right. \right)_{\mathbb{C}^4}\; e^b \left( n \right)\partial_t\hat{\phi} \left( t,n \right)\d t.
\end{multline*}
It is obvious that the term 
$$
\sum_{\left| n \right|\leqslant N} \sum_{\omega^{a,b}_n \neq 0} \frac{i\; \varepsilon}{\omega^{a,b}_n} \int e^{i \frac{t}{\varepsilon}\; \omega^{a,b}_n} \left(  \mathbb{D}\left( n \right)  {U}^{a, \varepsilon} \left(t, n \right) \left| e^b \left( n \right) \right. \right)_{\mathbb{C}^4}\; e^b \left( n \right)\partial_t\hat{\phi} \left( t,n \right)\d t \xrightarrow{\varepsilon\to 0} 0,
$$
hence we shall focus on the other one. Since $ \left| \omega^{a,b}_n \right|\geqslant c=c_N >0 $, on the set $ \left| n \right|\leqslant N $ and $ \left| e^b \right|\equiv 1 $ and the fact that the symbol $ \left| \mathbb{D} \left( n \right) \right|\leqslant C \left| n \right|^2 $ we can deduce
\begin{multline*}
\left| \sum_{\left| n \right|\leqslant N} \sum_{\omega^{a,b}_n \neq 0} 
\frac{i\; \varepsilon}{\omega^{a,b}_n}
\int e^{i \frac{t}{\varepsilon}\; \omega^{a,b}_n} \left(  \mathbb{D}\left( n \right)  \partial_t {U}^{a, \varepsilon} \left(t, n \right) \left| e^b \left( n \right) \right. \right)_{\mathbb{C}^4}\; e^b \left( n \right)\hat{\phi} \left( t,n \right)\d t \right|\\
\lesssim 
\sum_{\left| n \right|\leqslant N} \sum_{\omega^{a,b}_n \neq 0} 
\varepsilon
\left| \int    \partial_t {U}^{a, \varepsilon} \left(t, n \right) \left| n \right|^2\hat{\phi} \left( t,n \right)\d t \right| \\
\leqslant C\; \varepsilon
\left\| \partial_tU^\varepsilon \right\|_{L^2\left( \R_+; H^{-3/2} \right)}
\left\| \phi \right\|_{L^2\left( \R_+; H^{7/2} \right)}\to 0.
\end{multline*}
This concludes the proof of \eqref{eq:def_limit_linear_form}.\\

The proof of \eqref{eq:def_limit_bilinear_form} is more delicate.\\
At first: if we consider the equation of the filtered system \eqref{filtered system} it is easy to deduce that
\begin{align*}
\left( \mathcal{Q}^\varepsilon \left( U^\varepsilon, U^\varepsilon \right) \right)_\varepsilon \text{ bounded in } & \; L^4 \left( \R_+; H^{-1/2} \right) \cap L^2 \left( \R_+; H^{-3/2} \right),\\
\left( -\mathbb{D}^\varepsilon U^\varepsilon \right)_\varepsilon \text{ bounded in } & \;  L^2 \left( \R_+; H^{-1} \right),
\end{align*}
uniformly in $ \varepsilon $. From this we deduce that
$$
\left( \mathcal{Q}^\varepsilon \left( U^\varepsilon, U^\varepsilon \right) \right)_\varepsilon \text{ bounded in }  \; L^2_\loc \left( \R_+; H^{-1/2} \right) \cap L^2_\loc \left( \R_+; H^{-3/2} \right),
$$
since $ L^4_\loc \left( \R_+ \right)\hra L^2_\loc \left( \R_+ \right) $. Hence by interpolation in Sobolev spaces
$$
\left( \mathcal{Q}^\varepsilon \left( U^\varepsilon, U^\varepsilon \right) \right)_\varepsilon \text{ bounded in }  \; L^2_\loc \left( \R_+; H^{-1} \right),
$$
uniformly in $ \varepsilon $.\\
This implies hence that
\begin{align*}
\left( \partial_t U^\varepsilon \right)_\varepsilon = \left( - \mathcal{Q}^\varepsilon \left( U^\varepsilon, U^\varepsilon \right) + \mathbb{D}^\varepsilon U^\varepsilon \right)_\varepsilon  \text{ bounded in }  \; L^2_\loc \left( \R_+; H^{-1} \right),
\end{align*}
uniformly in $ \varepsilon $.\\
We can finally focus on the proof of \eqref{eq:def_limit_bilinear_form}.
As is is done for the linear part standard algebraic manipulations on the bilinear form allow us to deduce that
\begin{equation*}
\mathcal{Q}^\varepsilon \left( U^\varepsilon, U^\varepsilon \right) - \mathcal{Q} \left( U, U \right) =
\left( \mathcal{Q}^\varepsilon -\mathcal{Q} \right) \left( U^\varepsilon, U^\varepsilon \right) + \mathcal{Q} \left( U^\varepsilon, U^\varepsilon-U \right) + \mathcal{Q} \left( U^\varepsilon-U, U \right).
\end{equation*}
Again we can assert that
\begin{equation*}
\mathcal{Q} \left( U^\varepsilon, U^\varepsilon-U \right) + \mathcal{Q} \left( U^\varepsilon-U, U \right)\xrightarrow{\varepsilon\to 0} 0,
\end{equation*}
in $ \mathcal{D}' $ to to the convergence of $ U^\varepsilon $ to $ U $ in $ L^2_{\loc} \left( \R_+; \2 \right) $ proved in Lemma \ref{lem:conv_comp_arg}. What it remain hence to be proved is that
\begin{equation}\label{eq:remaining_contr}
\left( \mathcal{Q}^\varepsilon -\mathcal{Q} \right) \left( U^\varepsilon, U^\varepsilon \right) =
\sum_{\substack{k+m=n \\ \omega^{a,b,c}_{k,m,n}\neq 0 \\ a,b,c =0, \pm }}
e^{i \frac{t}{\varepsilon}\omega^{a,b,c}_{k,m,n}} \left(\left. \sum_{j=1,2,3} U^{a,j}\left( k \right) \ m_j U^b\left( m \right) \right| e^c \left( n \right)  \right)_{\mathbb{C}^4} \ e^c \left( n \right) \xrightarrow{\varepsilon\to 0} 0,
\end{equation}
in some weak sense.

\noindent To prove \eqref{eq:remaining_contr} is equivalent, thanks to the orthogonality of the eigenvectors $ e^i $ defined in \eqref{eigenvectors}, to prove that, for each $ \phi \in \mathcal{D}\left( \R_+\times \T^3 \right) $
\begin{align*}
S_2^\varepsilon = \sum_n \sum_{\substack{\omega^{a,b,c}_{k,m,n}\neq 0\\
a,b,c=0,\pm, \\ n=k+m}}\int e^{i\frac{t}{\varepsilon}\; \omega^{a,b,c}_{k,m,n}} U^{a,\varepsilon} \left(t, k \right)\otimes U^{b,\varepsilon} \left(t, m \right) \hat{\phi} \left( t,n \right)\d t\to 0,
\end{align*}
as $ \varepsilon\to 0 $.\\
As it has been done above for the term $ S_1^\varepsilon $ we can decompose $ S_2^\varepsilon $ into
\begin{align*}
S_{2,N}^\varepsilon= & \; 
 \sum_{\substack{\left| n \right|\leqslant N \\ \left| k \right|\leqslant N}} \sum_{\substack{\omega^{a,b,c}_{k,m,n}\neq 0\\
a,b,c=0,\pm, \\ n=k+m}}\int e^{i\frac{t}{\varepsilon}\; \omega^{a,b,c}_{k,m,n}} U^{a,\varepsilon} \left(t, k \right)\otimes U^{b,\varepsilon} \left(t, m \right) \hat{\phi} \left( t,n \right)\d t\\
S^{N, \varepsilon}_2 = & \; S_2^\varepsilon-S_{2,N}^\varepsilon.
\end{align*}
The term $ S^{N, \varepsilon}_2\to 0 $ as $ N\to \infty $ as for the term $ S^{N, \varepsilon}_1 $ above.
Using the fact that
\begin{align*}
e^{i \frac{t}{\varepsilon}\; \omega^{a,b,c}_{k,m,n}} = -\frac{i\; \varepsilon}{\omega^{a,b,c}_{k,m,n}}\; \partial_t\left( e^{i \frac{t}{\varepsilon}\; \omega^{a,b,c}_{k,m,n}} \right),
\end{align*}
and the fact that $ \left| \omega^{a,b,c}_{k,m,n} \right|\geqslant c=c_N >0 $ uniformly in $ k,m,n $ in the frequency set $ \set{ \left| n \right|,\left| k \right|\leqslant N } $ to deduce that
\begin{multline*}
S_{2,N}^\varepsilon=
\sum_{\substack{\left| n \right|\leqslant N \\ \left| k \right|\leqslant N}} \sum_{\substack{\omega^{a,b,c}_{k,m,n}\neq 0\\
a,b,c=0,\pm, \\ n=k+m}}\frac{i\; \varepsilon}{\omega^{a,b,c}_{k,m,n}}\int e^{i\frac{t}{\varepsilon}\; \omega^{a,b,c}_{k,m,n}} \partial_t U^{a,\varepsilon} \left(t, k \right)\otimes U^{b,\varepsilon} \left(t, m \right) \hat{\phi} \left( t,n \right)\d t\\
+\sum_{\substack{\left| n \right|\leqslant N \\ \left| k \right|\leqslant N}} \sum_{\substack{\omega^{a,b,c}_{k,m,n}\neq 0\\
a,b,c=0,\pm, \\ n=k+m}}\frac{i\; \varepsilon}{\omega^{a,b,c}_{k,m,n}}\int e^{i\frac{t}{\varepsilon}\; \omega^{a,b,c}_{k,m,n}}  U^{a,\varepsilon} \left(t, k \right)\otimes\partial_t U^{b,\varepsilon} \left(t, m \right) \hat{\phi} \left( t,n \right)\d t\\
+\sum_{\substack{\left| n \right|\leqslant N \\ \left| k \right|\leqslant N}} \sum_{\substack{\omega^{a,b,c}_{k,m,n}\neq 0\\
a,b,c=0,\pm, \\ n=k+m}}\frac{i\; \varepsilon}{\omega^{a,b,c}_{k,m,n}}\int e^{i\frac{t}{\varepsilon}\; \omega^{a,b,c}_{k,m,n}}  U^{a,\varepsilon} \left(t, k \right)\otimes U^{b,\varepsilon} \left(t, m \right) \partial_t \hat{\phi} \left( t,n \right)\d t.
\end{multline*}
Is obvious that the term 
$$
\sum_{\substack{\left| n \right|\leqslant N \\ \left| k \right|\leqslant N}} \sum_{\substack{\omega^{a,b,c}_{k,m,n}\neq 0\\
a,b,c=0,\pm, \\ n=k+m}}\frac{i\; \varepsilon}{\omega^{a,b,c}_{k,m,n}}\int e^{i\frac{t}{\varepsilon}\; \omega^{a,b,c}_{k,m,n}}  U^{a,\varepsilon} \left(t, k \right)\otimes U^{b,\varepsilon} \left(t, m \right) \partial_t \hat{\phi} \left( t,n \right)\d t \to 0,
$$
while for the first two term on the right-hand-side of the equation above the procedure is the same, hence we focus on the first one only. It is indeed true that 
\begin{multline*}
\left| \sum_{\substack{\left| n \right|\leqslant N \\ \left| k \right|\leqslant N}} \sum_{\substack{\omega^{a,b,c}_{k,m,n}\neq 0\\
a,b,c=0,\pm, \\ n=k+m}}\frac{i\; \varepsilon}{\omega^{a,b,c}_{k,m,n}}\int e^{i\frac{t}{\varepsilon}\; \omega^{a,b,c}_{k,m,n}} \partial_t U^{a,\varepsilon} \left(t, k \right)\otimes U^{b,\varepsilon} \left(t, m \right) \hat{\phi} \left( t,n \right)\d t \right|\\
\leqslant 
\sum_{\substack{\left| n \right|\leqslant N \\ \left| k \right|\leqslant N}} \sum_{\substack{\omega^{a,b,c}_{k,m,n}\neq 0\\
a,b,c=0,\pm, \\ n=k+m}}C\; \varepsilon\int \left|  \partial_t U^{a,\varepsilon} \left(t, k \right) \right| \left| U^{b,\varepsilon} \left(t, m \right) \hat{\phi} \left( t,n \right) \right|\d t\\
\leqslant
C\; \varepsilon \left\| \partial_t U^\varepsilon \right\|_{L^2_\loc \left( \R_+; H^{-1} \right)} \left\| U^\varepsilon \phi \right\|_{L^2_\loc \left( \R_+; H^{1} \right)},
\end{multline*}
which  uniformly tends to zero w.r.t. $ \varepsilon $ thanks to the uniform bounds given above, concluding the proof.
\end{proof}

\subsection{Oscillating behavior of $ V^\varepsilon $.} \label{sec:oscillating_behavior} The above uniform bounds give a shady determination of the limit function $ U^\varepsilon $. By definition of $ U^\varepsilon =\Lplus  V^\varepsilon $ with $ V^\varepsilon $ distributional solution of \eqref{perturbed BSSQ} we want to determinate some qualitative connection between the oscillating behavior of the filtered system \eqref{filtered system} and the initial system \eqref{perturbed BSSQ}.\\

Thanks to Lemma \ref{lem:conv_comp_arg} we can say that
$$
U^\varepsilon \left( t,x \right) = U \left( t,x \right) + r^\varepsilon \left( t,x \right),
$$
where
$ r^\varepsilon \to  \; 0, \ 
\text{in } L^2_\loc \left( \R_+, \2 \right).
$ Hence $ r^\varepsilon $ is a perturbative term in the  $ L^2_\loc \left( \R_+, \2 \right) $ topology. As it has been explained in detail in Section \ref{linear problem} we can decompose the (weak) limit $ U $ projecting it onto the non-oscillating and oscillating space
$
U= \bar{U}+ U_\osc,
$
where the orthogonal projection is defined in \eqref{eq:proj_bar}-\eqref{eq:proj_osc}. Since $ U^\varepsilon =\Lplus  V^\varepsilon $ we can hence deduce
$$
V^\varepsilon \left( t,x \right) = \Lminus \bar{U} \left( t,x \right) + \Lminus U_\osc \left( t,x \right) + \Lminus r^\varepsilon \left( t,x \right).
$$
By the definition itself of $ \bar{U} $ we know that $ \bar{U} $ belongs to the kernel of the penalized operator $ \PA $, hence 
$$
\Lminus \bar{U}= \bar{U}.
$$
Moreover the operator $ \mathcal{L} \left( \tau \right), \tau\in \mathbb{R} $ is unitary in $ \2 $, whence the function
$$
R^\varepsilon \left( t \right) = \Lminus r^\varepsilon \left( t \right),
$$
is still an $ o_\varepsilon \left( 1 \right) $ function in the $ L^2_\loc \left( \R_+, \2 \right) $ topology. Hence
$$
V^\varepsilon \left( t,x \right) =  \bar{U} \left( t,x \right) + \Lminus U_\osc \left( t,x \right) +  R^\varepsilon \left( t,x \right),
$$
i.e. $ V^\varepsilon $ is a (high) oscillation around a stationary state $  \bar{U} $ modulated by a $ L^2_\loc \left( \R_+, \2 \right) $ perturbation which tends to zero as $ \varepsilon \to 0 $.

\subsection{Proof of Lemma \ref{lem:limit_horizontal_average}.}
\begin{lemma}\label{lem:limit_quadratic_linear_underline}
The following limits hold, in the sense of distributions
\begin{equation}
\label{eq:def_Qunderline}
\begin{aligned}
\lim_{\varepsilon \to 0}  \left( \int_{\T^2_h} \left( \mathcal{Q}^\varepsilon \left( U^\varepsilon, U^\varepsilon \right) \right)^h \d x_h , 0, 0 \right) = & \; \mathcal{F}_v^{-1} 
\left( 
\sum_{ \left( a,b \right)\in \left\{ 0, \pm \right\}^2}
\sum_{\mathcal{I}_{a,b} \left( n_3 \right)} \ n_3 \ \left( {U}^{a,3} \left( k \right)  \; \hat{U}^{b,h} \left( m \right) \right)
 \right),\\
 = & \; \underline{\mathcal{Q}} \left( U, U \right),
\end{aligned}
\end{equation}
where, fixed $ \left( a,b \right)\in \left\{ 0, \pm \right\}^2 $ the summation set $ \mathcal{I}_{a,b} $ is defined as
\begin{align}\label{eq:Iab}
\mathcal{I}_{a,b} \left( n_3 \right) =
\left\lbrace
 \left(  k, m \right)\in \mathbb{Z}^6 \;
 \left|
  \; k+m= \left( 0,n_3 \right), \; \omega^a \left( k \right) + \omega^b \left( m \right)=0
 \right.
\right\rbrace 
\end{align}
\end{lemma}

\begin{proof}
Let us hence study the distributional limit for $ \varepsilon \to 0 $ of $ \left( \int_{\T^2_h} \left( \mathcal{Q}^\varepsilon \left( U^\varepsilon, U^\varepsilon \right) \right)^h \d x_h , 0, 0 \right) $. Let us  consider a function $ \phi \in \mathcal{D} \left( \R_+\times \T^1_v \right) $ of the form $ \phi = \left( \phi_1, \phi_2, 0, 0 \right) $, and  evaluate
\begin{multline*}
\int_{\R_+}\int_{\T^1_v} \left( \int_{\T^2_h} \left( \mathcal{Q}^\varepsilon \left( U^\varepsilon, U^\varepsilon \right) \right) \d x_h\left( t,x_3 \right)  \right)  \cdot\phi \left( t,x_3 \right) \d x_3 \; \d t
\\
=
\int_{\R_+}\int_{\T^1_v}
\left( \int_{\T^2_h}\Lplus \left[ \Lminus  U^\varepsilon \cdot \nabla \Lminus  U^\varepsilon\right]\d x_h \right) \left( t, x_3 \right) \cdot \phi \left( t,x_3 \right) \d x_3 \d t
\\
=
\int_{\R_+}\int_{\T^1_v}
\left( \int_{\T^2_h} \left[ \Lminus  U^\varepsilon \cdot \nabla \Lminus  U^\varepsilon\right]\d x_h \right) \left( t, x_3 \right) \cdot \phi \left( t,x_3 \right) \d x_3 \d t.
\end{multline*}
In the above equality (and for the rest of the proof) $ A\cdot B $ is the standard scalar product in $ \mathbb{C}^4 $. We underline the fact that, being $ \phi= \left( \phi_1, \phi_2, 0, 0 \right) $ an element of the form $ A\cdot \phi $ has only the horizontal components which give a non-null contribution to the scalar product. 
The last equality is justified by the fact that the adjoint of $ \Lplus $ is $ \Lminus $ and $ \Lminus \phi \left( t,x_3 \right) = \phi \left( t,x_3 \right) $. By use of Placherel theorem we can hence deduce
\begin{multline*}
\int_{\R_+}\int_{\T^1_v} \left( \int_{\T^2_h} \left( \mathcal{Q}^\varepsilon \left( U^\varepsilon, U^\varepsilon \right) \right) \d x_h\left( t,x_3 \right)  \right) \cdot \phi \left( t,x_3 \right) \d x_3 \; \d t
\\
=
\int_{\R_+} \sum_{\substack{n_3\in\mathbb{Z}\\ k+m=\left( 0,n_3 \right) \\ a,b}}
e^{i \frac{t}{\varepsilon} \omega^{a,b}_{k,m}} \  n_3 \ \left(  U^{a,3,\varepsilon}\left( t,k \right)  U^{b,\varepsilon}\left( t,m \right) \right) \cdot \hat{\phi} \left( t, n_3 \right)  \d t.
\end{multline*}
In this case $ \omega^{a,b}_{k,m}= \omega^a \left( k \right)+ \omega^b \left( m \right) $.
Indeed an application of stationary phase theorem a allows us to deduce that
\begin{multline*}
\int_{\R_+} \sum_{\substack{n_3\in\mathbb{Z}\\ k+m=\left( 0,n_3 \right) \\ a,b}}
e^{i \frac{t}{\varepsilon} \omega^{a,b}_{k,m}} \ n_3 \ \left( U^{a,3,\varepsilon}\left( t,k \right) U^{b,\varepsilon}\left( t,m \right) \right) \cdot \hat{\phi} \left( t, n_3 \right) \d t
\\
\xrightarrow{\varepsilon\to 0}
\int_{\R_+} \sum_{\substack{n_3\in\mathbb{Z}\\ k+m=(0,n_3) \\ \omega^a \left( k \right)+ \omega^b \left( m \right)=0}}
 n_3 \  \left( U^{a,3}\left( t,k \right)  U^{b}\left( t,m \right)\right)\cdot \hat{\phi} \left( t, n_3 \right)  \d t,
\end{multline*}
which is exactly the result stated.\\
\end{proof}

\begin{lemma}\label{lem:zero_limit_bilinear_underline}
Let $ \underline{\mathcal{Q}} \left( U, U \right) $ be defined as in \eqref{eq:def_Qunderline}, then
\begin{equation*}
\underline{\mathcal{Q}} \left( U, U \right) =0.
\end{equation*}
\end{lemma}

\begin{proof}
Let us recall that
\begin{align*}
\mathcal{F}_v \underline{\mathcal{Q}} \left( U,U \right) =
\sum_{ \left( a,b \right)\in \left\{ 0, \pm \right\}^2}
\sum_{\mathcal{I}_{a,b} \left( n_3 \right)} \ n_3 \  \left( {U}^{a,3} \left( k \right)  \; \hat{U}^{b,h} \left( m \right) \right),
\end{align*}
hence we shall prove that for any $ \left( a,b \right)\in \left\{ 0, \pm \right\}^2 $ the quantity $ \displaystyle\sum_{n_3, \ \mathcal{I}_{a,b} \left( n_3 \right)} \ n_3 \left( {U}^{a,3} \left( k \right)   \; \hat{U}^{b,h} \left( m \right) \right) $ is null. The summation set $ \mathcal{I}_{a,b} \left( n_3 \right) $ is defined in \eqref{eq:Iab}.
\begin{itemize}
\item We consider at first the case in which $ \left( a,b \right)=\left( 0, 0 \right) $, then the contributions of $ \mathcal{F}_v \underline{\mathcal{Q}} $ restricted on the set $ \left( a,b \right)=\left( 0, 0 \right) $ are 
\begin{equation*}
\sum_{k+m = \left( 0,n_3 \right)}
\ n_3 \ \left( {U}^{0,3} \left( k \right)   \; \hat{U}^{0, h} \left( m \right) \right),
\end{equation*}
but $ U^{0,3}\equiv 0 $ (see \eqref{eigenvectors} and \eqref{eigenvectors_nh=0}) hence this contribution is null.
\item Let us suppose $ \left( a,b \right)=\left( \pm, 0 \right) $, the contributions of \eqref{eq:def_Qunderline} restricted on such set are
\begin{equation*}
\sum_{\substack{ k+m = \left( 0, n_3 \right) \\ \omega^\pm \left( k \right)=0 }} \ n_3 \left( {U}^{\pm,3} \left( k \right)  \; \hat{U}^{0,h} \left( m \right) \right).
\end{equation*}
The condition $ \omega^\pm \left( k \right)=0 $ implies that $ k_h \equiv 0 $, while the condition $ k+m = \left( 0, n_3 \right) $ implies that $ m_h \equiv 0 $, but $ U^{a, 3} \left( 0,k_3 \right)\equiv 0 $ (see \eqref{eigenvectors_nh=0}), whence such term gives a null contribution. The same approach can be applied for the case $ \left( a, b \right)= \left( 0, \pm \right) $.
\item We consider now the case in which $ \left( a, b \right)= \left( \pm, \pm \right) $, the contributions are hence
\begin{equation*}
\sum_{\substack{ k+m = \left( 0, n_3 \right) \\ \omega^\pm \left( k \right) + \omega^\pm \left( m \right)=0 }} {U}^{\pm,3} \left( k \right)  m_3 \; \hat{U}^\pm \left( m \right).
\end{equation*}
Since $ k+m = \left( 0, n_3 \right) $ then $ \left| k_h \right|= \left| m_h \right|=\lambda $. Taking in consideration the constraint $ \omega^\pm \left( k \right) + \omega^\pm \left( m \right)=0 $, which reads as (thanks to the explicit formulation of the eigenvalues in \eqref{eigrnvalue})
\begin{equation*}
\frac{\lambda}{\sqrt{\lambda^2 + k_3^2}} + \frac{\lambda}{\sqrt{\lambda^2 + m_3^2}}=0,
\end{equation*}
which implies that $ \lambda\equiv 0 $. Then we can argue as in the two points above to deduce that such contribution is null.
\item Next we handle the more delicate case in which $ \left( a, b \right) = \left( \pm, \mp \right) $. In this case the contributions are given by
\begin{equation}\label{eq:bil_contr_pm_mp}
\sum_{\substack{ k+m = \left( 0, n_3 \right) \\ \omega^\pm \left( k \right) = \omega^\pm \left( m \right) }} n_3\left( {U}^{\pm,3} \left( k \right)  \; \hat{U}^{\mp, h} \left( m \right) \right),
\end{equation}
where we used implicitly the divergence-free property of the vector $ U^\pm $. The conditions $ k+m = \left( 0, n_3 \right) ,\ \omega^\pm \left( k \right) = \omega^\pm \left( m \right) $ imply now that $ k_h = -m_h $ and
\begin{equation*}
\frac{\lambda}{\sqrt{\lambda^2 + k_3^2}} = \frac{\lambda}{\sqrt{\lambda^2 + m_3^2}},
\end{equation*}
which implies that $ m_3=\pm k_3 $.
\begin{itemize}
\item If $ m_3=-k_3 $ the convolution constraint $ k_3+m_3=n_3 $ in \eqref{eq:bil_contr_pm_mp} implies that $ n_3\equiv 0 $, and hence the contributions in \eqref{eq:bil_contr_pm_mp} arising from this case are nil.
\item In this case $ k_h = -m_h $ and $ k_3=m_3= \frac{n_3}{2} $. Hence we are dealing with an interaction of the form
\begin{equation*}
B^{\pm, \mp}_{n_3}=
\sum_{m_h \in \mathbb{Z}^2} n_3 \left( {U}^{\pm,3} \left( -m_h , \frac{n_3}{2} \right)  \; \hat{U}^{\mp, h} \left( m_h, \frac{n_3}{2} \right) \right), \hspace{1cm} n_3 \in 2 \mathbb{Z}.
\end{equation*}
We shall now reformulate the infinite sum $ B^{+,-}_{n_3}+ B^{-,+}_{n_3} $ in way in which its symmetric properties are explicit.\\
If we consider the element 
\begin{small}
\begin{multline}\label{def_beta}
\beta \left( m_h, n_3 \right) =
  \frac{n_3}{2} \left[ {U}^{+,3} \left( -m_h , \frac{n_3}{2} \right)  \; \hat{U}^{-, h} \left( m_h, \frac{n_3}{2} \right) + {U}^{-,3} \left( -m_h , \frac{n_3}{2} \right)  \; \hat{U}^{+, h} \left( m_h, \frac{n_3}{2} \right)\right.\\
+
\left.
{U}^{+,3} \left( m_h , \frac{n_3}{2} \right)  \; \hat{U}^{-, h} \left( -m_h, \frac{n_3}{2} \right) + {U}^{-,3} \left( m_h , \frac{n_3}{2} \right)  \; \hat{U}^{+, h} \left( -m_h, \frac{n_3}{2} \right)
 \right],
\end{multline}
\end{small}
we can indeed say that
\begin{equation*}
B^{+,-}_{n_3}+ B^{-,+}_{n_3} = \sum _{m_h \in \mathbb{Z}^2}  \beta \left( m_h, n_3 \right).
\end{equation*}
We claim that
\begin{equation} \label{eq:beta=0}
\beta \left( m_h, n_3 \right)= 0 , \hspace{1cm}\forall \ m_h, n_3,
\end{equation}
the proof of \eqref{eq:beta=0} is postponed,
this implies that $ B^{+,-}_{n_3}+ B^{-,+}_{n_3}=0 $, and we finally conclude the proof.
\end{itemize}
\end{itemize}
\end{proof}

\textit{Proof of \eqref{eq:beta=0}.}:
Let us write the elements $ \beta \left( m_h, n_3 \right) $ as 
\begin{equation*}
\beta \left( m_h, n_3 \right)= \beta^\pm \left( m_h, n_3 \right)+ \beta^\mp \left( m_h, n_3 \right), 
\end{equation*}
where
\begin{small}
\begin{align}
\beta^\pm \left( m_h, n_3 \right)= & \ 
\frac{n_3}{2} \left[ {U}^{+,3} \left( -m_h , \frac{n_3}{2} \right)  \; \hat{U}^{-, h} \left( m_h, \frac{n_3}{2} \right) +
{U}^{-,3} \left( m_h , \frac{n_3}{2} \right)  \; \hat{U}^{+, h} \left( -m_h, \frac{n_3}{2} \right)
 \right],\label{eq:betapm}
 \\
 \beta^\mp \left( m_h, n_3 \right)= & \
 \frac{n_3}{2} \left[
 {U}^{-,3} \left( -m_h , \frac{n_3}{2} \right)  \; \hat{U}^{+, h} \left( m_h, \frac{n_3}{2} \right) +
 {U}^{+,3} \left( m_h , \frac{n_3}{2} \right)  \; \hat{U}^{-, h} \left( -m_h, \frac{n_3}{2} \right)
 \right].\nonumber
\end{align}
\end{small}

We shall prove only $ \beta^\pm \left( m_h, n_3 \right) \equiv 0 $ being the proof of $ \beta^\mp \left( m_h, n_3 \right) \equiv 0 $ identical. By definition itself of such elements we know that
\begin{align*}
{U}^{+,3} \left( -m_h , \frac{n_3}{2} \right)  \; \hat{U}^{-, h} \left( m_h, \frac{n_3}{2} \right) = & \ 
\left(\left. \hat{U} \left( -m_h, \frac{n_3}{2} \right) \right| e^+ \left( -m_h, \frac{n_3}{2} \right) \right)_{\mathbb{C}^4} e^{+, 3} \left( -m_h, \frac{n_3}{2} \right)\\
&\ \times\left(\left. \hat{U} \left( m_h, \frac{n_3}{2} \right) \right| e^- \left( m_h, \frac{n_3}{2} \right) \right)_{\mathbb{C}^4} e^{-, h} \left( m_h, \frac{n_3}{2} \right), \\
{U}^{-,3} \left( m_h , \frac{n_3}{2} \right)  \; \hat{U}^{+, h} \left( -m_h, \frac{n_3}{2} \right) = & \
\left(\left. \hat{U} \left( m_h, \frac{n_3}{2} \right) \right| e^- \left( m_h, \frac{n_3}{2} \right) \right)_{\mathbb{C}^4} e^{-, 3} \left( m_h, \frac{n_3}{2} \right)\\
& \ \times
\left(\left. \hat{U} \left( -m_h, \frac{n_3}{2} \right) \right| e^+ \left( -m_h, \frac{n_3}{2} \right) \right)_{\mathbb{C}^4} e^{+, h} \left( -m_h, \frac{n_3}{2} \right).
\end{align*}
By the aid of the explicit definition of the eigenvectors given in \eqref{eigenvectors} we can argue that
\begin{align*}
 e^{-, h} \left( m_h, \frac{n_3}{2} \right)
 = & \
 e^{+, h} \left( -m_h, \frac{n_3}{2} \right) = A^h_{m_h, n_3}  ,\\
 e^{+, 3} \left( -m_h, \frac{n_3}{2} \right)
 = & \
 - e^{-, 3} \left( m_h, \frac{n_3}{2} \right) = A^3_{m_h, n_3}  .
\end{align*}
Whence setting
\begin{equation*}
C_{m_h, n_3}=
\left(\left. \hat{U} \left( -m_h, \frac{n_3}{2} \right) \right| e^+ \left( -m_h, \frac{n_3}{2} \right) \right)_{\mathbb{C}^4} \ \left(\left. \hat{U} \left( m_h, \frac{n_3}{2} \right) \right| e^- \left( m_h, \frac{n_3}{2} \right) \right)_{\mathbb{C}^4},
\end{equation*}
by the aid of the above definitions we hence deduced that
\begin{align*}
{U}^{+,3} \left( -m_h , \frac{n_3}{2} \right)  \; \hat{U}^{-, h} \left( m_h, \frac{n_3}{2} \right) = & \ - C_{m_h, n_3} A^h_{m_h, n_3} A^3_{m_h, n_3},\\
{U}^{-,3} \left( m_h , \frac{n_3}{2} \right)  \; \hat{U}^{+, h} \left( -m_h, \frac{n_3}{2} \right) = & \ C_{m_h, n_3} A^h_{m_h, n_3} A^3_{m_h, n_3}.
\end{align*}
Inserting such relations in the definition of $ \beta^\pm $ \eqref{eq:betapm} we deduce hence that $ \beta^\pm \left( m_h, n_3 \right) \equiv 0 $, concluding. 

\section{The limit.}\label{the limit}

The  result we prove in the present section  is Theorem \ref{thm:simplification_limit_system}, in order to prove it we proceed as follows: we consider separately the evolution of $ U $ distributional solution of \eqref{limit system} onto the non-oscillating and oscillating subspace and prove that respectively $ \bar{U} $ solves \eqref{eq:2DstratifiedNS} and $ U_\osc $ solves \eqref{limit system perturbed BSSQ}. These results are codified respectively in Proposition \ref{prop:limit_kernel_part} and Proposition \ref{prop:limit_osc_part}.\\

\subsection{Derivation of the equation for $ \bar{U} $.}\label{sec:der_eq_kernel}
The procedure we adopt to derive the limit system is pretty straightforward, we mention the works \cite{gallagher_schochet} and \cite{paicu_rotating_fluids} where the authors adopted the same techniques.\\
As we already mentioned in this section we want to deduce the equations satisfied by the projection of $ U $ onto the non-oscillating space $ \mathbb{C}e^0 $. Such projection will be denoted as $ \bar{U}= \left( \uh, 0, 0 \right) $ as it is already mentioned in \eqref{eq:proj_bar}. The result we want to prove is codified in the following proposition

\begin{prop}\label{prop:limit_kernel_part}
Let $ \bar{U}_0= \left( \uh_0, 0, 0 \right)= \left( \nhp \Dh^{-1}\oh, 0, 0 \right) $ be in $ \2 $. The projection of $ U $ distributional solution of \eqref{limit system} onto the non-oscillating space $ \mathbb{C}e^0 $ (see \eqref{eigrnvalue}) defined as
$$
\bar{U}= \mathcal{F}^{-1}\left(  \left( \left. \mathcal{F}U \right| e^0 \right)_{\mathbb{C}^4} e^0 \right),
$$
satisfies the following two-dimensional stratified \NS\ equations with vertical diffusion (in the sense of distributions)
\begin{equation*}
\left\lbrace
\begin{aligned}
& \partial_t \uh + \uh\cdot \nh \uh -\nu \Delta \uh =-\nh \bar{p},\\
& \left. \uh \right|_{t=0}=\uh_0.
\end{aligned}
\right.
\end{equation*}
\end{prop}

We divide the proof in steps:
\begin{itemize}
\item[\textbf{Step 1}] We project the equation \eqref{limit system} onto the non-oscillatory  space generated  by the vector $ e^0 $ defined in \eqref{eigenvectors}. We recall again that such projection is defined as follows (see \eqref{eq:proj_bar} as well): given a vector field $ W $ the orthogonal projection is defined as
\begin{align*}
&\mathcal{F}\;\bar{W}= \left(\left. \hat{W} \right| e^0  \right)_{\mathbb{C}^4}\; e^0,
\end{align*}
with this projection we can derive the evolution equation for the limit flow $ \bar{U} $, i.e.
\begin{align}
&\label{eq:projection_kernel_U}
\left\lbrace
\begin{aligned}
& \partial_t \bar{U}+ \overline{\mathcal{Q}\left( U, U \right)} - \overline{\mathbb{D}\; U}=0,\\
& \left. \bar{U}_0\right|_{t=0}= \bar{U}_0= \bar{V}_0,
\end{aligned}
\right.
\end{align}

\item[\textbf{Step 2}] We prove by mean of a careful analysis that the projection of $ \mathcal{Q} \left( U,U \right) $ onto the non-oscillating subspace $ \mathbb{C}e^0 $, i.e. the element $ \overline{\mathcal{Q}\left( U, U \right)} $ is in fact
$$
\overline{\mathcal{Q}\left( U, U \right)} = \mathcal{B}\left( \bar{U}, \bar{U} \right),
$$
for a suitable bilinear form $ \mathcal{B} $. Hence the projection  onto the kernel of the penalized operator $ \PA $ of all the bilinear interactions is a suitable bilinear interaction of elements of the kernel.

\item[\textbf{Step 3}]  The last step of this section is to prove that
\begin{align*}
 -\overline{\mathbb{D}\; U}= & - \nu \Delta \bar{U}
 \end{align*}
\end{itemize}

We have explained the structure of the present session.

To prove Proposition \ref{prop:limit_kernel_part} it is sufficient to prove Step 1 -- Step 3 mentioned above.\\

To understand the limit of the system \eqref{perturbed BSSQ} means to diagonalize the system \eqref{perturbed BSSQ} in terms of the oscillating and non-oscillating modes introduced in Section \ref{linear problem}. To do so we introduce the following quantities
\begin{equation}\label{vorticities and stream}
\begin{aligned}
\omega^{h, \varepsilon}= & -\partial_2 u^{1,\varepsilon} +\partial_1 u^{2, \varepsilon}; \hspace{2cm}&
\bar{u}^{h, \varepsilon}=& \nhp \Dh^{-1}\omega^{h, \varepsilon};\\
\psi^\varepsilon= & \Dh^{-1} \omega^{h, \varepsilon}; & \tilde{\psi}^\varepsilon= & \Dh^{-1/2} \omega^{h, \varepsilon}.
\end{aligned}
\end{equation}

Step 1 is only a constructive consideration, hence there is nothing to prove.\\

\subsubsection{Proof of Step 2.} The proof of the Step 2 is codified in the following lemma:
\begin{lemma}
Let $ U^\varepsilon\to U $ in $ L^2_{\loc} \left( \R_+; \2 \right) $ as proved in Lemma \ref{lem:conv_comp_arg}, 
the limit of $\left(\left. \mathcal{FQ}^\varepsilon \left( U^\varepsilon, U^\varepsilon \right)\right| e^0\left( n \right)\right)$ as $\varepsilon \to 0$ is $\mathcal{F}\left( \Dh^{-1/2} \left( \bar{u}^h\cdot \nh \omega^h \right)\right)$, in the sense of distributions.
\end{lemma}
\begin{proof}
Let us recall that explicit expression of $\mathcal{FQ}^\varepsilon \left( U^\varepsilon, U^\varepsilon \right)$is given in \eqref{eq:def_limit_bilinear_form}.
As explained before in Lemma \ref{lem:limit_forms_generic} thanks to the stationary phase theorem in the limit as $\varepsilon\to 0$ the only contributions remaining in \eqref{eq:def_limit_bilinear_form} is
\begin{equation}\label{bilinear limit}
\mathcal{FQ} \left( U, U \right)\left( n \right) = \mathbb{P}_n
\sum_{\substack{ \omega^{a,b,c}_{k,m,n}=0\\
k+m=n\\ j=1,2,3}}
\left( \left.  U^{a,j}\left( k \right)m_j U^{b}\left( m \right)
\right| e^c\left( n \right) \right)_{\mathbb{C}^4} \; e^c\left( n \right).
\end{equation}
Now, since $e^0$ is orthogonal to $e^\pm$ as is evident from the definition of the eigenvectors given in \eqref{eigenvectors} we obtain that, projecting on the non-oscillatory potential subspace $ \left( \left. \mathcal{FQ} \left( U, U \right) \right| e^0 \right)_{\mathbb{C}^4}$ 
$$
\left( \left. \mathcal{FQ} \left( U, U \right)\left( n \right) \right| e^0\left( n \right) \right)_{\mathbb{C}^4} = \mathbb{P}_n
\sum_{\substack{ \omega^{a,b,0}_{k,m,n}=0\\
k+m=n\\ j=1,2,3}}
\left( \left.  U^{a,j}\left( k \right)m_j U^{b}\left( m \right) 
\right| e^0\left( n \right) \right)_{\mathbb{C}^4} \; \left| e^0\left( n \right)\right|^2,
$$
whence we can reduce to the case $c=0$.\\

Reading in the Fourier space the projection of the bilinear form it is clear that not all Fourier modes contribute to the bilinear interaction. In the special case that we are considering now in fact the set of bilinear interactions is 
$$
\left\lbrace \left( k,m,n \right)\in \mathbb{Z}^9: \omega^{a,b,0}_{k,m,n}=0, a,b=0,\pm, \;  k+m=n\right\rbrace = \mathcal{R}= \bigcup_{i=0}^3 \mathcal{R}_i,
$$
where
\begin{align*}
\mathcal{R}_0= & \left\lbrace \left( k,m,n \right)\in \mathbb{Z}^9: \omega^{0,0,0}_{k,m,n}=0,  k+m=n\right\rbrace,
\\
\mathcal{R}_1= & \left\lbrace \left( k,m,n \right)\in \mathbb{Z}^9: \omega^{\pm,\pm,0}_{k,m,n}=0,  k+m=n\right\rbrace,
\\
= & \left\lbrace \left( k,m,n \right)\in \mathbb{Z}^9: \pm\frac{\left| k_h \right|}{\left| k \right|}\pm\frac{\left| m_h \right|}{\left| m \right|}=0,  k+m=n\right\rbrace,\\
\mathcal{R}_2= & \left\lbrace \left( k,m,n \right)\in \mathbb{Z}^9:\left( \omega^{\pm,0,0}_{k,m,n}=0 \right) \vee \left(  \omega^{0,\pm,0}_{k,m,n}=0 \right) ,  k+m=n\right\rbrace,\\
= & \left\lbrace \left( k,m,n \right)\in \mathbb{Z}^9:\left(\pm\frac{\left| k_h \right|}{\left| k \right|} =0 \right) \vee \left(  \pm\frac{\left| m_h \right|}{\left| m \right|}=0 \right) ,  k+m=n\right\rbrace,\\
\mathcal{R}_3= & \left\lbrace \left( k,m,n \right)\in \mathbb{Z}^9: \omega^{\pm,\mp,0}_{k,m,n}=0,  k+m=n\right\rbrace,\\
= & \left\lbrace \left( k,m,n \right)\in \mathbb{Z}^9: \pm\frac{\left| k_h \right|}{\left| k \right|} \mp \frac{\left| m_h \right|}{\left| m \right|}=0,  k+m=n\right\rbrace
\end{align*}
Thanks to the above decomposition of the set of bilinear interactions we can assert that
$
\left( \left. \mathcal{FQ} \left( U, U \right) \right| e^0 \right)= \sum_{i=0}^3 \mathcal{B}_i,
$
where
$$
\mathcal{B}_i \left( n \right)
= \mathbb{P}_n
\sum_{\substack{\left( k,m,n \right)\in \mathcal{R}_i\\ j=1,2,3}}
\left( \left.  U^{a,j}\left( k \right)m_j U^{b,j'}\left( m \right)
\right| e^0\left( n \right) \right)\left| e^0\left( n \right)\right|^2.
$$

We start at this point to study the resonance effect on the expression \eqref{bilinear limit}. Indeed the triple $\left( a,b,c \right)=\left( 0,0,0 \right)$ is admissible which determinate the bilinear interaction set $ \mathcal{R}_0 $. Namely $ \mathcal{R}_0 $ describes the set of blinear interactions between element of $ \ker \PA $. The term $ \mathcal{B}_0 $ gives hence a non-null contribution, we want to show that the contributions coming from the other $ \mathcal{B}_i $'s are null. At first let us suppose that $a=b\neq 0$, i.e. we are considering the contributions coming from the term $ \mathcal{B}_1 $ which is defined by the resonant set $ \mathcal{R}_1 $. Let us say $a=b=+$. Whence the resonance condition $\omega^{+,+,0}_{k,m,n}=0$ reads as $\left|k_h\right|=\left|m_h\right|=0$.\\
As it was proved in Section \ref{linear problem} in the case in which $ n_h =0 $ the eigenvalues collapse all to zero, and hence we obtain that
$$
\left\lbrace \left( k,m,n \right)\in \mathbb{Z}^9: \omega^{+,+,0}_{k,m,n},  k+m=n\right\rbrace\subset \mathcal{R}_0.
$$
 The very same analysis can be done for the triplets $\left( -,-,0 \right), \left( \pm,0,0 \right), \left( 0,\pm, 0 \right)$, and hence to prove that $ \mathcal{B}_1= \mathcal{B}_2=0 $.\\
What is left hence at this point is to prove that the triplets $\left( \pm, \mp, 0 \right)$ do not produce any bilinear interaction, or, alternatively, to prove that the contribution coming from $ \mathcal{B}_3 $ is zero.  To do so let us set
\begin{align*}
\hat{U}^a_n = & \left( \left. \hat{U}\left( n \right)\right| e^a\left( n \right) \right),\\
C^{a,b,c}_{k,m,n}= &\sum_{j=1}^3 e^{a,j}\left( k \right)m_j \left( \left. e^b\left( m \right)\right| e^c\left( n \right) \right),
\end{align*}
in particular with this notation the limit form \eqref{bilinear limit} can be written as 
$$
\mathcal{FQ}\left( U,U \right)=
\mathbb{P}_n \sum_{\substack{ \omega^{a,b,c}_{k,m,n}=0\\
k+m=n}}
C^{a,b,c}_{k,m,n} \hat{U}^a_k\hat{U}^b_m e^c\left( n \right).
$$

Let us consider at this point the resonant condition $\omega^\pm\left( k \right)+ \omega^\mp\left( m \right)=0$, it is equivalent, after some algebraic manipulation, considering the explicit expression of the eigenvalues given in \eqref{eigrnvalue} to
\begin{equation}
\label{resonance condition bilinear projection nonosc}
k_3^2 \left| m_h\right|^2 = m_3^2\left|k_h\right|^2.
\end{equation}

Some straightforward computations, using the explicit expression of the eigenvectors given in \eqref{eigenvectors} gives us that
\begin{equation}
\label{def_elements_C}
C^{-,+,0}_{k,m,n}=C^{+,-,0}_{k,m,n}\overset{\text{def}}{=}
\frac{1}{2} C^{\pm,0}_{k,m,n}\in\mathbb{R}.
\end{equation}
Moreover $\hat{U}^{\pm}_n = \pm i \; c\left( n \right)+  d\left( n \right)$, $ c $ and $ d $ are complex-valued and assume the following form
\begin{align*}
c \left( n \right) = &  \frac{n_1 n_3}{\left| n_h \right| \;  \left| n \right|} \hat{U}^1 + 
 \frac{n_2 n_3}{\left| n_h \right| \;  \left| n \right|} \hat{U}^2 
 - \frac{\left| n_h \right|}{\left| n \right|}  \hat{U}^3,\\
 d \left( n  \right) = & \hat{U}^4.
\end{align*}
The $ \hat{U}^i $ above is the $ i $-th component of the Fourier transform of $ U $.
 Hence  we can write
\begin{align*}
\hat{U}^\mp_k\hat{U}^\pm_m = &\; C\left( k,m \right) \pm i D\left( k,m \right),\\
C \left( k,m \right) = &\; c \left( k \right) c\left( m \right) + d\left( k \right) d \left( m \right),\\
D \left( k,m \right) = &\; c \left( k \right) d \left( m \right) - c\left( m \right) d\left( k \right),
\end{align*}
with $C$ symmetric and $D$ skew-symmetric with respect to $k$ and $m$. With these considerations hence
\begin{align}
\sum_{\substack{ \omega^{\pm,\mp,0}_{k,m,n}=0\\
k+m=n}}
C^{\pm,\mp,0}_{k,m,n} \hat{U}^\pm_k\hat{U}^\mp_m =&
\sum_{\substack{
k_3^2 \left| m_h\right|^2 = m_3^2\left|k_h\right|^2  \\
k+m=n} }\left( 
C^{-,+,0}_{k,m,n}\hat{U}^-_k\hat{U}^+_m+
C^{+,-,0}_{k,m,n} \hat{U}^+_k\hat{U}^-_m
 \right),\nonumber\\
 = & 
 \sum_{\substack{
k_3^2 \left| m_h\right|^2 = m_3^2\left|k_h\right|^2\\
k+m=n} } C^{\pm,0}_{k,m,n} C\left( k,m \right).\label{blablabal}
\end{align}
We rely now on the following lemma whose proof is postponed at the end of the present section.
\begin{lemma}\label{skewsym of C}
Under the convolution constraint $k+m=n$ the element $C^{\pm,0}_{k,m,k+m}$ defined in \eqref{def_elements_C}, is skew symmetric with respect to the variables $k,m$.
\end{lemma}
Using at this point Lemma \ref{skewsym of C} it is easy to conclude. If we consider the expression in \eqref{blablabal}, and we remark that the summation set, given by the relation \eqref{resonance condition bilinear projection nonosc}, is symmetric with respect to $k$ and $m$, since $C$ is symmetric and $C^{\pm,0}_{k,m,k+m}$ is skew symmetric we obtain that the sum in \eqref{blablabal} is zero, hence the only admissible triple is $\left( 0,0,0 \right)$.\\
At this point hence all that remains is to fully describe what is the sum
$$
\left( \left. \mathcal{FQ} \left( U, U \right) \right| e^0 \right)= \left(\left. \mathbb{P}_n \sum_{\substack{k+m=n\\ j=1,2,3}} \hat{U}^0_k\hat{U}^0_m e^{0,j}\left( k \right)m_j \left( \left. e^0\left( m \right)\right| e^0\left( n \right) \right)
\;  e^0(n) \right| e^0\left( n \right)  \right)_{\mathbb{C}^4}.
$$
The matrix $ \mathbb{P}_n $ is symmetric and purely real, hence selfadjoint, and the vector $ e^0 $ is divergence-free, this implies that
$$
\left( \left. \mathcal{FQ} \left( U, U \right) \right| e^0 \right)= 
\sum_{\substack{k+m=n\\ j=1,2,3}} \hat{U}^0_k\hat{U}^0_m e^{0,j}\left( k \right)m_j \left( \left. e^0\left( m \right)\right| e^0\left( n \right) \right)_{\mathbb{C}^4}
\; \left|  e^0(n) \right|^2 ,
$$
by our choice of $e^0$ (see \eqref{eigenvectors}) we have that $\left| e^0(n)\right|^2 \equiv 1 $ and a straightforward computation gives us that, considering the relations defined in \eqref{vorticities and stream},
$$
\sum_{j=1}^3 \hat{U}^0_{\bar{k}}\hat{U}^0_{\bar{m}} e^{0,j}\left( \bar{k} \right)\bar{m}_j = \mathcal{F}\left( \bar{u}^h\cdot \nh \tilde{\psi} \right)\left( \bar{n} \right),
$$
where $ \bar{k}+ \bar{m}= \bar{n}$, whence evaluating what $ \left( \left. e^0\left( m \right)\right| e^0\left( n \right) \right)$ is, under the convolution condition $k+m=n$ we obtain
\begin{align*}
 \left( \left. e^0\left( m \right)\right| e^0\left( n \right) \right) = & \frac{1}{\left| n_h\right|\left| m_h\right|}\left(n_1 m_1+n_2m_2\right),\\
 = & \frac{1}{\left| n_h\right|\left| m_h\right|} \left[ \left( m_1^2 + m_2^2 \right) + \left( k_1m_1 + k_2m_2 \right)\right].
\end{align*}
At this point we first apply the operator defined by the symbol $\frac{1}{\left| n_h\right|\left| m_h\right|}  \left( m_1^2 + m_2^2 \right)$ to the element evaluated above $\mathcal{F}\left( \bar{u}^h\cdot \nh \tilde{\psi} \right)\left( {n} \right)$, this gives 
$$
\mathcal{F}\left( \Dh^{-1/2} \left( \bar{u}^h\cdot \nh \omega^h \right)\right)\left( {n} \right),
$$
while computing 
\begin{align*}
\frac{1}{\left| n_h\right|\left| m_h\right|} \left( k_1m_1 + k_2m_2 \right) \mathcal{F}\left( \bar{u}^1 \partial_1 \tilde{\psi} + \bar{u}^2 \partial_2 \tilde{\psi} \right)= &
\frac{1}{\left| n_h\right|} \left( k_1m_1 + k_2m_2 \right) \mathcal{F}\left( \bar{u}^1 \partial_1 {\psi} + \bar{u}^2 \partial_2 {\psi} \right),\\
= &  \frac{1}{\left| n_h\right|}\mathcal{F}
\left( 
\partial_1\bar{u}^1 \partial_1^2 {\psi} + \partial_2\bar{u}^2 \partial_2^2 {\psi}+ 
\partial_2\bar{u}^1 \partial_{1,2}^2 {\psi} + \partial_1\bar{u}^2 \partial_{1,2}^2 {\psi}
 \right),\\
 = & 0,
\end{align*}
where in the last equality we used the relation $\bar{u}^h = \left( \begin{array}{c}
-\partial_2 \psi \\ \partial_1\psi
\end{array} \right)$ already defined in \eqref{vorticities and stream}.\\
Putting together all the results we hence obtained that
$$
\left( \left. \mathcal{FQ} \left( U, U \right) \right| e^0 \right)= \mathcal{F}\left(  \Dh^{-1/2} \left( \bar{u}^h\cdot \nh \omega^h \right)\right),
$$
which concludes the proof of the lemma.
\end{proof}
This concludes the proof of the Step 2, the bilinear interactions of the kernel part are the same as the ones present in the evolution equation for 2d Euler equations in vorticity form.

\textit{Proof of Lemma \ref{skewsym of C}.}\label{proof:antisymmetry_form} We recall that
$$
\frac{1}{2}C^{\pm,0}_{k,m,n} = \sum_{j=1}^3 e^{\pm,j}\left( k \right)m_j \left( \left. e^\pm\left( m \right)\right| e^0\left( n \right) \right),
$$
whence in particular thanks to the explicit expressions of the eigenvectors $e^\pm$ given in \eqref{eigenvectors} and the convolution constrain $k+m=n$ we obtain that
\begin{align*}
-\frac{1}{2}C^{\pm,0}_{k,m,k+m}= & \left( k_1k_3m_1 + k_2k_3m_2 - \left| k_h \right|^2 m_3 \right)
\left( m_2m_3\left( k_1+m_1 \right) -m_1m_3\left( k_2+m_2 \right) \right)\\
=& \left( m_1m_2m_3 k_1^2k_3-k_1k_2k_3 m_1^2m_3\right)
  + \left( k_1k_2k_3 m_2^2m_3 - m_1m_2m_3 k_2^2k_3\right),
\end{align*}
which is indeed skew-symmetric.
\hfill$\Box$

\subsubsection{Proof of Step 3.} It remains to understand how the projection onto the non-oscillating space $ \mathbb{C} e^0 $ affects the second-order linear operator $ \mathbb{D} $ defined in \eqref{matrici}. I.e. we want to prove the Step 3 of the list above. We study the limit as $ \varepsilon \to 0 $ of the second order linear part. The result we prove is the following one 
\begin{lemma} The following limit holds in the sense of distributions
\begin{align*}
\lim_{\varepsilon\to 0} \mathcal{F}^{-1}\left( \left( \left. - \mathcal{F}  \left( \mathbb{D}^\varepsilon U^\varepsilon \right)_n \right| \left|n_h\right| e^0 \left( n \right) \right)_{\mathbb{C}^4} \right)= &-\nu \Delta\omega^h.
\end{align*}
\end{lemma}
\begin{proof}
Let us write explicitly what $ \lim_{\varepsilon\to 0} 
\left( \left. - \mathcal{F}  \left( \mathbb{D}^\varepsilon U^\varepsilon \right)_n \right| \left|n_h\right| e^0 \left( n \right) \right)_{\mathbb{C}^4}  $ is. By the aid of the limit formulation for the second order linear differential operator given in \eqref{eq:def_limit_linear_form} and some computations which can be performed explicitly thanks to the exact formulation of the eigevector $ e^0 $ given in \eqref{eigenvectors} (and recalling that the eigenvectors \eqref{eigenvectors} are orthonormal) we deduce
\begin{align}
\lim_{\varepsilon\to 0} 
\left( \left. - \mathcal{F}  \left( \mathbb{D}^\varepsilon U^\varepsilon \right)_n \right| \left|n_h\right| e^0 \left( n \right) \right)_{\mathbb{C}^4} 
=& \sum_{\omega^{0,b}_n=0} \nu \left| n \right|^2  \left( -n_2 U^{b,1}+n_1U^{b,2} \right).\label{risonanza ellittico QG}
\end{align}
$ \omega^{a,b}_n=\omega^a\left( n \right)-\omega^b\left( n \right) $. Let us consider hence what the interaction condition $ \omega^{0,b}_n=0$ means. If $b=\pm$ than indeed $ \omega^{0,b}_n=0 $ is equivalent to $n_h=0$ since the equation we derive it the following one
$$
\frac{\left| n_h \right|}{\left| n \right|}=0
.
$$
As it has been explained in Section \ref{linear problem} as long as $ n_h =0 $ the eigenvalue corresponding, i.e. $ \omega^b $, it collapses to zero, and hence it belongs to the kernel of the penalized operator. This implies that in \eqref{risonanza ellittico QG} the only nonzero contributions are given if $b=0$, proving Step 3.\\
\end{proof}

With the proof of Step 1--Step 3 above we hence proved that, given an initial $ \oh_0 $, the element
$$
\oh= \curlh \uh,
$$
solves in the sense of distribution the following \NS\ system in vorticity form
\begin{equation}\label{eq:equation_vorticity_2DNS}
\left\lbrace
\begin{aligned}
& \partial_t \oh + \uh \cdot \nh \oh -\nu \Delta \oh =0,
\\
& \left. \oh \right|_{t=0}=\oh_0.
\end{aligned}
\right.
\end{equation}
We hence  apply the 2d-Biot-Savart law
$
\uh= \nhp\Dh^{-1}\oh,
$
to the system \eqref{eq:equation_vorticity_2DNS} to deduce the claim of Proposition \ref{prop:limit_kernel_part}.

\subsection{Derivation of the equation for $ {U}_\osc $.}\label{sec:der_eq_osc}

The result we want to prove in the present section is the following one 

\begin{prop}\label{prop:limit_osc_part}
Let be $ U_{\osc, 0}= V_0-\bar{U}_0\in \2 $. Then the projection of $ U $ distributional solution of \eqref{limit system} onto the oscillating space defined $ \mathbb{C}e^-\oplus \mathbb{C}e^+ $ defined as
$$
U_\osc = \mathcal{F}^{-1} \left( \left( \left. \mathcal{F}U \right| e^- \right)_{\mathbb{C}^4}\;e^- + \left( \left. \mathcal{F}U \right| e^+ \right)_{\mathbb{C}^4}\;e^+ \right)
$$ 
satisfies, for almost all $ \left( a_1, a_2, a_3 \right) \in \R^3 $ parameters defining the three-dimensional periodic domain $ \T^3 $ the linear equation 
\begin{equation*}
\left\lbrace
\begin{array}{l}
\partial_t U_\osc +2\mathcal{Q}\left( \bar{U}, U_\osc \right) - \left( \nu+\nu' \right)\Delta U_\osc =0,\\
\left. U_{\osc}\right|_{t=0}=U_{\osc,0},
\end{array}
\right.
\end{equation*}
where $ \mathcal{Q} $ is defined \eqref{eq:def_limit_bilinear_form}.
\end{prop}

\begin{itemize}
\item[\textbf{Step 1}] We project the equation \eqref{limit system} onto the  oscillatory space generated  by the vectors  $ e^-, e^+ $ defined in \eqref{eigenvectors}. We recall again that such projection is defined as follows (see \eqref{eq:proj_osc} as well): given a vector field $ W $ the orthogonal projection onto the oscillating subspace is defined as
\begin{align*}
&\mathcal{F}\; W_\osc = \left(\left. \hat{W} \right| e^-  \right)_{\mathbb{C}^4}\; e^-+
\left(\left. \hat{W} \right| e^+  \right)_{\mathbb{C}^4}\; e^+,
\end{align*}
with this decomposition we can derive the evolution equation for the limit flow $ U $, i.e.
\begin{align}
&\label{eq:projection_oscillating_U}
\left\lbrace
\begin{aligned}
& \partial_t {U_\osc}+ \pare{\mathcal{Q}\left( U, U \right)}_\osc - \pare{\mathbb{D}\; U}_\osc=0,\\
& \left. {U}_\osc\right|_{t=0}= U_{\osc, 0}= V_{\osc,0}.
\end{aligned}
\right.
\end{align}

\

\item[\textbf{Step 2}] Next we turn our attention to the oscillating part of the bilinear interaction $ \pare{\mathcal{Q}\left( U, U \right)}_\osc  $. We prove that for almost all tori
$$
\pare{\mathcal{Q}\left( U, U \right)}_\osc = 2 {\mathcal{Q}\left( \bar{U}, U_\osc \right)}. 
$$
This result is not a free-deduction and it can be attained only thanks to some geometrical hypothesis on the domain. We say in fact in this case that we consider \textit{non-resonant} domain.\\
A direct consequence is that $ U_\osc $ satisfies hence a \textit{linear equation}, hence it is globally well posed if the perturbation $ \bar{U} $ acting on his evolution system is globally well posed as well.
\item[\textbf{Step 3}]  The last step of this section is to prove that
\begin{align*}
 - \pare{\mathbb{D}\; U}_\osc= & -\left( \nu+\nu' \right)\Delta U_\osc.
\end{align*}
\end{itemize}

As well as in the previous section in order to prove Proposition \ref{prop:limit_osc_part} it i to prove Step 1--Step 3 above.\\

As well as above the Step 1 consists of constructive considerations only, hence there is nothing really to prove.

\subsubsection{Proof of Step 2.}
Our  goal is to study the interaction of the kind $\left( \mathcal{Q}^\varepsilon\left( U^\varepsilon, U^\varepsilon \right) \right)_\osc$, hence to prove the Step 2. These are bilinear interactions between highly oscillating modes, which create a bilinear interaction of the same form of the classical three-dimensional \NS\ equations.  We want to prove  that in the limit as $\varepsilon \to 0$, for almost each torus $\T^3$, interactions between highly oscillating modes vanishes, leaving linear interactions between $ U_{\osc} $ and $ \bar{U} $ only.\\
Since $ U^\varepsilon= \bar{U}^\varepsilon + U^\varepsilon_{\osc} $ it shall hence suffice to prove that
\begin{align}
\lim_{\varepsilon \to 0} \left(\mathcal{Q}^\varepsilon \left( U^\varepsilon_{\osc}, U^\varepsilon_{\osc} \right) \right)_{\osc}=0, \label{eq:lim_osc-osc-osc}\\
\lim_{\varepsilon \to 0} \left(\mathcal{Q}^\varepsilon\left( \bar{U}^\varepsilon, \bar{U}^\varepsilon \right) \right)_{\osc}=0.\label{eq:lim_bar-bar-osc}
\end{align}
We prove \eqref{eq:lim_osc-osc-osc} and \eqref{eq:lim_bar-bar-osc} respectively in Lemma \ref{resonance Uosc Uosc} and Lemma \ref{lem:bil_int_ker_on_osc}.

\begin{lemma}\label{resonance Uosc Uosc}
For almost each torus $\T^3$ 
the following limit holds in the sense of distributions 
$$
\lim_{\varepsilon \to 0} \left( \mathcal{Q}^\varepsilon\left( U^\varepsilon_{\osc}, U^\varepsilon_{\osc} \right) \right)_\osc = 0.
$$
\end{lemma}
\begin{proof}
In the proof of this lemma we shall see how the resonant effects play a fundamental role in the limit of the projection of the bilinear form onto the oscillatory space. In this proof only we will use again the check notation on the Fourier modes since the structure of the torus itself shall play a significant role. First of all we recall that
$$
\left(\mathcal{FQ}^\varepsilon \left( U^\varepsilon, U^\varepsilon \right)\right)_\osc = \left(\mathbb{P}_n\sum_{\substack{ a,b,c\in \left\{ 0,\pm\right\}\\
k+m=n}} e^{i\frac{t}{\varepsilon}\omega^{a,b,c}_{k,m,n}} 
\left( \left.  \sum_{j=1,2,3}U^{a,\varepsilon,j}\left( k \right)m_j U^{b,\varepsilon}\left( m \right) 
\right| e^c\left( n \right) \right)e^c\left( n \right)\right)_\osc,
$$
whence, since $ F_\osc = \left( \left. F \right| e^\pm \right)e^\pm $ and $ e^0\perp e^\pm $ we easily deduce that $c=\pm$. Letting $ \varepsilon\to 0 $ by stationary phase theorem all that remain are interactions of the form,
$$
\mathcal{FQ} \left( U, U \right) = \mathbb{P}_n 
\sum_{\substack{ \omega^{a,b,\pm}_{k,m,n}=0\\
k+m=n\\ j=1,2,3}}
\left( \left.  U^{a,j}\left( k \right)m_j U^{b}\left( m \right) 
\right| e^\pm\left( n \right) \right)e^\pm\left( n \right),
$$
and in particular we focus on the ones which have purely highly oscillating modes interacting, i.e. when $ a=\pm, b=\pm $ (but they may be different the one from the other) and the frequency set of bilinear interaction satisfies the relation
\begin{equation}
\label{resonance 3d-3d}
\frac{\left| \ck_h\right|}{\left| \ck \right|}+\epsilon_1  \frac{\left| \cm_h\right|}{\left| \cm \right|}=
\epsilon_2 \frac{\left| \cn_h\right|}{\left| \cn \right|}\hspace{2cm} \epsilon_1, \epsilon_2=\pm1.
\end{equation}
We want to prove, specifically, that the bilinear interaction restricted on these modes gives a zero contribution for almost all tori.\\
The above relation can be expressed as a polynomial in the variables $\left( \ck, \cm, \cn \right)$ at the cost of long and tedious computations. In particular we shall use the following expansion 
\begin{multline}\label{resonance polinomial}
2\left| \ck_h \right|^2 \left| \cm_h \right|^2 \left( \left| \ck_h \right|^2 + \left| \ck_3 \right|^2 \right) \left( \left| \cm_h \right|^2 + \left| \cm_3 \right|^2 \right) \left( \left| \cn_h \right|^4 + \left| \cn_3 \right|^4+2 \left| \cn_h \right|^2 \cn_3^2 \right)=\\
\begin{aligned}
 &
\left| \ck_h \right|^4 \left( \left| \cm_h \right|^4 + \cm^4_3+2 \left| \cm_h \right|^2 \cm_3^2 \right)\left( \left| \cn_h \right|^4 + \cn^4_3+2 \left| \cn_h \right|^2 \cn_3^2 \right)\\
& +\left|  \cm_h \right|^4 \left( \left| \ck_h \right|^4 + \ck^4_3+2 \left| \ck_h \right|^2 \ck_3^2 \right)\left( \left| \cn_h \right|^4 + \cn^4_3+2 \left| \cn_h \right|^2 \cn_3^2 \right)\\
&+ \left| \cn_h \right|^4 \left( \left| \ck_h \right|^4 + \ck^4_3+2 \left| \ck_h \right|^2 \ck_3^2 \right)\left( \left| \cm_h \right|^4 + \cm^4_3+2 \left| \cm_h \right|^2 \cm_3^2 \right)\\
& -2\left| \ck_h \right|^2 \left| \cn_h \right|^2 \left( \left| \ck_h \right|^2 + \ck^2_3 \right)\left( \left| \cn_h \right|^2 + \cn^2_3 \right)\left( \left| \cm_h \right|^4 + \cm^4_3+2 \left| \cm_h \right|^2 \cm_3^2 \right)\\
& -2\left| \cm_h \right|^2 \left| \cn_h \right|^2 \left( \left| \cm_h \right|^2 + \cm^2_3 \right)\left( \left| \cn_h \right|^2 + \cn^2_3 \right)\left( \left| \ck_h \right|^4 + \ck^4_3+2 \left| \ck_h \right|^2 \ck_3^2 \right).
\end{aligned}
\end{multline}
 We underline the fact that \eqref{resonance 3d-3d} and \eqref{resonance polinomial} are equivalent. The expression in \eqref{resonance polinomial} could be further expanded and refined, but for our purposes the form in \eqref{resonance polinomial} shall be sufficient.\\
We take the expression in \eqref{resonance polinomial} and we evaluate the sum of monomials in the leading order in the variables $\ck_h,\cm_h,\cn_h$, which is 
$$
\check{P}_0 \left( \ck_h, \cm_h \right)= -3 \left| \ck_h \right|^4  \left| \cm_h \right|^4  \left| \cn_h \right|^4,
$$
while the sum of monomial in the leading order for the variables $\ck_3, \cm_3, \cn_3$ is
\begin{multline*}
\check{P}_8 \left( \ck, \cm \right) =\cm_3^4 \cn_3^4 \left| \ck_h \right|^4+
 \ck_3^4 \cn_3^4 \left| \cm_h \right|^4  + \ck_3^4 \cm_3^4  \left| \cn_h \right|^4\\
 -2 \left| \ck_h \right|^2 \left| \cm_h \right|^2 \ck_3^2 \cm_3^2 \cn_3^4 
-2 \left| \ck_h \right|^2 \left| \cn_h \right|^2 \ck_3^2 \cn_3^2\cm_3^4 
-2 \left| \cm_h \right|^2 \left| \cn_h \right|^2 \cm_3^2 \cn_3^2\ck_3^4 .
\end{multline*}
We point out the $\check{P}_8$ is homogeneous of degree 8 in the variables $\ck_3, \cm_3, \cn_3$ while $\check{P}_0$ is homogeneous of degree zero.\\
Since $\check{P}_8\left( \ck,\cm \right)$ is homogeneous of degree 8 we can rewrite is as
$$
\check{P}_8\left( \ck,\cm \right)= \check{P}_8\left( \frac{k}{a},\frac{m}{a} \right) = a_3^{-8}P_8\left( k,m,  a_h \right).
$$
Since $a_1,a_2,a_3$ are parameters of a torus we can indeed consider them different from 0. Moreover 
$$
P_0\left( k,m,  a_h \right)= -3 \left( a_1a_2 \right)^{-12}\left( a_2^2k_1^2+ a_1^2k_2^2 \right)^2\left( a_2^2m_1^2+ a_1^2m_2^2 \right)^2 \left( a_2^2n_1^2+ a_1^2n_2^2 \right)^2=0,
$$
if and only of $k_h$ or $m_h$ or $ n_h= k_h+m_h $ is equal to zero. Let us suppose hence that one of these three conditions is satisfied. We have seen in Section \ref{linear problem} that once we consider (say) $ n_h=0 $ all the eigenvalues collapse to the degenerate case of $ \omega = 0 $ with multiplicity four, whence there is no triple interaction of highly oscillating modes and we can consider $k_h,m_h,n_h\neq 0$.\\
As explained under this condition hence $P_0\left( k,m,  a_h \right)\neq 0$, hence we can rewrite the resonant condition \eqref{resonance polinomial} in the abstract form
\begin{equation}
\label{resonance in abstract form}
P_0\left( k,m,  a_h \right) a_3^8 +\sum_{\alpha =1}^8 P_\alpha\left( k,m,  a_h \right) a_3^{8-\alpha} = 0,
\end{equation}
where we made sure that $P_0\left( k,m,  a_h \right) \neq 0$. Whence fixing $ \left( k,m, a_h \right)\in\mathbb{Z}^6\times \left(\mathbb{R}_+\right)^2 $ we can state that there exists a finite $ a_3\left( k,m, a_h \right) $ solving \eqref{resonance in abstract form}. These elements are finite and unique once we fix a 8-tuple $ \left( k,m, a_h \right) $. At this point hence it is obvious that
$$
a_3\left( \mathbb{Z}^6, a_h \right)= \bigcup_{\left( k,m \right)\in \mathbb{Z}^6}a_3\left( k,m, a_h \right),
$$
has zero measure in $\mathbb{R}$. Whence we proved that outside a null measure set in $\mathbb{R}^3$ there is not bilinear interaction of highly oscillating modes, proving the lemma.
\end{proof}

We turn now our attention to study the limit dynamic as $ \varepsilon\to 0 $ of the projection of the bilinear interactions of elements in the kernel onto the oscillating subspace, i.e. we prove \eqref{eq:lim_bar-bar-osc}.\\

\begin{lemma}\label{lem:bil_int_ker_on_osc}
The following limit
\begin{equation*}
\left( \mathcal{Q}^\varepsilon \left( \bar{U}^\varepsilon, \bar{U}^\varepsilon \right) \right)_{\osc} \xrightarrow{\varepsilon \to 0} 0,
\end{equation*}
holds in the sense of distributions.
\end{lemma}

Lemma \ref{lem:bil_int_ker_on_osc} states that, on the oscillatory subspace in the limit $ \varepsilon\to 0 $ there is no bilinear interaction of elements of the kernel.

\begin{proof}
The element $ \left( \mathcal{Q}^\varepsilon \left( \bar{U}^\varepsilon, \bar{U}^\varepsilon \right) \right)_{\osc} $ reads as, in the Fourier space
\begin{equation*}
\mathcal{F} \left( \mathcal{Q}^\varepsilon \left( \bar{U}^\varepsilon, \bar{U}^\varepsilon \right) \right)_{\osc} = \sum_{\substack{k+m=n \\ j=1,2,3}}
e^{i\frac{t}{\varepsilon} \omega^\pm \left( n \right)}
\left(\left. \mathbb{P}_n {U}^{0, j,  \varepsilon} \left( k \right) m_j \ {U}^{0, \varepsilon} \left( m \right) \right| e^\pm \left( n \right)  \right)_{\mathbb{C}^4} \ e^\pm \left( n \right),
\end{equation*}
letting $ \varepsilon \to 0 $ and applying the stationary phase theorem the limit results to be
\begin{equation*}
\lim_{\varepsilon \to 0}
\mathcal{F} \left( \mathcal{Q}^\varepsilon \left( \bar{U}^\varepsilon, \bar{U}^\varepsilon \right) \right)_{\osc} = \sum_{\substack{k+m=n \\ j=1,2,3\\ \omega^\pm \left( n \right)=0}}
\left(\left. \mathbb{P}_n {U}^{0, j  } \left( k \right) m_j \ {U}^{0 } \left( m \right) \right| e^\pm \left( n \right)  \right)_{\mathbb{C}^4} \ e^\pm \left( n \right).
\end{equation*} 
The condition 
\begin{equation*}
\omega^{\pm} \left( n \right)= \frac{\left| n_h \right|}{\left| n \right|}=0 \ \Rightarrow \ n_h =0,
\end{equation*}
combined with the convolution condition $ k+m=\left( 0, n_3 \right) $ imply that $ m_h=-k_h $. Under this assumption
\begin{align*}
\sum_{j=1,2,3}{U}^{0, j  } \left( k \right) m_j \ {U}^{0 } \left( m \right) = &\  k_3 {U}^{0, 3  } \left( k \right)  \ {U}^{0 } \left( m \right)
+
 {U}^{0, 3  } \left( k \right)  \ m_3 {U}^{0 } \left( m \right)
,
\\
= & \  n_3 {U}^{0, 3  } \left( k \right)  \ {U}^{0 } \left( m \right).
\end{align*}
We deduced hence that
\begin{equation}\label{boh??????}
\lim_{\varepsilon \to 0}
\mathcal{F} \left( \mathcal{Q}^\varepsilon \left( \bar{U}^\varepsilon, \bar{U}^\varepsilon \right) \right)_{\osc} = \sum_{\substack{k+m=n \\ j=1,2,3\\ n_h=0}} \left(\left.  n_3 {U}^{0, 3  } \left( k \right)  \ {U}^{0 } \left( m \right) \right| e^\pm \left( 0,n_3 \right)  \right)_{\mathbb{C}^4} \  e^\pm \left( 0,n_3 \right).
\end{equation}
The term $ U^0\left( m \right)= \left(\left. \hat{U} \left( m \right) \right| e^0 \left( m \right)  \right)_{\mathbb{C}^4} \ e^0 \left( m \right) $ and $ e^0 $ has the first two components only which are nonzero (see \eqref{eigenvectors}), while $ e^\pm \left( 0,n_3 \right)= \left( 0,0,0,1 \right) $ as it is given in \eqref{eigenvectors_nh=0}, hence the contribution in \eqref{boh??????} is null, concluding.
\end{proof}

\subsubsection{Proof of Step 3.} 
It remains hence only to prove the Step 3 above, i.e. to understand the (distributional) limit as $ \varepsilon\to 0 $ of the interaction generated by the second-order elliptic operator $ \mathbb{D} $ defined in \eqref{matrici}. This is done in the following lemma:

\begin{lemma} The following limit holds in the sense of distributions
\begin{align*}
\lim_{\varepsilon\to 0} \left( \left. -\mathbb{D}^\varepsilon U^\varepsilon \right|  e^-+ e^+ \right)=& - \left( \nu+\nu' \right)\Delta U_\osc.
\end{align*}
\end{lemma}
\begin{proof}
We proceed as follows. By  definition of the projection onto the oscillatory space (see \eqref{eq:proj_osc} $ \left( - \mathbb{D}^\varepsilon U^\varepsilon \right)_\osc $ is given by the formula \eqref{eq:def_limit_linear_form}
\begin{equation}
\label{eq:E0}
\begin{aligned}
\mathcal{F}\left( - \mathbb{D}^\varepsilon U^\varepsilon \right)_\osc \left( n \right) = & \sum_{\omega^{a,\pm}_n=0} \left( \left. -\mathbb{D}^\varepsilon\left( n \right) \hat{U}^\varepsilon \left( n \right) \right| e^\pm \left( n \right) \right)_{\mathbb{C}^4}\; e^\pm\left( n \right),
\end{aligned}
\end{equation} 
where for the second equality we used the decomposition $ \hat{U}^\varepsilon= \sum_{a=0,\pm} \hat{U}^{a, \varepsilon} \; e^a $ and the fact that the eigenvectors are orthogonal.\\
Computing the explicit expression of $ \left( \left. -\mathbb{D}^\varepsilon\left( n \right)  e^+\left( n \right) \right| e^+\left( n \right) \right)_{\mathbb{C}^4} $ we deduce
\begin{equation}
\label{eq:E1}
\begin{aligned}
 \left( \left. -\mathbb{D}^\varepsilon\left( n \right)  e^+ \left( n \right) \right| e^+\left( n \right) \right)_{\mathbb{C}^4} = & \left( \left. -\mathbb{D}\left( n \right) e^{-i\frac{t}{\varepsilon}\omega\left( n \right)} \;  e^+\left( n \right) \right| \; e^{-i\frac{t}{\varepsilon}\omega\left( n \right)} \; e^+\left( n \right) \right)_{\mathbb{C}^4},\\
 = & \left( \left. -\mathbb{D} \left( n \right)  e^+\left( n \right) \right| e^+\left( n \right) \right)_{\mathbb{C}^4},\\
 = & \left( \nu+ \nu' \right)\; \left| n \right|^2.
\end{aligned}
\end{equation}
While for the element $ \left( \left. -\mathbb{D}^\varepsilon\left( n \right)  e^- \left( n \right)\right| e^+\left( n \right) \right)_{\mathbb{C}^4} $
\begin{align}\label{eq:E2}
\left( \left. -\mathbb{D}^\varepsilon\left( n \right)  e^- \left( n \right) \right| e^+ \left( n \right) \right)_{\mathbb{C}^4} = & \; e^{2i \; \frac{t}{\varepsilon}\omega\left( n \right)}
\left( \left. -\mathbb{D} \left( n \right)  e^-\left( n \right) \right| e^+\left( n \right) \right)_{\mathbb{C}^4} \to 0,
\end{align}
in the sense of distributions thanks to the stationary phase theorem. In this case we automatically excluded the case $ \omega \left( n \right)=0 $ since, as explained in Section \ref{linear problem}, saying $ \omega \left( n \right)=0 $  is equivalent to say that $ n_h =0 $ and hence, in this case, all eigenvectors belong to the kernel of the penalized operator and hence $ \left. \mathcal{Q} \left( U, U \right)_\osc\right|_{n_h=0}=0 $.   \\
The same ideas can be applied to deduce
\begin{align}
\left( \left. -\mathbb{D}^\varepsilon\left( n \right)  e^-\left( n \right) \right| e^-\left( n \right) \right)_{\mathbb{C}^4} = & \left( \nu+ \nu' \right)\; \left| n \right|^2,\label{eq:E3}\\
\left( \left. -\mathbb{D}^\varepsilon\left( n \right)  e^+\left( n \right) \right| e^-\left( n \right) \right)_{\mathbb{C}^4} \to &\; 0\label{eq:E4}.
\end{align}
The limit \eqref{eq:E4} has to be understood in the sense of distributions. 
Inserting \eqref{eq:E1}--\eqref{eq:E4} into \eqref{eq:E0} we deduce the claim, proving the Step 3.
\end{proof}

\section{Global existence of the limit system.}\label{global_existence}
In Section \ref{sec:der_eq_kernel} and \ref{sec:der_eq_osc} we performed a careful analysis whose goal was to understand which equations are solved (in the sense of distributions) by the functions $ \bar{U} $ and $ U_\osc $ which were defined as the projection respectively onto the non-oscillating subspace $ \mathbb{C}e^0 $ and the oscillating space $ \mathbb{C}e^-\oplus \mathbb{C}e^+ $ of $ U $, distributional solution of \eqref{limit system}. The present section is devoted to study the propagation of strong (Sobolev) norms under the assumption that the initial data is sufficiently regular.\\
In particular we are interested to understand if the system \eqref{limit system} propagates (isotropic) Sobolev data $ \Hs, s>1/2 $ and, if so, under which conditions on the initial data. Our expectation in that such system can propagate sub-critical Sobolev regularity globally-in-time without any particular smallness assumption on the initial data. The results we prove are the following ones:

\begin{prop}
\label{propagation_isotropic_Sobolev_regularity_ubarh}
Let $ U_0 \in\Hs\cap L^\infty \left( \T_v; H^\sigma \left( \T^2_h \right) \right)$, and $ \nh U \in L^\infty \left( \T_v; H^\sigma \left( \T^2_h \right) \right) $ for $s>1/2, \sigma >0$, and let $ U $ be of zero horizontal average, i.e.
\begin{align*}
\int_{\T^2_h} U_0\left( x_h,x_3 \right)\dx_h=0 & \text{ for each } x_3\in \T^1_v,
\end{align*}
 then the weak solution of 
\begin{equation}\label{2D_stratified_NS_no_pressure}
\left\lbrace
\begin{aligned}
&\partial_t \bar{u}^h +  \bar{u}^h \cdot\nh \bar{u}^h  - \nu\Delta \bar{u}^h=0,\\
&\dive_h \bar{u}^h= -\nh \bar{p},\\
& U\left( 0,x \right)= U_0
\end{aligned}
\right.
\end{equation}
is in fact strong, and has the following regularity:
$$
\bar{u}^h \in  \mathcal{C}\left( \mathbb{R}_+;\Hs \right) \cap L^2\left( \mathbb{R}_+; {H}^{s+1}\left( \T^3 \right) \right) .
$$
Moreover for each $ t>0 $ the following estimate holds true
\begin{multline}\label{eq:stong_Hs_bound_ubar}
\left\| \uh \left( t \right)\right\|_\Hs^2 + \nu \int_0^t \left\| \uh\left( \tau \right) \right\|_{H^{s+1}\left( \T^3 \right)}^2\d\tau
 \leqslant C \left\| \uh _0\right\|_\Hs^2
 \exp\set{
\frac{C  {K}}{c\nu} \;\Phi \left( U_0 \right)
 \left\| \nh \uh_0 \right\|_{L^p_v \left( H^\sigma_h \right)}}.
\end{multline}
where
\begin{equation}\label{eq:def_Phi_U0}
\Phi \left( U_0 \right)=
\;\exp \set{\frac{ C K^2
 \left\| \nh\uh_0 \right\|_{L^\infty_v\left( L^2_h \right)}^2 }{c\nu}
\exp \set{ \frac{K}{c\nu} \left( 1+ \left\| \uh_0 \right\|_{L^\infty_v \left( L^2_h \right)}^2 \right) \left\| \nh\uh_0 \right\|_{L^\infty_v \left( L^2_h \right)}^2 }.
}
\end{equation}
\end{prop}

\begin{prop}\label{regularity_osc_part}

If $ U_{\osc,0}\in \Hs, s>1/2 $ then 
 $ U_\osc $, weak solution of 
 
 \begin{equation*}
\left\lbrace
\begin{array}{l}
\partial_t U_\osc +2\mathcal{Q}\left( \bar{U}, U_\osc \right) - \left( \nu+\nu' \right)\Delta U_\osc =0,\\
\left. U_{\osc}\right|_{t=0}=U_{\osc,0}.
\end{array}
\right.
\end{equation*}
 is global-in-time and belongs to the space 
 $$
U_\osc\in  \mathcal{C}\left( \mathbb{R}_+; H^s\left( \T^3 \right) \right) \cap L^2\left( \mathbb{R}_+; H^{s+1}\left( \T^3 \right) \right),
$$
for $ s>1/2 $. For each $ t> 0  $ the following bound holds true
\begin{multline}\label{eq:stong_Hs_bound_uosc}
\left\| U_\osc \left( t \right) \right\|_{\Hs}^2 + 
\frac{\nu+\nu'}{2} \int_0^t \left\| U_\osc \left( \tau \right) \right\|^2_{H^{s+1}\left( \T^3 \right)}\d\tau\\
\leqslant  C \left\| U_{\osc,0} \right\|^2_\Hs \exp\left\{C \left\|\nabla \bar{u}^h \right\|^2_{ L^2\left( \mathbb{R}_+; H^{s}\left( \T^3 \right) \right)}\right\}
.
\end{multline}
\end{prop}
For a proof for Proposition \ref{regularity_osc_part} we refer to \cite[Appendix B]{gallagher_schochet}.
\\

Thanks to the bounds above we can hence claim that, if $ U_0 \in\Hs\cap L^\infty \left( \T_v; H^s \left( \T^2_h \right) \right)$, and $ \nh U \in L^\infty \left( \T_v; H^s \left( \T^2_h \right) \right) $ for $s>1/2$, and let $ U $ be of zero horizontal average, then 
$$
U= \bar{U}+ U_\osc,
$$
distributional solution of \eqref{limit system} is in fact global-in-time and belongs to the space 
$$
U\in  \mathcal{C}\left( \mathbb{R}_+; H^s\left( \T^3 \right) \right) \cap L^2\left( \mathbb{R}_+; H^{s+1}\left( \T^3 \right) \right),
$$
for $ s>1/2 $.

\subsection{The kernel part: propagation of $ \Hs, s>1/2 $ data.}

This section is devoted to the proof of Proposition \ref{propagation_isotropic_Sobolev_regularity_ubarh}.\\
The equation \eqref{eq:equation_vorticity_2DNS} is the vorticity equation associated to $ \bar{u}^h $, which solves distributionally \eqref{eq:2DstratifiedNS}.
Hence $ \bar{u}^h $ satisfies a stratified 2D Euler equation with full diffusion and the  2D Biot-Savart $\bar{u}^h= \nhp\Dh^{-1}\omega^h
$  law holds.

\begin{lemma}
Let $ \uh_0 \in \Hs, s>1/2 $ and of zero horizontal average i.e.
$$
\int_{\T^2_h} \uh_0 \left( y_h, x_3 \right) \d y_h =0.
$$
 The function $ \uh $ local solution of \eqref{eq:2DstratifiedNS} defined in the space 
 $$
 \uh \in \mathcal{C}\left( \left[0, T^\star\right]; \Hs \right)\cap L^2 \left( \left[0, T^\star\right]; H^{s+1} \left( \T^3 \right) \right),
 $$
 for some $ T^\star >0 $ is of zero horizontal average in its lifespan, i.e.
 $$
 \int_{\T^2_h} \uh \left(t, y_h, x_3 \right) \d y_h =0,
 $$
for each $ 0<t<T^\star $.
\end{lemma}

\begin{rem}\label{remark_equivalence_horizontal_Sobolev_spaces}
The above lemma in particular implies that, for local solutions of equation \eqref{eq:2DstratifiedNS}, the horizontal homogeneous and nonhommogeneous Sobolev spaces are equivalent i.e. 
$$
\left\| \left( -\Dh \right)^{s/2} u \left( \cdot, x_3 \right) \right\|_{L^2_h} \sim \left\| u \left( \cdot, x_3 \right) \right\|_{H^s_h}.
$$
 For this reason, from now on, we shall always use the nonhomogeneous Sobolev space (although, as explained, for equation \eqref{eq:2DstratifiedNS} they are equivalent) since the embedding $ H^{1+\varepsilon}\left( \T^2_h \right)\hra L^\infty \left( \T^2_h \right) , \varepsilon > 0 $ holds true (which is not the case with homogeneous spaces, generally) and we do not leave any place to  ambiguity.
\fine
\end{rem}

\subsubsection{The kernel part : smoothing effects oh the heat flow.}\label{sec:smoothing_heat}

This first subsection is aimed to prove some global-in-time integrability results  for some suitable norms for (weak) solutions of the limit system \eqref{eq:2DstratifiedNS}. The result we present here are a consequence of the fact that \eqref{eq:2DstratifiedNS} is a transport-diffusion equation in the horizontal variables, but a purely diffusion equation in the vertical one, in the sense that there is no vertical transport contribution.\\

The final result we want to prove is the following one
\begin{prop}\label{prop:Linfty_integrability_uh}
Let $ \uh $ be a weak solution of \eqref{eq:2DstratifiedNS}, and assume that $ \uh_0, \nh \uh_0\in L^\infty_v \left( H^\sigma_h \right) $. Let the inital data be of zero horizontal average, i.e.
$$
\int_{\T^2_h} \uh_0 \left( y_h, x_3 \right) \d y_h =0,
$$
for each $ x_3\in \T^1_v $. Then the solution $ \uh $ belongs to the space
$$
\uh \in L^2 \left( \R_+; \Linfty \right),
$$
and in particular
\begin{equation*}
\left\| \uh \right\|_{L^2 \left( \R_+; \Linfty \right)}
  \leqslant
\frac{C  {K}}{c\nu} \;\Phi \left( U_0 \right)
 \left\| \nh \uh_0 \right\|_{L^p_v \left( H^\sigma_h \right)},
\end{equation*}
where $c, C,K $ are constants which do not depend by any parameter of the problem and $ \Phi \left( U_0 \right) $ is defined in \eqref{eq:def_Phi_U0}.
\end{prop}

The tools required in order to prove  Proposition \ref{prop:Linfty_integrability_uh} are rather easy, but the procedure adopted is slightly involved, for this reason we decide to outline the structure of the proof in the following lines:
\begin{enumerate}
\item \label{L1} Using the fact that the transport effects in \eqref{eq:2DstratifiedNS} are horizontal only we perform an $L^2$-energy estimate in the horizontal direction. Next, on the vertical direction we exploit the fact that \eqref{eq:2DstratifiedNS} is purely diffusive and linear equation: this fact allows us to use the smoothing effects of the heat kernel (at least along the $x_3$-direction) in order to prove that
\begin{align*}
\uh \in L^2 \pare{ \R_+ ; L^\infty_v \pare{ L^2_h } },\\
\nh \uh \in L^2 \pare{ \R_+ ; L^\infty_v \pare{ L^2_h } }.
\end{align*}
\item We use the results of the point \ref{L1} in order improve the  regularity result to the following statement (at a cost of having smoother initial data):
\begin{align*}
\uh \in L^2 \pare{ \R_+ ; L^\infty_v \pare{ H^{\sigma}_h } },\\
\nh \uh \in L^2 \pare{ \R_+ ; L^\infty_v \pare{ H^{\sigma}_h } }.
\end{align*}
for $\sigma >0$.

\item Since the equation \eqref{eq:2DstratifiedNS} propagates the horizontal average we exploit the embedding $L^\infty_v \pare{ H^{1+\sigma}_h }\hra L^\infty \pare{ \T^3 }$ to deduce the result.
\end{enumerate}

The following Poincar\'e inequality shall be crucial in the proof of time-smoothing effects we want to prove

\begin{lemma}
\label{lem:poincare}
Let $ f\in W^{1,2} \left( \left[0, 2\pi \; a_1\right] \times \left[0, 2\pi \; a_2\right] \right) $ and such that it zero average, i.e.
$$
\int_0^{ 2\pi \; a_1} \int_0^{ 2\pi \; a_2} f \left( x_1, x_2 \right)\d x_2 \d x_1=0.
$$
 Then the following inequality holds true
$$
\left\| f \right\|_{ L^2 \left( \left[0, 2\pi \; a_1\right] \times \left[0, 2\pi \; a_2\right] \right)} \leqslant C \left\| \nabla f \right\|_{ L^2 \left( \left[0, 2\pi \; a_1\right] \times \left[0, 2\pi \; a_2\right] \right)},
$$
where in particular the constant $ C $ is independent of the parameters $ a_1, a_2 $ characterizing the torus $  \left[0, 2\pi \; a_1\right] \times \left[0, 2\pi \; a_2\right] $.
\end{lemma}

The following lemma is a key step for the rest of the results presented in the present paper

\begin{lemma}\label{lem:key_lemma}
Let $ \bar{u}^h $ be a (weak) solution of the equation \eqref{eq:2DstratifiedNS}. Let us suppose moreover that $ u_0, \nh u_0 \in L^p_v \left( L^2_h \right) $ for some $ p\in \left[ 2 , \infty\right] $. Let us assume as well that
$$
\int_{\T^2_h} \bar{u}^h_0 \left( y_h, x_3 \right)\d y_h=0,
$$
for each $ x_3\in \T^1_v $. Then
\begin{align}
\bar{u}^h \in & L^q \left( \R_+; L^p_v \left( L^2_h \right) \right), &
\text{for } q \in \left[ 1 , \infty\right], \; p\in \left[ 2, \infty\right],\\
\nh\bar{u}^h \in & L^q \left( \R_+; L^p_v \left( L^2_h \right) \right), &
\text{for } q \in \left[ 1 , \infty\right], \; p\in \left[ 2, \infty\right].
\end{align}
In particular the time-decay rate is exponential, i.e.
\begin{align*}
\left\| \bar{u}^h \left( t \right) \right\|_{L^p_v \left( L^2_h \right)} \leqslant &\; e^{-\nu c \;t} \left\| \bar{u}^h_0 \right\|_{L^p_v \left( L^2_h \right)},\\
\left\| \nh\bar{u}^h \left( t \right) \right\|_{L^p_v \left( L^2_h \right)} \leqslant & \; {K} \; e^{-\nu c \;t} \left\| \nh\bar{u}^h_0 \right\|_{L^p_v \left( L^2_h \right)},
\end{align*}
where $ c,K $ are strictly positive constants which depend on the dimension of the horizontal domain only (in this case two).
\end{lemma}
\begin{proof}
Let us multiply the equation \eqref{eq:2DstratifiedNS} for $ \bar{u}^h $ and let us take $ L^2_h $ scalar product. Since the vector field $ \bar{u}^h $ is horizontal-divergence-free, i.e. $ \partial_1 \bar{u}^1 \left( x_1,x_2, x_3 \right) + \partial_2 \bar{u}^2 \left( x_1,x_2, x_3 \right)=0 $ for each $ x\in \T^3 $ we deduce the following normed equality
$$
\frac{1}{2}\; \frac{\d}{\d t} \left\| \uh \left( x_3 \right) \right\|_{L^2_h}^2 +\nu \left\| \nh \uh \left( x_3 \right) \right\|_{L^2_h}^2 + \nu \left\| \partial_3 \uh \left( x_3 \right) \right\|_{L^2_h}^2 - \nu \partial_3^2 \left\| \uh \left( x_3 \right) \right\|_{L^2_h}^2 =0.
$$
The term $ \nu \left\| \partial_3 \uh \left( x_3 \right) \right\|_{L^2_h}^2 $ has indeed a positive contribution, hence we deduce the following inequality
$$
\frac{1}{2}\; \frac{\d}{\d t} \left\| \uh \left( x_3 \right) \right\|_{L^2_h}^2 +\nu \left\| \nh \uh \left( x_3 \right) \right\|_{L^2_h}^2  - \nu \partial_3^2 \left\| \uh \left( x_3 \right) \right\|_{L^2_h}^2 \leqslant 0.
$$
At the same time we can use the Poincar\'e inequality as stated in Lemma \ref{lem:poincare} to argue that
$$
\nu \left\| \nh \uh \left( x_3 \right) \right\|_{L^2_h}^2 \geqslant c \nu \left\|  \uh \left( x_3 \right) \right\|_{L^2_h}^2,
$$
where $ c=C^{-1} $ appearing in Lemma \ref{lem:poincare}. Whence we deduced the inequality
\begin{equation}\label{eq:stima_Lpv_kernel_part1}
\frac{1}{2}\; \frac{\d}{\d t} \left\| \uh \left( x_3 \right) \right\|_{L^2_h}^2 +c\nu \left\|  \uh \left( x_3 \right) \right\|_{L^2_h}^2  - \nu \partial_3^2 \left\| \uh \left( x_3 \right) \right\|_{L^2_h}^2 \leqslant 0.
\end{equation}

Let us now consider a $ p\in \left[2 , \infty \right) $, and let us multiply \eqref{eq:stima_Lpv_kernel_part1} by 
\begin{equation*}
\left\| \uh \left( x_3 \right) \right\|_{L^2_h}^{\left( p-2 \right)} = \left( \int_{\T^2_h} \uh \left( y_h, x_3 \right)^2 \d y_h \right)^{\frac{p-2}{2}},
\end{equation*}
and hence integrate the resulting inequality with respect to $ x_3\in \T^1_v $. The resulting inequality we deduce is
\begin{equation*}
\frac{1}{p} \ \frac{\d}{\d t} \left\| \uh \right\|_{L^p_v \left( L^2_h \right)}^p + c\nu  \left\| \uh \right\|_{L^p_v \left( L^2_h \right)}^p
+ \underbrace{\frac{8 \left( p-2 \right)}{p^2} \int_{\T^1_v} \left[ \partial_3
 \pare{ \int_{\T^2_h} u^2 \d x_h }^{\frac{p}{4}} \right]^2 \d x_3}_{= I_p\pare{ u }} \leqslant 0,
\end{equation*}
and since $I_p\pare{ u } \geqslant 0$ for each $p$ we deduce the following inequality neglecting it
\begin{equation*}
 \frac{\d}{\d t} \left( \left( e^{c\nu \ t} \left\| \uh \right\|_{L^p_v \left( L^2_h \right)} \right)^p \right)\leqslant 0.
\end{equation*}

\noindent Integrating in-time the above equation we deduce hence that.

\begin{align}\label{eq:stima_Lpv_kernel_part2}
\left\| \uh \left( t \right) \right\|_{L^p_v \left( L^2_h \right)} \leqslant \;  e^{-c\nu\;t} \left\| \uh_0 \right\|_{L^p_v \left( L^2_h \right)} ,
\end{align}
and hence $ \uh $ is $ L^q $-in-time for each $ p\in \left[2 , \infty \right) $. In order to lift the result when $ p=\infty $ it suffice to recall that, given a finite measure space $ \pare{ \mathcal{X}, \Sigma, \mu } $ and a
 $ \phi \in L^p \pare{ \mathcal{X}, \Sigma, \mu } $ for each $ p\in \left[1, \infty\right] $, the application $ p\mapsto \left| \mathcal{X} \right|^{-1} \left\| \phi \right\|_{L^p \pare{ \mathcal{X}, \Sigma, \mu }} $ is continuous, increasing in $ p $ and converges to $ \left\| \phi \right\|_{L^\infty \left( \mathcal{X} \right)} $ as $ p\to \infty $, hence it suffice to consider the limit for $ p\to\infty $ in \eqref{eq:stima_Lpv_kernel_part2}. \\

To prove the statement for $ \nh\uh $ let us consider the equation satisfied by $ \oh= \curlh \uh $. The equation is the following one
\begin{equation*}
\left\lbrace
\begin{aligned}
& \partial_t \oh + \uh \cdot \nh \oh -\nu \Delta \oh=0,\\
 & \left. \oh \right|_{t=0}=\oh_0=\curlh \uh_0.
\end{aligned}
\right.
\end{equation*}
We can perform the exactly same procedure as it has been done with $ \uh $, obtaining hence that
\begin{align*}
\left\| \oh \left( t \right) \right\|_{L^p_v \left( L^2_h \right)} \leqslant \;  e^{-c\nu\;t} \left\| \oh_0 \right\|_{L^p_v \left( L^2_h \right)} .
\end{align*}
Since the application $ \oh \mapsto \nh \uh  $ is a Calder\`on-Zygmund operator it maps continuously $ L^2_h $ to itself and has operator norm $ K $ we deduce that
\begin{align*}
\left\| \nh \uh \left( t \right) \right\|_{L^p_v \left( L^2_h \right)} \leqslant \; {K} \;  e^{-c\nu\;t} \left\| \nh \uh_0 \right\|_{L^p_v \left( L^2_h \right)} ,
\end{align*}
for each $ p\in \left[2 , \infty\right] $.
\end{proof}

Lemma \ref{lem:key_lemma} deals hence with the propagation of some anisotropic $ L^p_v \left( L^2_h \right) $ regularity for (weak) solutions of equation \eqref{eq:2DstratifiedNS}. In our context we are particularly interested to study the propagation of  the anisotropic   $ L^\infty_v \left( L^2_h \right) $ norm.\\
Similarly we are interested to understand how \eqref{eq:2DstratifiedNS} propagates data which are bounded in the horizontal variables. Standard theory of two-dimensional \NS\ and Euler equations  suggests that, if the data is sufficiently  regular in terms of Sobolev regularity, the propagation of horizontal norms should not be problematic. \\
The regularity statements proved until now are not sufficient to perform our analysis, for this reason we require the following lemma

\begin{lemma}\label{lem:key_lemma2}
Let us consider $ \uh $ a (weak) solution of \eqref{eq:2DstratifiedNS}, with initial data $ \uh_0, \nh \uh_0 \in L^\infty_v \left( H^\sigma_h \right), \sigma \geqslant 0 $ and assume $ \uh_0 $ has zero horizontal average, then
\begin{align*}
\bar{u}^h \in & L^q \left( \R_+; L^p_v \left( H^\sigma_h \right) \right), &
\text{for } q \in \left[ 1 , \infty\right], \; p\in \left[ 2, \infty\right],\\
\nh\bar{u}^h \in & L^q \left( \R_+; L^p_v \left(  H^\sigma_h  \right) \right), &
\text{for } q \in \left[ 1 , \infty\right], \; p\in \left[ 2, \infty\right].
\end{align*}
Moreover the decay rate of the $  L^p_v \left( H^\sigma_h \right) $ norms is exponential-in-time, in particular the following bounds hold
\begin{align}
& \label{eq:exponential_decay_LpvHsigma_uh}\left\|  \uh \left( t \right) \right\|_{L^p_v \left( H^\sigma_h \right)} \leqslant \; C\;
\exp \set{ \frac{K}{c\nu} \left( 1+ \left\| \uh_0 \right\|_{L^\infty_v \left( L^2_h \right)}^2 \right) \left\| \nh\uh_0 \right\|_{L^\infty_v \left( L^2_h \right)}^2 }
   e^{-\frac{c \nu}{2} \; t}\left\|  \uh_0 \right\|_{L^p_v\left( H^\sigma_h \right)},\\
& \label{eq:exponential_decay_LpvHsigma_nhuh}  
   \left\| \nh \uh \left( t \right) \right\|_{L^p_v \left( H^\sigma_h \right)}
\leqslant
C  {K} \;\Phi \left( U_0 \right)
 e^{-\frac{c\nu}{2}\; t} 
 \left\| \nh \uh_0 \right\|_{L^p_v \left( H^\sigma_h \right)},
\end{align}
where $ \Phi \left( U_0 \right) $ is defined in \eqref{eq:def_Phi_U0}.
\end{lemma}

\begin{proof}
We prove at first \eqref{eq:exponential_decay_LpvHsigma_uh}.\\
Let us recall the bound 
\begin{multline}\label{eq:exp_decay_Hsigma_norms_uh}
\left( \left. \uh \left( \cdot, x_3 \right)\cdot \nh \uh \left( \cdot, x_3 \right) \right| \uh \left( \cdot, x_3 \right) \right)_{H^\sigma_h} \\
\begin{aligned}
&\leqslant C \left( 1+ \left\| \uh \left( \cdot, x_3 \right) \right\|_{L^2_h}  \right) \left\| \nh \uh \left( \cdot, x_3 \right) \right\|_{L^2_h}
\left\| \uh \left( \cdot, x_3 \right) \right\|_{H^\sigma_h}
\left\|\nh \uh \left( \cdot, x_3 \right) \right\|_{H^\sigma_h},\\
&\leqslant C \left( 1+ \left\| \uh \left( \cdot, x_3 \right) \right\|_{L^2_h}^2  \right) \left\| \nh \uh \left( \cdot, x_3 \right) \right\|_{L^2_h}^2
\left\| \uh \left( \cdot, x_3 \right) \right\|_{H^\sigma_h}^2 + \frac{\nu}{2}\left\| \nh \uh \left( \cdot, x_3 \right) \right\|_{H^\sigma_h}^2,\\
& =C \; f \left(t,  x_3 \right) \left\| \uh \left( \cdot, x_3 \right) \right\|_{H^\sigma_h}^2 + \frac{\nu}{2}\left\| \nh \uh \left( \cdot, x_3 \right) \right\|_{H^\sigma_h}^2,
\end{aligned}
\end{multline}
Performing an $ H^\sigma_h $ energy estimate onto \eqref{eq:2DstratifiedNS} with the bound \eqref{eq:exp_decay_Hsigma_norms_uh} we deduce that
\begin{multline*}
\frac{1}{2}\frac{\d}{\d t} \left\| \uh \left( \cdot, x_3 \right)  \right\|_{H^\sigma_h}^2 + \frac{\nu}{2} \left\| \nh \uh \left( \cdot, x_3 \right)  \right\|_{H^\sigma_h}^2
+ \nu \left\| \partial_3\uh \left( \cdot, x_3 \right)  \right\|_{H^\sigma_h}^2 -
 \nu \partial_3^2 \left\| \uh \left( \cdot, x_3 \right)  \right\|_{H^\sigma_h}^2\\
 - C \; f \left(t,  x_3 \right) \left\| \uh \left( \cdot, x_3 \right) \right\|_{H^\sigma_h}^2 \leqslant 0.
\end{multline*}
By use of Lemma \ref{lem:poincare} and the fact that $ \nu \left\| \partial_3\uh \left( \cdot, x_3 \right)  \right\|_{H^\sigma_h}^2\geqslant 0 $ we deduce
\begin{equation}\label{eq:auxiliary eq Hsigma uh 1}
\frac{1}{2}\frac{\d}{\d t} \left\| \uh \left( \cdot, x_3 \right)  \right\|_{H^\sigma_h}^2 + \left( \frac{c\nu}{2} - \nu \partial_3^2 - C \; f \left( t, x_3 \right) \right)\left\| \uh \left( \cdot, x_3 \right)  \right\|_{H^\sigma_h}^2\leqslant 0.
\end{equation}

\noindent Let us define
\begin{equation*}
F \left(t, x_3 \right) = C \; \int_0^t f \left( t', x_3 \right) \d t',
\end{equation*}
The function $ F $ is bounded in $ L^\infty_v $ thanks to the results in Lemma \ref{lem:key_lemma}, in particular 
\begin{equation*}
e^{\left\| F \right\|_{L^\infty}}\leqslant C\;
\exp \set{ \frac{K}{c\nu} \left( 1+ \left\| \uh_0 \right\|_{L^\infty_v \left( L^2_h \right)}^2 \right) \left\| \nh\uh_0 \right\|_{L^\infty_v \left( L^2_h \right)}^2 }
\end{equation*}

\noindent hence again as it was done in equation \eqref{eq:stima_Lpv_kernel_part1} we multiply \eqref{eq:auxiliary eq Hsigma uh 1} for
\begin{equation*}
\left\| \uh \left( x_3 \right) \right\|_{H^\sigma_h}^{p-2} = \left( \int_{\T^2_h} \left( 1-\Dh \right)^{\sigma } \uh \left( y_h, x_3 \right)^2\d y_h \right)^{\frac{p-2}{2}},
\end{equation*}
where $ p>2 $, $ \sigma >0 $ and we integrate in $ x_3 $ to deduce

\begin{align*}
 \left\|  \uh \left( t \right) \right\|_{L^p_v \left( L^2_h \right)} \leqslant \; C\;
\exp \set{ \frac{K}{c\nu} \left( 1+ \left\| \uh_0 \right\|_{L^\infty_v \left( L^2_h \right)}^2 \right) \left\| \nh\uh_0 \right\|_{L^\infty_v \left( L^2_h \right)}^2 }
   e^{-\frac{c \nu}{2} \; t}\left\|  \uh \left( t \right) \right\|_{L^p_v \left( L^2_h \right)},
 \end{align*}
in the same fashion as it was done in \eqref{eq:stima_Lpv_kernel_part2} for any $ p\in \left[ 2 , \infty \right] $. The bound \eqref{eq:exponential_decay_LpvHsigma_uh} is then proved.\\

 For the  inequality \eqref{eq:exponential_decay_LpvHsigma_nhuh} the procedure is the same but slightly more involved. We recall that the following bound holds true for zero-horizontal average vector fields:
 \begin{multline}\label{eq:bilinear_Hsigma_oh}
 \left( \left. \uh \left( \cdot, x_3 \right)\cdot \nh \oh \left( \cdot, x_3 \right) \right| \oh \left( \cdot, x_3 \right) \right)_{H^\sigma_h}
 \leqslant \frac{\nu}{2} \left\| \nh \oh \left( \cdot, x_3 \right) \right\|_{H^\sigma_h}^2 \\
+ C\; K^2\;
 \left\| \uh \left( \cdot, x_3 \right)\right\|_{H^\sigma_h}^2
  \left\| \oh \left( \cdot, x_3 \right)\right\|_{L^2_h}^2 
\left\| \oh \left( \cdot, x_3 \right) \right\|_{H^\sigma_h}^2.
 \end{multline}
 We postpone the proof of \eqref{eq:bilinear_Hsigma_oh}.\\
 We set
 \begin{align*}
 & g \left( t, x_3 \right) = 
 \left\| \uh \left(t, \cdot, x_3 \right)\right\|_{H^\sigma_h}^2
  \left\| \oh \left(t, \cdot, x_3 \right)\right\|_{L^2_h}^2, \\
  & G  \left( t, x_3 \right)= C K^2 \int_0^t g\left( t', x_3 \right)\d t' ,
 \end{align*}
 where $ K $ denotes again the norm of $ \oh \mapsto \nh \uh $ as a Calderon-Zygmung application in $ L^2_h $.

 Performing an $ H^\sigma_h $ energy estimate onto the equation satisfied by $ \oh $ with the bound \eqref{eq:bilinear_Hsigma_oh} we deduce the inequality
 \begin{align*}
\frac{1}{2}\frac{\d}{\d t} \left\| \oh \left( x_3 \right) \right\|^2_{H^{\sigma}_h} 
+ \left( \frac{c\nu}{2}- C\; K^2\; g \left( t , x_3 \right) - \nu \partial_3^2 \right) \left\| \oh \left( x_3 \right) \right\|^2_{H^{\sigma}_h} \leqslant 0.
\end{align*}
Net, we multiply the above inequality for $ \left\| \oh \left( x_3 \right) \right\|_{H^\sigma_h}^{p-2}, \ p\geqslant 2 $ in order to deduce as it was done for $ \uh $ that
\begin{align*}
\left\| \oh \left( t \right) \right\|_{L^p_v \left( H^\sigma_h \right)}\leqslant
C e^{-\frac{c\nu}{2} t} \left\| e^G \right\|_{L^\infty} \left\|   \oh_0 \right\|_{L^p_v \left( H^\sigma_h \right)}.
\end{align*}
The function $ e^G \in L^\infty \left( \T^1_v \right) $ thanks to the results in Lemma \ref{lem:key_lemma} and the estimate \eqref{eq:exponential_decay_LpvHsigma_uh}, hence we deduce the bound
$$
e^{\left\| G \right\|_{L^\infty_v}}\leqslant \;
C \exp \set{ \frac{ C K^2
 \left\| \oh_0 \right\|_{L^\infty_v\left( L^2_h \right)}^2 }{c\nu}
\exp \set{ \frac{K}{c\nu} \left( 1+ \left\| \uh_0 \right\|_{L^\infty_v \left( L^2_h \right)}^2 \right) \left\| \nh\uh_0 \right\|_{L^\infty_v \left( L^2_h \right)}^2 }
},
$$
which lead to the final bound
\begin{multline*}
\left\| \oh \left( t \right) \right\|_{L^p_v\left( H^\sigma_h \right)}
\\
\leqslant
C \;\exp \set{\frac{ C K^2
 \left\| \oh_0 \right\|_{L^\infty_v\left( L^2_h \right)}^2 }{c\nu}
\exp \set{ \frac{K}{c\nu} \left( 1+ \left\| \uh_0 \right\|_{L^\infty_v \left( L^2_h \right)}^2 \right) \left\| \nh\uh_0 \right\|_{L^\infty_v \left( L^2_h \right)}^2 }
}\\
\times
 e^{-\frac{c\nu}{2}\; t} \left\|  \oh \right\|_{L^p_v\left( H^\sigma_h \right)},
\end{multline*}
for $ p\in \left[2, \infty\right] $.\\
Since the application $ \oh\mapsto \nh \uh $ is a Calder\`on-Zygmung application we can conclude with the following estimate
\begin{equation*}
\left\| \nh \uh \right\|_{L^p_v \left( H^\sigma_h \right)}
\leqslant
C  {K} \;\Phi \left( U_0 \right)
 e^{-\frac{c\nu}{2}\; t} 
 \left\| \nh \uh_0 \right\|_{L^p_v \left( H^\sigma_h \right)},
\end{equation*}
where $ \Phi \left( U_0 \right) $ is defined in \eqref{eq:def_Phi_U0},  this proves \eqref{eq:exponential_decay_LpvHsigma_nhuh}.

\end{proof}

\textit{Proof of Proposition \ref{prop:Linfty_integrability_uh}} At this point the proof of Proposition \ref{prop:Linfty_integrability_uh} is direct corollary of Lemma \ref{lem:key_lemma2}. Since the vector field $ \uh $ has zero horizontal average the following equivalence of norms hold true
$$
\left\| \nh \uh \left( \cdot, x_3 \right) \right\|_{H^\sigma_h}\sim \left\|  \uh \left( \cdot, x_3 \right) \right\|_{H^{\sigma+1}_h}
$$
It is sufficient in fact to remark now that, for vector fields with zero horizontal average, the embedding $ H^{1+\sigma} \left( \R^2_h \right) \hra L^\infty \left( \R^2_h \right), \sigma >0 $ holds true. I.e.
$$
\left\|  \uh \left( \cdot, x_3 \right) \right\|_{L^\infty_h}
\leqslant C \; \left\|  \uh \left( \cdot, x_3 \right) \right\|_{H^{\sigma+1}_h}  .
$$
These considerations together with the inequality \eqref{eq:exponential_decay_LpvHsigma_nhuh} (setting $ p=\infty $) lead us to the following estimate
\begin{equation*}
\left\|  \uh \right\|_{\Linfty}
\leqslant
C  {K} \; \Phi \left( U_0 \right)
 e^{-\frac{c\nu}{2}\; t} 
 \left\| \nh \uh_0 \right\|_{L^p_v \left( H^\sigma_h \right)}.
\end{equation*}
An integration-in-time completes hence the proof of Proposition \ref{prop:Linfty_integrability_uh}. \hfill $ \Box $\\

\textit{Proof of \eqref{eq:bilinear_Hsigma_oh}}.\label{proof:bilinear_Hsigma_oh} This is the only part of the present paper in which we use the anisotropic (horizontal) paradifferential calculus introduced at Section \ref{anisotropic_paradiff}. We recall that, given two functions $ f,g\in H^\sigma_h $
$$
\left(\left. f \right| g  \right)_{H^\sigma_h}\sim \sum_{q\in\mathbb{Z}} 2^{2q\sigma} \left(\left. \thq f \right| \thq g  \right)_{L^2_h}.
$$
This deduction is a consequence of the almost-orthogonality property of dyadic blocks. Whence it is sufficient to prove bounds for terms of the form
\begin{align*}
A_q= & \left(\left. \thq \left( \uh \left( \cdot, x_3 \right)\cdot \nh \oh \left( \cdot, x_3 \right) \right) \right| \thq\oh \left( \cdot, x_3 \right) \right)_{L^2_h},\\
= & \left(\left. \thq \left( \uh \left( \cdot, x_3 \right) \oh \left( \cdot, x_3 \right) \right) \right| \thq\nh\oh \left( \cdot, x_3 \right) \right)_{L^2_h}.
\end{align*}
Using the (horizontal) Bony decomposition \eqref{Paicu Bony deco} we decompose $ A_q $ into the following infinite sum
\begin{align*}
A_q= & \left(\left. \thq \left( \uh \left( \cdot, x_3 \right) \oh \left( \cdot, x_3 \right) \right) \right| \thq\nh\oh \left( \cdot, x_3 \right) \right)_{L^2_h},\\
= & \sumf \left(\left. \thq \left( S^h_{q'-1} \uh \left( \cdot, x_3 \right) \triangle^h_{q'} \oh \left( \cdot, x_3 \right) \right) \right| \thq \nh \oh \left( \cdot, x_3 \right)  \right)_{L^2_h},\\
& + \sumi \left(\left. \thq \left( \triangle^h_{q'} \uh \left( \cdot, x_3 \right) S^h_{q'+1}  \oh \left( \cdot, x_3 \right) \right) \right| \thq \nh \oh \left( \cdot, x_3 \right)  \right)_{L^2_h},\\
= & A^1_q+A^2_q.
\end{align*}
We start bounding the term $ A^1_q $. We recall that thanks to Bernstein inequality the operator $ \thq $ maps continuously any $ H^\sigma_h $ space to itself.\\
Using H\"older inequality (twice)
\begin{align*}
\left| A^1_q \right|\leqslant & \sumf \left\| \uh\left( \cdot, x_3 \right) \right\|_{L^4_h} \left\| \triangle^h_{q'}\oh\left( \cdot, x_3 \right) \right\|_{L^4_h}  \left\| \thq \nh \oh\left( \cdot, x_3 \right) \right\|_{L^2_h}.
\end{align*}
Thanks to the Remark \ref{remark_equivalence_horizontal_Sobolev_spaces} we know that $ \uh $ and $ \oh $ are vector fields with zero horizontal average for each $ x_3 $. Hence we can use the Gagliardo-Nirenberg-type inequality  \eqref{GN type ineq}, to deduce
\begin{align*}
 \left\| \uh\left( \cdot, x_3 \right) \right\|_{L^4_h} \leqslant & \; C \;  \left\| \uh\left( \cdot, x_3 \right) \right\|_{L^2_h}^{1/2}  \left\| \nh \uh\left( \cdot, x_3 \right) \right\|_{L^2_h}^{1/2},\\
 \left\| \triangle^h_{q'}\oh\left( \cdot, x_3 \right) \right\|_{L^4_h} \leqslant & \; C \; \left\| \triangle^h_{q'}\oh\left( \cdot, x_3 \right) \right\|_{L^2_h}^{1/2} \left\| \triangle^h_{q'} \nh \oh\left( \cdot, x_3 \right) \right\|_{L^2_h}^{1/2} .
\end{align*}
Since the application $ \oh\mapsto \nh \uh $ is a Calderon-Zygmung application we can say that, there exists a $ K $ constant independent of any parameter of the problem such that
$$
\left\| \nh \uh\left( \cdot, x_3 \right) \right\|_{L^2_h}^{1/2}\leqslant K^{1/2} \left\|  \oh\left( \cdot, x_3 \right) \right\|_{L^2_h}^{1/2}.
$$
Moreover for vector fields whose horizontal average is zero the embedding $ H^\sigma_h\hra L^2_h, \sigma >0 $ holds true, hence 
$$
\left\| \uh\left( \cdot, x_3 \right) \right\|_{L^2_h}^{1/2}\leqslant C \left\| \uh\left( \cdot, x_3 \right) \right\|_{H^\sigma_h}^{1/2}.
$$
Since there exists a $ \ell^2 \left( \mathbb{Z} \right) $ sequence such that
\begin{align*}
\left\| \triangle^h_{q'}\oh\left( \cdot, x_3 \right) \right\|_{L^2_h}^{1/2} \left\| \triangle^h_{q'} \nh \oh\left( \cdot, x_3 \right) \right\|_{L^2_h}^{1/2} &
\leqslant C c_{q'}\; 2^{-q'\sigma}
 \left\| \oh\left( \cdot, x_3 \right) \right\|_{H^\sigma_h}^{1/2} \left\|  \nh \oh\left( \cdot, x_3 \right) \right\|_{H^\sigma_h}^{1/2},\\
 \left\| \thq \nh \oh \left( \cdot, x_3 \right)  \right\|_{L^2_h} \leqslant & C c_q 2^{-q\sigma} \left\|   \nh \oh \left( \cdot, x_3 \right)  \right\|_{H^\sigma_h},
\end{align*}
we formally deduced the bound
\begin{align*}
\left| A^1_q \right| \leqslant C K^{1/2} c_q 2^{-2q\sigma} \left\| \uh\left( \cdot, x_3 \right) \right\|_{H^\sigma_h}^{1/2} \left\|  \oh\left( \cdot, x_3 \right) \right\|_{L^2_h}^{1/2}\left\| \oh\left( \cdot, x_3 \right) \right\|_{H^\sigma_h}^{1/2} \left\|  \nh \oh\left( \cdot, x_3 \right) \right\|_{H^\sigma_h}^{3/2}.
\end{align*}

It remains to prove the same kind of bound for the term $ A^2_q $. Again, using H\"older inequality
\begin{align*}
\left| A^2_q \right| \leqslant \sumi \left\| \triangle^h_{q'}\uh \left( \cdot, x_3 \right) \right\|_{L^4_h} \left\| S^h_{q'+2} \oh \left( \cdot, x_3 \right) \right\|_{L^4_h} \left\| \thq \nh \oh \left( \cdot, x_3 \right) \right\|_{L^2_h}.
\end{align*}
Since the vector fields have zero horizontal average we can apply \eqref{GN type ineq}, the fact that $ \oh\mapsto \nh \uh $ is a Calderon-Zygmund operator and \eqref{regularity_dyadic} to deduce
\begin{align*}
\left\| \triangle^h_{q'}\uh \left( \cdot, x_3 \right) \right\|_{L^4_h} \leqslant &\; C \left\| \triangle^h_{q'}\uh \left( \cdot, x_3 \right) \right\|_{L^2_h}^{1/2} \left\| \triangle^h_{q'} \nh \uh \left( \cdot, x_3 \right) \right\|_{L^2_h}^{1/2},\\
\leqslant & \; C K^{1/2 }
\left\| \triangle^h_{q'}\uh \left( \cdot, x_3 \right) \right\|_{L^2_h}^{1/2} \left\| \triangle^h_{q'}\oh \left( \cdot, x_3 \right) \right\|_{L^2_h}^{1/2},\\
\leqslant & \; CK^{1/2 }\; c_q 2^{-q'\sigma} 
\left\| \uh \left( \cdot, x_3 \right) \right\|_{H^\sigma_h}^{1/2} \left\| \oh \left( \cdot, x_3 \right) \right\|_{H^\sigma_h}^{1/2}.
\end{align*}
Using \eqref{GN type ineq} and the embedding $ H^\sigma_h\hra L^2_h $ which hods for vector fields with zero horizontal average
\begin{align*}
 \left\| S^h_{q'+2} \oh \left( \cdot, x_3 \right) \right\|_{L^4_h} \leqslant & \; C  \left\|  \oh \left( \cdot, x_3 \right) \right\|_{L^2_h}^{1/2} \left\| \nh \oh \left( \cdot, x_3 \right) \right\|_{L^2_h}^{1/2},\\
 \leqslant & \; C  \left\|  \oh \left( \cdot, x_3 \right) \right\|_{H^\sigma_h}^{1/2} \left\| \nh \oh \left( \cdot, x_3 \right) \right\|_{L^2_h}^{1/2}.
\end{align*}
Hence with the aid of \eqref{regularity_dyadic}
$$
 \left\| \thq \nh \oh \left( \cdot, x_3 \right) \right\|_{L^2_h} \leqslant C c_q 2^{-q\sigma}  \left\|  \nh \oh \left( \cdot, x_3 \right) \right\|_{H^\sigma_h}.
$$
We hence deduced that
\begin{align*}
\left| A^2_q \right| \leqslant C K^{1/2} c_q 2^{-2q\sigma} \left\| \uh\left( \cdot, x_3 \right) \right\|_{H^\sigma_h}^{1/2} \left\|  \oh\left( \cdot, x_3 \right) \right\|_{L^2_h}^{1/2}\left\| \oh\left( \cdot, x_3 \right) \right\|_{H^\sigma_h}^{1/2} \left\|  \nh \oh\left( \cdot, x_3 \right) \right\|_{H^\sigma_h}^{3/2}.
\end{align*}
With these bounds we hence proved that
\begin{multline*}
\left(\left. \uh\left( \cdot, x_3 \right) \cdot \nh \oh\left( \cdot, x_3 \right) \right|   \oh\left( \cdot, x_3 \right) \right)_{H^\sigma_h}\\
\leqslant C K^{1/2}\left\| \uh\left( \cdot, x_3 \right) \right\|_{H^\sigma_h}^{1/2} \left\|  \oh\left( \cdot, x_3 \right) \right\|_{L^2_h}^{1/2}\left\| \oh\left( \cdot, x_3 \right) \right\|_{H^\sigma_h}^{1/2} \left\|  \nh \oh\left( \cdot, x_3 \right) \right\|_{H^\sigma_h}^{3/2}.
\end{multline*}
To deduce the estimate \eqref{eq:bilinear_Hsigma_oh} it is sufficient hence to apply the convexity inequality $ ab \leqslant \frac{C^4}{4}a^4 + \frac{3}{4 C^{4/3}}b^{4/3} $ to the above estimate. \hfill $ \Box $

\subsubsection{Propagation of isotropic Sobolev regularity.} We apply in this Section the result proved in the previous one in order to conclude the proof of Proposition \ref{propagation_isotropic_Sobolev_regularity_ubarh}.\\

\textit{Proof of Proposition \ref{propagation_isotropic_Sobolev_regularity_ubarh}.} 
Let us apply the operator $ \tq $ to both sides of \eqref{2D_stratified_NS_no_pressure} and let us multiply what we obtain with $ \tq \uh $ and let us take scalar product in $ \2 $, we obtain in particular
\begin{equation}\label{localized_Hs_equation}
\frac{1}{2}\frac{\d}{\d t} \left\| \tq \uh \right\|_\2^2 +  \nu  \left\| \tq \nabla \uh \right\|_\2^2 \leqslant \left| \left( \left. \tq \left( \uh\cdot \nh \uh \right) \right|\tq \uh \right)_\2 \right|,
\end{equation}
whence to obtain the claim everything reduces to bound the term $ \left| \left( \left. \tq \left( \uh\cdot \nh \uh \right) \right|\tq \uh \right)_\2 \right| $. By Bony decomposition \eqref{bony decomposition asymmetric} we know that
\begin{multline*}
\left| \left( \left. \tq \left( \uh\cdot \nh \uh \right) \right|\tq \uh \right)_\2 \right| \leqslant \left| \left( \left. S_{q-1}\uh \tq \nh \uh \right| \tq \uh \right)_\2 \right|\\
+
\sumf \left\lbrace \left| \left( \left. \left[ \tq,\SQ\uh \right] \tQ\nh\uh \right| \tq \uh \right)_\2 \right|+
  \left| \left( \left. \left( \Sq-\SQ \right)\uh \tq\tQ\nh\uh \right| \tq \uh \right)_\2 \right|
\right\rbrace\\
+ \sumi \left| \left(\left. \tq \left( S_{q'+2}\nh \uh\tQ\uh \right)\right| \tq \uh \right)_\2 \right|= \sum_{k=1}^4 I_k.
\end{multline*}
Since $ \dive_h \uh=0 $ we immediately obtain that $ I_1\equiv 0 $, whence if we consider the second term, thanks to H\"older inequality and Lemma \ref{estimates commutator} we can argue that
\begin{align*}
I_2 \leqslant & C \sumf 2^{-q}\left\| \SQ \nabla \uh \right\|_{\Linfty} \left\| \tQ \nh \uh \right\|_\2 \left\| \tq \uh \right\|_\2.
\end{align*}
Accordingly to Bernstein inequality we have that
$$
\left\| \SQ \nabla \uh \right\|_{\Linfty} \leqslant C 2^{q'}\left\| \SQ  \uh \right\|_{\Linfty},
$$
and hence, since $ \left\| \tq f \right\|_\2 \leqslant C c_q\left( t \right) 2^{-qs}\left\| f \right\|_\Hs $ we obtain that
\begin{align*}
I_2 \leqslant C c_q 2^{-2qs}\left\|   \uh \right\|_{\Linfty} \left\| \uh \right\|_\Hs \left\| \nh\uh \right\|_\Hs.
\end{align*}
Similar calculations lead to the same bound for $ I_3 $, i.e.
\begin{align*}
I_3 \leqslant C c_q 2^{-2qs}\left\|   \uh \right\|_{\Linfty} \left\| \uh \right\|_\Hs \left\| \nh\uh \right\|_\Hs.
\end{align*}
Finally form the reminder term $ I_4 $ the following computations hold
\begin{align*}
I_4 = & \sumi \left| \left(\left. \tq \left( S_{q'+2}\nh\uh\tQ\uh \right)\right| \tq \uh \right)_\2 \right|\\
\leqslant & \sumi \left\|  S_{q'+2}\nh\uh \right\|_{\Linfty} \left\| \tQ\uh \right\|_\2 \left\| \tq \uh \right\|_\2,
\end{align*}
but, since we are dealing with localized functions
\begin{align*}
\left\|  S_{q'+2}\nh\uh \right\|_{\Linfty}\leqslant& 2^{q'} \left\|  S_{q'+2}\uh \right\|_{\Linfty},\\
\left\| \tQ\uh \right\|_\2 \leqslant&\; c_{q'} 2^{-q's-q'}\left\| \nabla\uh \right\|_\Hs.\\
\left\| \tq\uh \right\|_\2 \leqslant& \; c_{q}  2^{-qs}\left\|  \uh \right\|_\Hs,
\end{align*}
whence we obtain 
$$
I_4\leqslant C c_q 2^{-2qs}\left\|   \uh \right\|_{\Linfty} \left\| \uh \right\|_\Hs \left\| \nabla \uh \right\|_\Hs,
$$
which in particular implies that
\begin{equation}\label{bound_bilinear_form_est_Hs}
\left| \left( \left. \tq \left( \uh\cdot \nh \uh \right) \right|\tq \uh \right)_\2 \right|
\leqslant
C c_q 2^{-2qs}\left\|   \uh \right\|_{\Linfty} \left\| \uh \right\|_\Hs \left\| \nh\uh \right\|_\Hs.
\end{equation}
Whence inserting \eqref{bound_bilinear_form_est_Hs} into \eqref{localized_Hs_equation}, multiplying both sides for $ 2^{2qs} $ and summing over $ q $ we obtain that
\begin{equation}\label{ultima?}
\frac{1}{2}\frac{\d}{\d t}\left\| \uh \right\|_\Hs^2 + \nu \left\| \uh \right\|_{H^{s+1}\left( \R^3 \right)}^2 \leqslant C \left\| \uh \right\|_{\Linfty} \left\| \uh \right\|_\Hs \left\| \uh \right\|_{H^{s+1}\left( \R^3 \right)},
\end{equation}
whence by Young inequality
$$
\left\| \uh \right\|_{\Linfty} \left\| \uh \right\|_\Hs \left\| \uh \right\|_{H^{s+1}\left( \R^3 \right)} \leqslant \frac{\nu}{2}\left\| \uh \right\|_{H^{s+1}\left( \R^3 \right)}^2 + 
C \left\| \uh \right\|_{\Linfty}^2 \left\| \uh \right\|_\Hs^2,
$$
which, together with \eqref{ultima?} and a Gronwall argument lead to the following estimate
$$
\left\| \uh \left( t \right)\right\|_\Hs^2 + \nu \int_0^t \left\| \uh\left( \tau \right) \right\|_{H^{s+1}\left( \T^3 \right)}^2\d\tau \leqslant C \left\| \uh _0\right\|_\Hs^2 \exp\left\{\int_0^t \left\| \uh\left( \tau \right) \right\|^2_{\Linfty}\d\tau\right\}.
$$
We can hence apply on the above inequality Proposition \ref{prop:Linfty_integrability_uh} to deduce the bound \eqref{eq:stong_Hs_bound_ubar}.
\hfill $ \Box $

\section{Convergence for $ \varepsilon\to 0 $ and  proof of Theorem \ref{main result}}\label{sec:convergence}

\begin{rem}
Given an $ N\in \mathbb{N} $ (generally large) in the present section we denote with $ K_N $ and $ k_N $ two constant such that $ K_N\to \infty $ and $ k_N\to 0 $ respectively as $ N\to \infty $. These constant depend on $ N $ only, and their value may vary from line to line.
\\
In the present proof for the convergence we shall reduce ourselves to the simplified case $ \nu=\nu' $. It is a simple procedure to lift such result when the diffusivity is different. We chose to make such simplification in order not to make an already very complex notation even heavier.  \fine
\end{rem}

The previous section has been devoted to the study of the global-well-posedness of the limit system \eqref{limit system} in some sub-critical $ \Hs, s>1/2 $ Sobolev space. The present section shall use this result to prove that, for $ 0<\varepsilon \leqslant \varepsilon_0 $ sufficiently small the (local, strong) solutions of \eqref{perturbed BSSQ} converge (globally) in the space
$$
\mathcal{C}\left( \mathbb{R}_+; H^s\left( \T^3 \right) \right) \cap L^2\left( \mathbb{R}_+; H^{s+1}\left( \T^3 \right) \right),
$$
to the now global and strong solution $ U $ of \eqref{limit system}.  This shall imply that as long as $ \varepsilon $ is sufficiently close to zero the strong solutions of \eqref{perturbed BSSQ} are in fact global.

 The method we are going to explain reduces to a smart choice of variable substitution that cancels some problematic term appearing in the equations. This technique has been introduced by S. Schochet in \cite{schochet} in the context of hyperbolic systems with singular perturbation. I. Gallagher in \cite{gallagher_schochet} adapted the method to parabolic systems. We mention as well the work of M. Paicu \cite{paicu_rotating_fluids} and E. Grenier \cite{grenierrotatingeriodic}.\\
 Let us subtract \eqref{limit system} from \eqref{filtered system}, and we denote the difference unknown by $ W^\varepsilon=U^\varepsilon- {U} $. After some basic algebra we reduced hence ourselves to the following difference system
 \begin{equation}\label{equation_W_schochet_method}
 \left\lbrace
 \begin{aligned}
 &\partial_t W^\varepsilon+ \mathcal{Q}^\varepsilon\left( W^\varepsilon, W^\varepsilon+2 {U} \right) - \nu \Delta W^\varepsilon = -\left( \mathcal{Q}^\varepsilon\left( {U} , {U}\right) - \mathcal{Q} \left(  {U},  {U} \right) \right),\\
& \dive W^\varepsilon=0,\\
&\left. W^\varepsilon \right|_{t=0}= 0.
 \end{aligned}
 \right.
 \end{equation}
 We define $ \mathcal{R}^\varepsilon_\osc = \mathcal{Q}^\varepsilon\left( {U} , {U}\right) - \mathcal{Q} \left(  {U},  {U} \right)  $. We remark that $ \mathcal{R}^\varepsilon_\osc \to 0 $ only in $ \mathcal{D}' $, since it is defined as $  \mathcal{R}^\varepsilon_\osc= \mathcal{R}^\varepsilon_{\osc, \RN{1}}+\mathcal{R}^\varepsilon_{\osc, \RN{2}} $ where
 \begin{equation}
  \mathcal{R}^\varepsilon_{\osc, \RN{1}}=  \ \mathcal{F}^{-1}\left( \sum_{\substack{\omega^{a,b,c}_{k,n-k,n}\neq 0\\j=1,2,3}} e^{i\frac{t}{\varepsilon}\omega^{a,b,c}_{k,n-k,n}}
 \left( \left.{U}^{a,j}\left( k \right)\left( n_j-k_j \right) {U}^{b}\left( n-k \right)  \right| e^c\left( n \right) \right)e^c\left( n \right)\right), 
 \end{equation}
 \begin{equation}
 \mathcal{R}^\varepsilon_{\osc, \RN{2}} =  \ \mathcal{F}_v^{-1} \left( \sum_{\substack{k+m= \left( 0, n_3 \right) \\ \tilde{\omega}^{a,b}_{k,m}\neq 0}}
e^{i\frac{t}{\varepsilon} \tilde{\omega}^{a,b}_{k,m}} \  n_3 \left( U^{a,3} \left( k \right)  U^{b,h}\left( m \right), 0, 0  \right)^\intercal
  \right),
\end{equation}  
where $  \tilde{\omega}^{a,b}_{k,m} = \omega^a \left( k \right)+\omega^b \left( m \right) $. The term $ \mathcal{R}^\varepsilon_{\osc, \RN{2}} $ represents the high-frequency vertical perturbations induced by the horizontal average $ \left( \int_{\T^2_h} \left( \mathcal{Q}^\varepsilon \left( U^\varepsilon, U^\varepsilon \right) \right)^h \d x_h , 0, 0 \right) $ which converges to zero only weakly as explained in Lemma \ref{lem:limit_horizontal_average}. 
Hence we divide it in high-low frequencies in the following way, for the low-frequency part
\begin{align*}
 \mathcal{R}^\varepsilon_{\osc,\RN{1}, N}=& \ \mathcal{F}^{-1}\left( {1}_{\left\{\left| n \right|\leqslant N\right\}\cap \left\{\left| k \right|\leqslant N\right\}}\mathcal{F \ R}^\varepsilon_{\osc,\RN{1}} \right),
 \\
 \mathcal{R}^\varepsilon_{\osc,\RN{2}, N}= & \ \mathcal{F}^{-1}\left( {1}_{\left\{\left| n_3 \right|\leqslant N\right\}\cap \left\{\left| k \right|\leqslant N\right\}}\mathcal{F \ R}^\varepsilon_{\osc,\RN{2}} \right),\\
 \mathcal{R}^\varepsilon_{\osc, N} = & \ \mathcal{R}^\varepsilon_{\osc,\RN{1}, N} + \mathcal{R}^\varepsilon_{\osc,\RN{2}, N},
\end{align*} 
while the high-frequency part is defined as 
\begin{equation}\label{eq:hi-freq_part_Rosc}
 \mathcal{R}^{\varepsilon,N}_{\osc}= \mathcal{R}^\varepsilon_{\osc}- \mathcal{R}^\varepsilon_{\osc,N}.
\end{equation}

\begin{lemma}\label{lem:hi-freq_part_Rosc}
Let $ \mathcal{R}^{\varepsilon,N}_{\osc} $ be defined as in \eqref{eq:hi-freq_part_Rosc}. $ \mathcal{R}^{\varepsilon,N}_{\osc} \xrightarrow{N\to\infty} 0 $ in $ L^2 \left( \R_+; H^{s-1} \right) $ uniformly in $ \varepsilon $, and the following bound holds
\begin{equation}\label{estimate_high_freq_schochet}
\left\| \mathcal{R}^{\varepsilon,N}_\osc \right\|_{L^2 \left( \R_+; H^{s-1} \right)}\leqslant k_N \xrightarrow{N\to\infty} 0.
\end{equation}
\end{lemma}

The proof of Lemma \ref{lem:hi-freq_part_Rosc} is postponed to the end of the present section, at Subsection \ref{sec:tecnicita}.\\

Let us now perform the following change of unknown
\begin{equation}
\label{definizione psi}
\psi^\varepsilon_N= W^\varepsilon+\varepsilon \tilde{\mathcal{R}}^{\varepsilon}_{\osc,N},
\end{equation}
where, in particular, $ \tilde{\mathcal{R}}^{\varepsilon}_{\osc,N} $ is defined as $ \tilde{\mathcal{R}}^{\varepsilon}_{\osc,N}= \tilde{\mathcal{R}}^{\varepsilon}_{\osc,\RN{1},N} + \tilde{\mathcal{R}}^{\varepsilon}_{\osc,\RN{2},N} $ where
\begin{align*}
 \tilde{\mathcal{R}}^\varepsilon_{\osc,N}= & \ \mathcal{F}^{-1}\left( {1}_{\left\{\left| n \right|\leqslant N\right\}} \sum_{\substack{\omega^{a,b,c}_{k,n-k,n}\neq 0\\j=1,2,3}} {1}_{\left\{\left| k \right|\leqslant N\right\}}\frac{ e^{i\frac{t}{\varepsilon}\omega^{a,b,c}_{k,n-k,n}}}{i\omega^{a,b,c}_{k,n-k,n}}
 \left( \left.{U}^{a,j}\left( k \right)\left(n_j-k_j\right) {U}^{b}\left( n-k \right)\right| e^c\left( n \right) \right)e^c\left( n \right)\right),\\
 \tilde{\mathcal{R}}^{\varepsilon}_{\osc,\RN{2},N}=& \  \mathcal{F}_v^{-1} \left({1}_{\left\{\left| n_3 \right|\leqslant N\right\}} \sum_{\substack{k+m= \left( 0, n_3 \right) \\ \tilde{\omega}^{a,b}_{k,m}\neq 0}}
{1}_{\left\{\left| k \right|\leqslant N\right\}}
\frac{e^{i\frac{t}{\varepsilon} \tilde{\omega}^{a,b}_{k,m}}}{i\ \tilde{\omega}^{a,b}_{k,m}} \  n_3 \left( U^{a,3} \left( t,  k \right)  U^{b,h}\left( t, m \right), 0, 0  \right)^\intercal
  \right)
\end{align*}
in particular we remark that 
\begin{equation}\label{eq:cancellation_schochet}
\partial_t\left( \varepsilon  \tilde{\mathcal{R}}^\varepsilon_{\osc,N} \right)= {\mathcal{R}}^\varepsilon_{\osc,N} +\varepsilon \tilde{\mathcal{R}}^{\varepsilon,t}_{\osc,N},
\end{equation}
 with $ \tilde{\mathcal{R}}^{\varepsilon, t}_{\osc,N} $ is defined as $ \tilde{\mathcal{R}}^{\varepsilon, t}_{\osc,N}= \tilde{\mathcal{R}}^{\varepsilon, t}_{\osc,\RN{1},N} + \tilde{\mathcal{R}}^{\varepsilon, t}_{\osc,\RN{2},N} $ where
\begin{align*}
\tilde{\mathcal{R}}^{\varepsilon,t}_{\osc,\RN{1},N}=& \ \mathcal{F}^{-1}\left( {1}_{\left\{\left| n \right|\leqslant N\right\}} \sum_{\substack{\omega^{a,b,c}_{k,n-k,n}\neq 0\\j=1,2,3}} {1}_{\left\{\left| k \right|\leqslant N\right\}}\frac{ e^{i\frac{t}{\varepsilon}\omega^{a,b,c}_{k,n-k,n}}}{i\omega^{a,b,c}_{k,n-k,n}}
 \partial_t\left( \left.{U}^{a,j}\left(t, k \right)\left( n_j-k_j \right) {U}^{b}\left(t, n-k \right)\right| e^c\left( n \right) \right)e^c\left( n \right)\right),
 \end{align*}
 \begin{align*}
\tilde{\mathcal{R}}^{\varepsilon,t}_{\osc,\RN{2},N}=& \  \mathcal{F}_v^{-1} \left({1}_{\left\{\left| n_3 \right|\leqslant N\right\}} \sum_{\substack{k+m= \left( 0, n_3 \right) \\ \tilde{\omega}^{a,b}_{k,m}\neq 0}}
{1}_{\left\{\left| k \right|\leqslant N\right\}}
\frac{e^{i\frac{t}{\varepsilon} \tilde{\omega}^{a,b}_{k,m}}}{i\ \tilde{\omega}^{a,b}_{k,m}} \ \partial_t \  n_3 \left( U^{a,3} \left( t,  k \right)  U^{b,h}\left( t, m \right), 0, 0  \right)^\intercal
  \right)
.
\end{align*}

We underline the fact that the term ${\mathcal{R}}^\varepsilon_{\osc,N}$ in \eqref{eq:cancellation_schochet} is what we require in order to cancel the low frequencies of ${\mathcal{R}}^\varepsilon_{\osc}$ which otherwise converge to zero only weakly due to stationary phase theorem. This here is the key observation and most important idea on which Schochet method is  based: despite the fact that the difference system presents nonlinearities which does not converge strongly to zero we can define an alternative unknown $ \pN $ which is an $ \mathcal{O} \left( \varepsilon \right) $-corrector of the difference $ W^\varepsilon $ which solves an equation in which this problematic nonlinear interaction vanishes.\\

Tanks to definition \eqref{definizione psi} and system \eqref{equation_W_schochet_method} we can deduce the equation satisfied by $ \psi^\varepsilon_N $ after some elementary algebraic manipulation, which is 
\begin{equation}
\label{equation_psi_schochet}
\left\lbrace
\begin{aligned}
&\partial_t \pN + \mathcal{Q}^\varepsilon \left( \pN, \pN -2\varepsilon \tR + 2{U} \right)-\nu\Delta \pN= -\mathcal{R}^{\varepsilon,N}_{\osc} -
\varepsilon \Gamma^\varepsilon_N,\\
&\dive \pN=0,\\
&\left. \pN \right|_{t=0}=\psi^\varepsilon_{N,0}= \left. \varepsilon \tR \right|_{t=0},
\end{aligned}
\right.
\end{equation}
with $ \Gamma^\varepsilon_N $ defined as 
\begin{align*}
\Gamma^\varepsilon_N= & \nu \Delta \tR +\mathcal{Q}^\varepsilon \left( \tR, \varepsilon \tR +2 {U} \right)+\tilde{\mathcal{R}}^{\varepsilon,t}_{\osc,N}.
\end{align*}
We outline that $ \pN $ is divergence-free since it is a linear combination of the eigenvectors $ e^0,e^\pm $ defined in \eqref{eigenvectors} which are all divergence-free.

Now we claim that
\begin{lemma}\label{boundedness_Gamma}
 $ \Gamma^\varepsilon_N $ is bounded in $ L^2\left( \R_+; H^{s-1} \right) $ by a constant $ K_N $ which depends on $ N $ solely.
\end{lemma}
This is usually referred as \textit{small divisor estimate} in the literature. The proof is due to the fact that all the elements composing $ \Gamma^\varepsilon_N=\Gamma^\varepsilon_N\left( {U} \right) $ are localized in the frequency space, hence they have all the regularity we want them to have at the cost of some power of $ N $. We omit a detailed proof only for the sake of brevity, but this can be deduced thanks to the energy estimates performed on $ {U} $ in the previous section.\\

Let us, at this point, perform an $ \Hs $ energy estimate on equation \eqref{equation_psi_schochet}, we obtain that
\begin{multline}\label{Gronwall_schochet_first_stage}
\frac{1}{2}\frac{\d}{\d t} \left\| \pN \right\|_\Hs^2+ \nu \left\| \pN \right\|_{H^{s+1}\left( \R^3 \right)}^2 \\
\leqslant 
\left| \left( \left. \mathcal{R}^{\varepsilon,N}_{\osc} + \varepsilon\Gamma^\varepsilon_N + \mathcal{Q}^\varepsilon \left( \pN, \pN -2\varepsilon \tR + 2{U} \right) \right|\pN \right)_\Hs \right| .
\end{multline} 
Now, if we consider two four component vector fields $ A,B $ such that their first three components are divergence-free it is indeed true that $ \left\| \mathcal{Q}^\varepsilon \left( A,B \right) \right\|_\Hs \leqslant C \left\| A\otimes B \right\|_{H^{s+1}\left( \R^3 \right)} $, we shall use repeatedly this property in what follows. We shall use as well the fact that $ H^{s+1}\left( \R^3 \right), \ s>1/2 $ is a Banach algebra. Whence
\begin{equation}
\label{estimates_schochet}
\begin{aligned}
&\left( \left. \mathcal{Q}^\varepsilon \left( \pN, \pN\right) \right|\pN \right)_\Hs 
&\;\leqslant &\; C \left\| \pN \otimes \pN \right\|_{H^{s+1}\left( \R^3 \right)} \left\| \pN \right\|_\Hs,\\
& &\;\leqslant & \; C \left\| \pN \right\|_\Hs\left\| \pN  \right\|_{H^{s+1}\left( \R^3 \right)}^2,\\
&\left( \left. \mathcal{Q}^\varepsilon \left( \pN, 2U \right) \right|\pN \right)_\Hs &\;\leqslant &\; C \left\| \pN \right\|_\Hs \left\| \pN  \right\|_{H^{s+1}\left( \R^3 \right)}\left\| U  \right\|_{H^{s+1}\left( \R^3 \right)},\\
&\left( \left. \mathcal{Q}^\varepsilon \left( \pN, 2\varepsilon \tR \right) \right|\pN \right)_\Hs &\; \leqslant &\; C \varepsilon \left\| \pN \right\|_\Hs \left\| \pN  \right\|_{H^{s+1}\left( \R^3 \right)}\left\| \tR  \right\|_{H^{s+1}\left( \R^3 \right)}.
\end{aligned}
\end{equation}

Using the estimates in \eqref{estimates_schochet} into \eqref{Gronwall_schochet_first_stage} and using repeatedly Young inequality $ ab \leqslant \frac{\eta}{2} a^2 +\frac{1}{2\eta}b^2 $ we obtain
\begin{multline}\label{Gronwall_schochet_second_stage}
\frac{1}{2}\frac{\d}{\d t} \left\| \pN \right\|_\Hs^2+ \left( \frac{\nu}{2} - C \left\| \pN \right\|_\Hs \right) \left\| \pN \right\|_{H^{s+1}\left( \R^3 \right)}^2\\
\leqslant
C \left( \left\| U  \right\|_{H^{s+1}\left( \R^3 \right)}^2 + \varepsilon \left\| \tR  \right\|_{H^{s+1}\left( \R^3 \right)}^2 \right)\left\| \pN \right\|_\Hs^2 + C \left\| \mathcal{R}^{\varepsilon,N}_{\osc} + \varepsilon\Gamma^\varepsilon_N \right\|_{H^{s-1}}^2.
\end{multline}
Whence let us define 
\begin{equation}
\label{definizione_Theta}
\frac{1}{2} \Theta_{\varepsilon,N} \left( t \right)=C \left( \left\| U  \right\|_{H^{s+1}\left( \R^3 \right)}^2 + \varepsilon \left\| \tR  \right\|_{H^{s+1}\left( \R^3 \right)}^2 \right),
\end{equation}
by variation of constant method  we transform \eqref{Gronwall_schochet_second_stage} into
\begin{multline}\label{Gronwall_schochet_third_stage}
\frac{1}{2}\frac{\d}{\d t } \left( \left\| \pN \right\|_\Hs^2 e^{-\int_0^t \Theta_{\varepsilon, N} \left( s \right)\d s} \right) + 
\left( \frac{\nu}{2} - C \left\| \pN \right\|_\Hs \right) \left\| \pN \right\|_{H^{s+1}\left( \R^3 \right)}^2e^{-\int_0^t \Theta_{\varepsilon, N} \left( s \right)\d s}
\\
\leqslant C \left\| \mathcal{R}^{\varepsilon,N}_{\osc} + \varepsilon\Gamma^\varepsilon_N \right\|_{H^{s-1}}^2 e^{-\int_0^t \Theta_{\varepsilon, N} \left( s \right)\d s}.
\end{multline}

Now we claim the following
\begin{lemma}\label{boundedness_Theta}
The function $ \Theta_{\varepsilon,N} $ defined in \eqref{definizione_Theta} is an $ L^1\left( \R_+ \right) $ function uniformly in $ \varepsilon $, moreover we can write the $ L^1\left( \R_+ \right) $-bound as
\begin{equation}\label{eq:bound_Theta}
\left\| \Theta_{\varepsilon,N} \right\|_{L^1\left( \R_+ \right)}\leqslant C + \varepsilon \ K_N .
\end{equation}
\end{lemma}

We do not give a detailed proof of Lemma \ref{boundedness_Theta}. What it has to be retained is that it is possible to bound the term $ \tR $ since it is localized on the low frequencies, at the cost of making appear a (large in $ N $) constant $ K_N $ depending on $ N $ only. \\

Lemma \ref{boundedness_Theta} in particular asserts, that fixing an (eventually large) $ N>0 $ there exists an $ \varepsilon=\varepsilon_{N}>0 $ such that  there exist two constants $ 0<c_1 \left( \varepsilon, N \right)\leqslant c_2\left( \varepsilon, N \right) $ such that
\begin{equation*}
c_1 \left( \varepsilon, N \right) \leqslant \left| e^{-\int_0^t \Theta_{\varepsilon,N}\left( s \right)\d s} \right| \leqslant c_2 \left( \varepsilon, N \right),
\end{equation*}
 independently of $ t\in\R_+ $. \\
 
 We \textit{fix} now an $ \eta >0 $ (which we can suppose to be small) and we \textit{select} two quantities $ N= N_{\eta} $ and $ \varepsilon=\varepsilon_{N}=\varepsilon_{N_{\eta}} $ such that
 \begin{equation}\label{eq:hyp_bootstrap}
 \begin{aligned}
 \left\| \psi^\varepsilon_{N,0} \right\|_{\Hs}\leqslant & \ \frac{\nu}{8C},&
 e^{C + \varepsilon K_N} \ \left\| \psi^\varepsilon_{N,0} \right\|_{\Hs}\leqslant & \ \frac{\eta}{2},&
 C \ c_2 \left( \varepsilon, N \right) \ \left( k_N + \varepsilon K_N \right) \leqslant & \ \frac{\eta}{2}.
  \end{aligned}
 \end{equation}
{The first and second inequality in \eqref{eq:hyp_bootstrap} holds true thanks to the following procedure: we consider the definition of $ \pN $ given in equation \eqref{definizione psi} we immediately deduce that $ \psi^\varepsilon_{N,0}= \varepsilon \left. \tR \right|_{t=0} $, but in particular $ \left\| \left. \tR \right|_{t=0} \right\|_\Hs \leqslant K_N $ thanks to an argument similar to the one which proves Lemma \ref{boundedness_Gamma}, i.e. we exploit the fact that $ \tR $ is supported in a ball of radius $ N $ in the frequency space and hence we can gain all the integrability we want at the price of some power of $ N $.}
 The constants $ C $ and $ K_N $ in particular are considered to be the ones appearing in \eqref{eq:bound_Theta}.\\

  We integrate now \eqref{Gronwall_schochet_third_stage} in time, using the above consideration combined with Lemma \ref{boundedness_Gamma} and \eqref{estimate_high_freq_schochet} we transform \eqref{Gronwall_schochet_third_stage} into
\begin{multline}\label{Gronwall_schochet_fourth_stage}
  \left\| \pN \left( t \right) \right\|_\Hs^2+
   \int_0^t \left( {\nu} - 2C \left\| \pN \left( s \right) \right\|_\Hs \right) \left\| \pN \left( s \right) \right\|_{H^{s+1}\left( \R^3 \right)}^2 e^{\int_s^t \Theta_{\varepsilon,N}\left( s' \right)\d s'} \d s\\
   \leqslant C c_2 \left( k_N+\varepsilon K_N \right) + \left\| \psi^\varepsilon_{N,0} \right\|_\Hs e^{\int_0^t \Theta_{\varepsilon,N}\left( s \right)\d s},
\end{multline}
where we used the following notation $ \psi^\varepsilon_{N,0}= \left. \pN \right|_{t=0} $. 
 Whence considering the hypothesis \eqref{eq:hyp_bootstrap} that we set for the bootstrap procedure we deduce
\begin{equation}
\label{smallness_RHS_schochet}
\begin{aligned}
\left\| \psi^\varepsilon_{N,0} \right\|_\Hs e^{\int_0^t \Theta_{\varepsilon,N}\left( s \right)\d s}\leqslant & \frac{\eta}{2},& \hspace{1cm}
 C c_2 \left( k_N+\varepsilon K_N \right)\leqslant & \frac{\eta}{2}.
\end{aligned}
\end{equation}

Thanks to the existence theorem given in Theorem \ref{thm local existence strong solution} we can assert that the application $ t\mapsto \left\| \pN \left( t \right) \right\|_\Hs $ is continuous, hence, since we  considered $ \left\| \psi^\varepsilon_{N,0} \right\|_{\Hs} $ small in \eqref{eq:hyp_bootstrap} it makes sense to define the time
\begin{equation}
\label{def_T_tilde_star_epsilon}
\tilde{T}_\varepsilon^\star =\sup\left\lbrace
0\leqslant t \leqslant T^\star \left| \left\| \pN\left( t \right) \right\|_\Hs \leqslant \frac{\nu}{4C}\right. \right\rbrace.
\end{equation}
The definition of $ \tilde{T}^\star_\varepsilon $ implies that $ {\nu} - 2C \left\| \pN \left( s \right) \right\|_\Hs \leqslant \nu/2 $  in $ \left[ 0 , \tilde{T}^\star_\varepsilon \right] $, and moreover, since \linebreak $ \left| e^{\int_s^t \Theta_{\varepsilon,N}\left( s' \right)\d s'} \right|\geqslant 1 $ and estimates \eqref{smallness_RHS_schochet} transform \eqref{Gronwall_schochet_fourth_stage} in the following differential inequality
\begin{equation}\label{Gronwall_schochet_fifth_stage}
\left\| \pN \left( t \right) \right\|_\Hs^2 + \frac{\nu}{2}  \int_0^t \left\| \pN \left( s \right) \right\|_{H^{s+1}\left( \R^3 \right)}^2\d s \leqslant \eta.
\end{equation}
Now the bound on the right hand side of \eqref{Gronwall_schochet_fifth_stage} is independent of $ t $ and arbitrary small, whence selecting $ \eta < \frac{\nu^2}{16 C^2} $ the condition \eqref{def_T_tilde_star_epsilon} defining $ \tilde{T}^\star_\varepsilon $ is always satisfied, whence we can assert that $ \tilde{T}^\star_\varepsilon=\infty $ (bootstrap) and hence we obtained the following result 
\begin{prop}\label{prop:smallness_psi_schochet}
Let be $ \eta >0 $, there exists an $ \varepsilon_\eta >0 $ and $ N_\eta \in \mathbb{N}^\star $ such that for each $ \varepsilon \in \left( 0, \varepsilon_{\eta} \right), \ N > N_{\eta} $ the function $ \pN $ defined as in \eqref{definizione psi} solves globally \eqref{equation_psi_schochet} and for each $ t>0 $ the following bound holds true
\begin{equation*}
\left\| \pN \left( t \right) \right\|^2_{\Hs} + \frac{\nu}{2} \int_0^t \left\| \nabla \pN \left( s \right) \right\|_{\Hs}^2 \d s \leqslant \eta.
\end{equation*}
\end{prop}

To prove the end of Theorem \ref{main result} is now a corollary pf Proposition \ref{prop:smallness_psi_schochet}. Let us set
\begin{equation*}
\mathcal{E}^s = L^\infty\left( \R_+; \Hs \right)\cap L^2\left( \R_+; H^{s+1}\left( \T^3 \right) \right).
\end{equation*}
Thanks to the same procedure as always ($ \tR $ is localized in the frequency set) we can safely assert that
\begin{equation}\label{aiuto1}
\left\| W^\varepsilon \right\|_{\mathcal{E}^s}\leqslant \left\| \pN \right\|_{\mathcal{E}^s} + \varepsilon K_N <\infty,
\end{equation}
which accidentally implied that $ W^\varepsilon $ belongs to $ \mathcal{E}^s $. Let us remind that $ W^\varepsilon= U^\varepsilon -U $, and that $ U $ belongs to $ \mathcal{E}^s $ thanks to the results proved in Proposition \ref{propagation_isotropic_Sobolev_regularity_ubarh} and \ref{regularity_osc_part}, hence $ U^\varepsilon \in \mathcal{E}^s $ if $ \varepsilon $ is sufficiently small.\\
From \eqref{aiuto1} we deduce that
\begin{equation*}
\limsup_{\varepsilon\to 0} \left\| W^\varepsilon \right\|_{\mathcal{E}^s}\leqslant 2 \eta,
\end{equation*}
for any $ \eta>0 $, whence we finally deduced that
$U^\varepsilon \xrightarrow{\varepsilon\to 0} U $ in $ \mathcal{E}^s.
$ \hfill $ \Box $

 \subsection{Proofs of technical lemmas.}\label{sec:tecnicita}
 \ \\

 \textit{Proof of Lemma \ref{lem:hi-freq_part_Rosc}} \ : 
 The proof of Lemma \ref{lem:hi-freq_part_Rosc} consists in an application of Lebesgue dominated convergence theorem. 
Since every time that Schochet method is applied (notably we refer to \cite{gallagher_schochet}) an estimate of this form on the high frequencies  has to be performed we shall outline the proof of Lemma \ref{lem:hi-freq_part_Rosc}. \\
The element $  \mathcal{R}^{\varepsilon,N}_{\osc}  $  converges point-wise (in the frequency space) to zero when $ N\to \infty $ (computations omitted), and it is indeed true that
\begin{equation*}
\left| \left| n \right|^{s-1} \left| \mathcal{F \ R}^{\varepsilon, N}_{\osc} \left( t, n \right) \right| \right|^2 \leqslant \Big| \left| n \right|^s \left| \mathcal{F} \left( U\otimes U \right) \left( t, n \right) \right| \Big|^2 = \mathcal{G}_s \left( t, n \right).
\end{equation*}
 By Plancherel theorem the $ L^1 \left( \R_+\times \mathbb{Z}^3, \d t \times \d \# \right) $ norm of $ \mathcal{G}_s $ is indeed the square of the $ L^2 \left( \R_+; H^s \right) $ norm of $ U\otimes U $ (here we denote with $ \# $ the discrete homogeneous measure on $ \mathbb{Z}^3 $). The function $ \mathcal{G}_s $ will be the dominating function.
\noindent
 We  apply the following product rule (for a proof of which we refer to \cite[Corollary 2.86, p. 104]{bahouri_chemin_danchin_book})
$$
\left\| {U}\otimes {U} \right\|_\Hs \lesssim \left\| {U} \right\|_{\Linfty} \left\| {U} \right\|_\Hs,
$$ 
while thanks to the embedding $ H^{s+1}\left( \R^3 \right) \hra \Linfty $ for $ s>1/2 $ we can finally state that 
$$
\left\| U\otimes U \right\|_{L^2 \left(\R_+; H^s \right)}
\lesssim \left\| {U} \right\|_{L^2\left( \R_+; H^{s+1}\left( \R^3 \right)\right)} \left\| {U} \right\|_{L^\infty\left( \R_+;\Hs\right)}<\infty.
$$
Since $\mathcal{R}^{\varepsilon,N}_\osc$ converges point-wise to zero in the Fourier space as $N\to \infty$ we can hence deduce \eqref{estimate_high_freq_schochet}.
\hfill $ \Box $

\footnotesize{
\providecommand{\bysame}{\leavevmode\hbox to3em{\hrulefill}\thinspace}
\providecommand{\MR}{\relax\ifhmode\unskip\space\fi MR }
\providecommand{\MRhref}[2]{%
  \href{http://www.ams.org/mathscinet-getitem?mr=#1}{#2}
}
\providecommand{\href}[2]{#2}
}
\end{document}